\documentclass[11pt]{amsart}

\usepackage[latin1]{inputenc} 
\usepackage{amsthm,amsmath,amssymb,bm,mathtools} 
\usepackage{graphicx}
\usepackage{color}
\usepackage{verbatim}
\usepackage{float}
\usepackage{nccmath}
\usepackage[noadjust]{cite}
\usepackage[foot]{amsaddr}
\usepackage{enumitem}
\usepackage{thmtools}
\usepackage{thm-restate}

\setlist[enumerate,1]{label=\upshape(\roman*),ref=(\roman*)}

\newcommand{\COMMENT}[1]{}
\renewcommand{\COMMENT}[1]{\footnote{\textcolor{blue!70!black}{#1}}}

\usepackage{hyperref}       
\usepackage[capitalise]{cleveref} 
\hypersetup{colorlinks=true,
   citecolor=blue,
   filecolor=blue,
   linkcolor=blue,
   urlcolor=blue
}

\creflabelformat{equation}{(#2#1#3)}
\crefdefaultlabelformat{#2#1#3}

\Crefname{enumi}{}{}
\Crefname{thm}{Theorem}{Theorems}
\Crefname{lm}{Lemma}{Lemmas}
\Crefname{cor}{Corollary}{Corollaries}
\Crefname{prop}{Proposition}{Propositions}
\Crefname{claim}{Claim}{Claims}
\Crefname{equation}{}{}
\Crefname{conjecture}{Conjecture}{Conjectures}
\Crefname{figure}{Figure}{Figures}
\Crefname{fact}{Fact}{Facts}

\usepackage{array}  
\newcolumntype{C}{>{\centering \arraybackslash}m{3cm}}
\newcolumntype{D}{>{\centering \arraybackslash}m{1.25cm}}

\usepackage{enumitem} 
\setlist[enumerate]{topsep=0pt,itemsep=-1ex,partopsep=1ex,parsep=1ex}
\setlist[itemize]{topsep=0pt,itemsep=-1ex,partopsep=1ex,parsep=1ex}

\usepackage{bbm}

\theoremstyle{plain}
\newtheorem{theo}{Theorem}[section]
\newtheorem{prop}[theo]{Proposition}
\newtheorem{lemma}[theo]{Lemma}
\newtheorem{cor}[theo]{Corollary}

\newtheorem{fact}[theo]{Fact}

\theoremstyle{definition}
\newtheorem{defn}[theo]{Definition}

\newtheorem{prob}[theo]{Problem}

\newcommand{\mc}[1]{\mathcal{#1}}
\newcommand{\mb}[1]{\mathbb{#1}}

\newcommand{\sm}{\setminus}
\newcommand{\ov}{\overline}

\newcommand{\eps}{\varepsilon}
\renewcommand{\epsilon}{\varepsilon}

\newcommand{\ova}{\overrightarrow}

\newcommand{\lL}{\lambda}

\newcommand{\ind}{I}
\newcommand{\fs}{\#}
\newcommand{\fmax}{{\rm max}}
\newcommand{\ff}{\#}
\newcommand{\ffmax}{{\rm max}}
\newcommand{\bic}{{\rm bic}}
\newcommand{\mon}{{\rm mon}}
\newcommand{\Aut}{{\rm Aut}}

\newcommand*\BB{\raisebox{-1pt}{\includegraphics[scale=0.35]{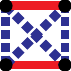}}}
\newcommand*\BR{\raisebox{-1pt}{\includegraphics[scale=0.35]{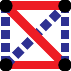}}}
\newcommand*\RR{\raisebox{-1pt}{\includegraphics[scale=0.35]{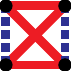}}}

\newcommand*\BBn{\raisebox{-1pt}{\includegraphics[scale=0.35]{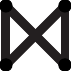}}}
\newcommand*\BRn{\raisebox{-1pt}{\includegraphics[scale=0.35]{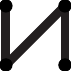}}}
\newcommand*\RRn{\raisebox{-1pt}{\includegraphics[scale=0.35]{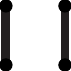}}}
\newcommand*\Kplus{\raisebox{-1pt}{\includegraphics[scale=0.35]{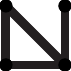}}}
\newcommand*\Kminus{\raisebox{-1pt}{\includegraphics[scale=0.35]{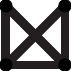}}}
\newcommand*\Kfour{\raisebox{-1pt}{\includegraphics[scale=0.35]{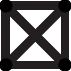}}}

\newcommand*\CC{\raisebox{-1pt}{\includegraphics[scale=0.35]{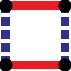}}}

\newcommand*\PP{\raisebox{-1pt}{\includegraphics[scale=0.35]{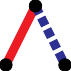}}}
\newcommand*\CCext{\raisebox{-1pt}{\includegraphics[scale=0.35]{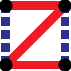}}}
\newcommand*\CCextC{\raisebox{-1pt}{\includegraphics[scale=0.35]{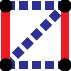}}}
\newcommand*\CCextt{\raisebox{-1pt}{\includegraphics[scale=0.35]{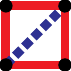}}}
\newcommand*\CCexttC{\raisebox{-1pt}{\includegraphics[scale=0.35]{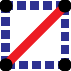}}}

\newcommand*\RRR{\raisebox{-1pt}{\includegraphics[scale=0.35]{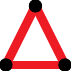}}}
\newcommand*\BBB{\raisebox{-1pt}{\includegraphics[scale=0.35]{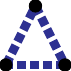}}}

\newcommand*\RRBB{\raisebox{-1pt}{\includegraphics[scale=0.35]{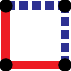}}}
\newcommand*\RRRB{\raisebox{-1pt}{\includegraphics[scale=0.35]{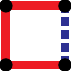}}}
\newcommand*\BBBR{\raisebox{-1pt}{\includegraphics[scale=0.35]{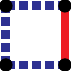}}}
\newcommand*\RRBBexta{\raisebox{-1pt}{\includegraphics[scale=0.35]{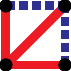}}}
\newcommand*\RRBBextaC{\raisebox{-1pt}{\includegraphics[scale=0.35]{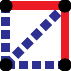}}}
\newcommand*\RRBBextb{\raisebox{-1pt}{\includegraphics[scale=0.35]{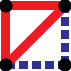}}}
\newcommand*\RRBBextbC{\raisebox{-1pt}{\includegraphics[scale=0.35]{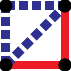}}}

\usepackage[dvipsnames]{xcolor}

\usepackage{stackengine}
\stackMath
\newcommand{\stacklEq}[1]{%
  \setbox0=\hbox{${}\mathrel{\stackon[-1pt]{\leq}{\scriptstyle\text{#1\strut}}}{}$}
  \xdef\tmpwd{\dimexpr\the\wd0\relax}
  \kern.5\tmpwd\mathclap{\box0}&\kern.5\tmpwd
}

\newcommand{\stackEq}[1]{
  \setbox0=\hbox{${}\mathrel{\stackon[-1pt]{=}{\scriptstyle\text{#1\strut}}}{}$}
  \xdef\tmpwd{\dimexpr\the\wd0\relax}
  \kern.5\tmpwd\mathclap{\box0}&\kern.5\tmpwd
}

\title{The semi-inducibility problem}
\author[A.~Basit]{Abdul Basit}
\email{abdul.basit@monash.edu}
\address[A.~Basit, D.~Horsley]{School of Mathematics, Monash University, Melbourne VIC 3800, Australia.}
\author[B.~Granet]{Bertille Granet}
\email{granet@informatik.uni-heidelberg.de}
\address[B.~Granet]{Institut f\"ur Informatik, Universit\"at Heidelberg, 69120 Heidelberg, Germany.}
\author[D.~Horsley]{Daniel Horsley}
\email{danhorsley@gmail.com}
\author[A.~K\"undgen]{Andr\'e K\"undgen}
\email{akundgen@csusm.edu}
\address[A.~K\"undgen]{Department of Mathematics, California State University San Marcos, San Marcos, CA
92096, United States.}
\author[K.~Staden]{Katherine Staden}
\email{katherine.staden@open.ac.uk}
\address[K.~Staden]{School of Mathematics and Statistics, The Open University, Walton Hall, Milton Keynes MK7
6AA, United Kingdom.}
\thanks{Abdul Basit and Daniel Horsley were supported by Australian Research Council grant DP220102212. Bertille Granet was supported by the Alexander von Humboldt Foundation. Katherine Staden was supported by EPSRC Fellowship EP/V025953/1.}

\begin{document}

\begin{abstract}
Let $H$ be a $k$-edge-coloured graph and let $n$ be a positive integer. What is the maximum number of copies of $H$ in a $k$-edge-coloured complete graph on $n$ vertices? This paper studies the case $k=2$, which we call the semi-inducibility problem. This problem is a generalisation of the inducibility problem of Pippenger and Golumbic which is solved only for some small graphs and limited families of graphs. We prove sharp or almost sharp results for alternating walks, for alternating cycles of length divisible by 4, and for 4-cycles of every colour pattern.

Liu, Mubayi and Reiher asked whether there is a graph $F$ for which the binomial random graph is an asymptotically extremal graph in the inducibility problem over all graphs of a given edge density. This was recently answered in a strong negative sense by Jain, Michelen and Wei.
In contrast, we find a \emph{quantum} graph $Q$ with positive coefficients and an interval of edge densities for which the only extremal graphs are quasirandom.
\end{abstract}

\maketitle

\section{Introduction}

Counting subgraphs inside a host graph is arguably one of the oldest and most studied problems in extremal combinatorics. In 1941, Tur\'an \cite{Turan41} established the maximum number of edges in a $K_r$-free $n$-vertex graph and characterised the extremal examples. This initiated a long and fruitful line of research, which investigates multiple variants of this problem. The most natural generalisation consists in varying the forbidden structure: for example, the famous Erd\H{o}s-Stone theorem \cite{erdos1946structure} gives an asymptotic bound on the maximum number of edges in an $H$-free $n$-vertex graph for any fixed $H$. Rather than maximise the number of edges, one can also try to maximise other types of substructures, such as larger cliques in the case of Zykov's theorem \cite{Zykov52}. 
Another avenue has been to consider coloured variants, such as the rainbow Tur\'an problem introduced by Keevash, Mubayi, Sudakov, and  Verstra\"ete \cite{keevash2004multicolour}, which asks for the maximum number of edges in a properly edge-coloured $n$-vertex graph with no rainbow copy of a fixed graph $H$.

In this paper, we consider a new problem of this type: given a graph $H$ and a fixed (not necessarily proper) $k$-edge-colouring of $H$, what is the maximum number of copies of $H$ in a $k$-edge-coloured $K_n$?

\begin{prob}\label{prob}
Let $H$ be a $k$-edge-coloured graph
and let $n$ be a positive integer.
Given a $k$-edge-coloured complete graph $G$ on $n$ vertices,
let $\fs(H,G)$ be the number of copies of $H$ in $G$.
Determine
$$
\fmax(H,n) \coloneqq \max_{|V(G)|=n}\fs(H,G)
\quad\text{and/or}\quad
\fmax(H) \coloneqq \lim_{n \to \infty}\frac{\fmax(H,n)}{n^{|V(H)|}}.
$$
\end{prob}
It is not hard to see that this limit exists (see Proposition~\ref{pr:limit}).
Of course, $\#(H,G)$ can also be defined for graphs $G$ which are not complete, but adding edges (in any colour) to $G$ will not decrease $\#(H,G)$, and hence for the maximisation problem it suffices to consider complete coloured graphs.

Our focus will be on the case $k=2$, which we call the \emph{semi-inducibility problem}. We explain the terminology, in particular, its connection with the well-known inducibility problem, in Section~\ref{sec:ind}.

\begin{prob}[The semi-inducibility problem]\label{prob:si}
Let $H$ be a graph whose edges are coloured red and blue
and let $n$ be a positive integer.
Given a complete graph $G$ on $n$ vertices whose edges are coloured red and blue,
let $\fs(H,G)$ be the number of copies of $H$ in $G$.
Determine
$\fmax(H,n)$ and/or $\fmax(H)$.
\end{prob}

\subsection{Main results}

Our work focusses on 2-edge coloured graphs. For us, the colours will always be red and blue, and we say that such a graph is a {\em red-blue graph}.

Let $G$ be a red-blue complete graph. Given a vertex $v\in V(G)$, the set of vertices joined to $v$ by a red edge in $G$ is the \emph{red neighbourhood} $N_R(v)$ and its size $d_R(v)$ is the \emph{red degree} of $v$. The \emph{blue neighbourhood} $N_B(v)$ and \emph{blue degree} $d_B(v)$ are defined analogously.

A path or cycle (on $t$ edges) in a red-blue graph is \emph{alternating} (of \emph{length} $t$) if each of its vertices is incident with at most one edge of each colour.
A \emph{walk} of length $t$ (or $t$-walk) in a graph is an ordered tuple $(v_0,v_1,\ldots,v_t)$ of vertices such that $v_{i-1}v_{i}$ is an edge for each $i \in \{1,\ldots,t\}$. A walk in a red-blue graph is \emph{alternating} when $v_{i-1}v_i$ and $v_iv_{i+1}$ have different colours for each $i \in \{1,\ldots,t-1\}$. Note that reversing the order of the vertices in a walk of length at least 1 produces a different walk.
Hence, for any walk of length $t \geq 1$, there is at least one other walk (and possibly many other walks) of length $t$ using the same set of edges. Our first main result generalises a theorem of Goodman~\cite{Goodman59} (see Theorem~\ref{th:goodman}) by giving a tight upper bound on the number of alternating walks of length $t$ in a red-blue complete graph.
\begin{restatable}{theo}{CSpathBound}
\label{th:CSpathBound}
Every red-blue $K_n$ has at most $2n\left(\frac{n-1}{2}\right)^{t}$ alternating walks of length $t \ge 1$. Equality is attained for colourings in which every vertex is incident with $\frac{n-1}{2}$ red and $\frac{n-1}{2}$ blue edges.
\end{restatable}

Observe that \cref{th:CSpathBound} only gives graphs attaining the maximum possible number of walks when $n \equiv 1 \bmod 4$ (as there are no $\frac{n-1}{2}$-regular graphs on $n$ vertices when $n$ is even or $3\bmod 4$). It would be interesting to know whether other structures attain the bound given in \cref{th:CSpathBound} and if not, figure out what the maximum number of alternating walks of a given length is when $n\not\equiv 1 \bmod 4$.

An  immediate corollary of Theorem~\ref{th:CSpathBound}
is the following bound on the number of alternating paths of length $t$ in a red-blue complete graph.
\begin{cor}
\label{co:CSpathBound}
Every red-blue $K_n$ has at most $n\left(\frac{n-1}{2}\right)^{t}$ alternating paths of length $t \ge 1$.
\end{cor}
The bound in Corollary~\ref{co:CSpathBound} is tight up to the constant in the leading term. For integers $n$ and $t$ we let $(n)_t$ denote the falling factorial $n(n-1)\cdots(n-t+1)$. If in the red graph every vertex has degree exactly $\frac{n-1}{2}$, then the number of alternating paths of length $t$ is at least
$\frac{1}{2}n(n-1)(\frac{n-1}{2})_{t-1}= n \left(\frac{n-1}{2}\right)^{t}+O(n^{t})$.

Our second main result gives a bound on the number of alternating cycles of length $2t$ in a red-blue complete graph. Moreover, for alternating cycles of length $4t$,
we show that this bound is best possible and characterise the extremal graphs.  
\begin{restatable}{theo}{EvenCyclesBound}
\label{th:2kCyclesBound}
For each $t \geq 2$, a red-blue $K_n$ has at most $\frac{1}{t}(\frac{n}{2})^{2t} + O(n^{2t-1})$ alternating $2t$-cycles.
Furthermore, for all $t \geq 1$ and for all sufficiently large $n$, a red-blue copy $G$ of $K_n$ has at most $\frac{1}{2t}(\lceil\frac{n}{2}\rceil)_{2t}(\lfloor\frac{n}{2}\rfloor)_{2t}$ alternating $4t$-cycles. This is achieved if and only if the edges of one colour induce a $K_{\lfloor\frac{n}{2}\rfloor, \lceil \frac{n}{2}\rceil }$.
\end{restatable}

Note that \cref{th:2kCyclesBound} only applies for large $n$. We can extend this to all values of $n$ when $t=1$, that is, in the case of alternating $4$-cycles, denoted by $\CC$.

\begin{restatable}{theo}{RBRBbound}
\label{th:RBRBbound}
For all positive integers $n$, a red-blue $K_n$ has at most 
\[ \mfrac 12 \Bigl\lfloor \mfrac{n^2}{4} \Bigr\rfloor \Bigl\lfloor\mfrac{(n-2)^2}{4}\Bigr\rfloor \]
alternating $4$-cycles.
Moreover, for $n \geq 4$, this is achieved if and only if the edges of one colour induce a $K_{\lfloor\frac{n}{2}\rfloor, \lceil \frac{n}{2}\rceil }$.
\end{restatable}

In the case of $4$-cycles, we also consider non-alternating colour patterns.
An \emph{$RRBB$-cycle}, denoted $\RRBB$, is a red-blue 4-cycle in which precisely two edges are red and there is a vertex which is incident with both of these edges. $RRBB$-cycles share almost the same extremal host graph as alternating $4k$-cycles.
\begin{restatable}{theo}{RRBBbound}
\label{th:RRBBbound}
For all sufficiently large $n$, 
there is a nonempty set $A \subseteq \{\lfloor\frac{1}{2}(n+\sqrt{3n-4})\rfloor,\lceil\frac{1}{2}(n+\sqrt{3n-4})\rceil\}$ such that, for all $a \in A$, a red-blue $K_n$ has at most 
\[ 3\left(a\mbinom{n-a}{3} + (n-a)\mbinom{a}{3}\right) = \mfrac{1}{16}n^4 + O(n^3)\] $RRBB$-cycles.
Moreover, this is achieved if and only if the edges of one colour induce a $K_{a, n-a}$ with $a \in A$.
\end{restatable}

An \emph{$RRRB$-cycle}, denoted $\RRRB$, is a red-blue 4-cycle in which precisely three edges are red. In contrast to alternating and $RRBB$-cycles, the extremal host graph for $RRRB$-cycles is `randomlike'.

\begin{defn}[Quasirandom]\label{def:quasirandom}
Given $\sigma \in [0,1]$ and $\eps>0$, we say an $n$-vertex graph $J$ of density $\sigma$ is \emph{$\eps$-quasirandom} if
$$
\sum_{x,y \in V(J) : x \neq y}\left||N_J(x) \cap N_J(y)| - \sigma^2 n\right| < \eps n^3.
$$
\end{defn}

A fundamental result of Chung, Graham, and Wilson~\cite{chung1989quasi} is that this notion is equivalent, up to polynomial changes in $\eps$, to several other notions, including that for every given graph $F$, the density of copies of $F$ is $(1\pm \eps)\sigma^{e(F)}$, and in fact the weaker statement for just $F=C_4$.
Note that the binomial
random graph $G_{n,\sigma}$ where each edge appears independently with probability $\sigma$ is $o(1)$-quasirandom with probability $1-o(1)$;
this follows from the Chernoff-Hoeffding inequality (see e.g~{\cite[Theorem 21.6]{friezekaronski}}).

We show that, for large $n$, a red-blue $K_n$ is extremal for $RRRB$-cycles if and only if the red graph is $o(1)$-quasirandom of density $\frac{3}{4} \pm o(1)$ (so the blue graph is $o(1)$-quasirandom of density $\frac{1}{4} \pm o(1)$).

\begin{restatable}{theo}{RRRBbound}
\label{th:RRRBbound}
For all positive integers $n$, a red-blue $K_n$ has at most $ \frac{27}{512}n^ 4+O(n^3) $ RRRB-cycles.
Furthermore, for all $0 < \delta < 10^{-6}$, whenever $n$ is sufficiently large and $G$ is a red-blue $K_n$ with more than  
$ \tfrac{27}{512}n^4 - \delta n^4$ RRRB-cycles, the red graph $R$ of $G$ is $(4\delta^{1/8})$-quasirandom of density in $[\frac{3}{4}-4\delta^{1/4},\frac{3}{4}+4\delta^{1/4}]$.
\end{restatable}

Finally, we prove results for certain $H$ which are red-blue $K_{1,1,2}$, that is, $K_4$ minus an edge, or in other words, a $4$-cycle plus an edge. Most of these follow from our results on the (unique) red-blue $4$-cycle in $H$. In fact, we can determine $\max(H,n)$ for large $n$ and all such $H$ apart from those where the $4$-cycle is an $RRRB$-cycle.

Table~\ref{table} contains a summary of our results.

\begin{table}[t]
\def\arraystretch{2.5}

\begin{tabular}{|CC|c|cc|c|c|}
\hline
\multicolumn{2}{|c|}{\textbf{Graph or quantum graph $H$}}                                                                        & \textbf{$\max(H)$}                      & \multicolumn{2}{c|}{\textbf{Extremal graph}}                                                                                                                       & \textbf{Exact?} & \textbf{Reference}                   \\ \hline\hline
\multicolumn{2}{|c|}{Alternating walk of length $t$}                                                            & $\mfrac{1}{2^{t-1}}$          & \multicolumn{2}{c|}{Balanced degrees}                                                                                                                              & Yes    & Theorem~\ref{th:CSpathBound} \\ \hline
\multicolumn{2}{|c|}{Alternating path of length $2t$}                                                            & $\mfrac{1}{2^{t}}$          & \multicolumn{2}{c|}{Balanced degrees}                                                                                                                              & No    & Theorem~\ref{co:CSpathBound} \\ \hline
\multicolumn{2}{|c|}{Alternating cycle of length $4t$}                                                                & $\mfrac{1}{t\cdot 2^{4t+1}}$ & \multicolumn{2}{c|}{\raisebox{-2mm}{\includegraphics[height=0.9cm]{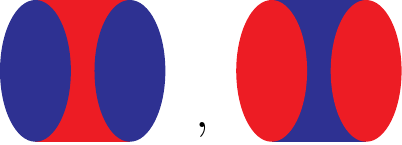}}}                                                                    & Yes    & Theorem~\ref{th:2kCyclesBound}              \\ \hline
\multicolumn{2}{|c|}{\scalebox{1.5}{$\raisebox{-2pt}\CC$}}                                                      & $\mfrac{1}{32}$               & \multicolumn{2}{c|}{\raisebox{-2mm}{\includegraphics[height=0.9cm]{figs/bothbip.pdf}}}                                                                    & Yes    & Theorem~\ref{th:RBRBbound}              \\ \hline
\multicolumn{2}{|c|}{\scalebox{1.5}{$\raisebox{-2pt}\RRBB$}}                                                    & $\mfrac{1}{96}$               & \multicolumn{2}{c|}{\raisebox{-2mm}{\includegraphics[height=0.9cm]{figs/bothbip.pdf}}}                                                                    & Yes    & Theorem~\ref{th:RRBBbound}              \\ \hline
\multicolumn{1}{|C|}{\scalebox{1.5}{$\raisebox{-2pt}\RRRB$}}      & \scalebox{1.5}{$\raisebox{-2pt}\BBBR$}      & $\mfrac{27}{512}$             & \multicolumn{1}{D|}{\raisebox{-2mm}{\includegraphics[height=0.9cm]{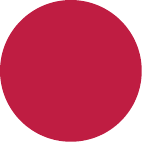}}} & \raisebox{-2mm}{\includegraphics[height=0.9cm]{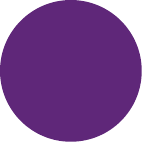}} &  No      & Theorem~\ref{th:RRRBbound}              \\ \hline
\multicolumn{1}{|c|}{\scalebox{1.5}{$\raisebox{-2pt}\CCext$}}     & \scalebox{1.5}{$\raisebox{-2pt}\CCextC$}    & $\mfrac{1}{16}$               & \multicolumn{1}{c|}{\raisebox{-2mm}{\includegraphics[height=0.9cm]{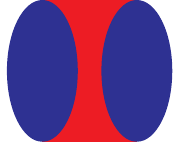}}} & \raisebox{-2mm}{\includegraphics[height=0.9cm]{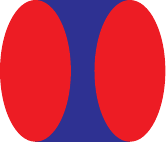}} & Yes    &   Section~\ref{sec:K4-e1}                          \\ \hline
\multicolumn{1}{|c|}{\scalebox{1.5}{$\raisebox{-2pt}\RRBBexta$}}  & \scalebox{1.5}{$\raisebox{-2pt}\RRBBextaC$} & $\mfrac{1}{96}$               & \multicolumn{1}{c|}{\raisebox{-2mm}{\includegraphics[height=0.9cm]{figs/redbip.pdf}}} & \raisebox{-2mm}{\includegraphics[height=0.9cm]{figs/bluebip.pdf}} & Yes    &    Section~\ref{sec:K4-e1}                         \\ \hline
\multicolumn{1}{|c|}{\scalebox{1.5}{$\raisebox{-2pt}\RRBBextbC$}} & \scalebox{1.5}{$\raisebox{-2pt}\RRBBextb$}  & $\mfrac{1}{96}$               & \multicolumn{1}{c|}{\raisebox{-2mm}{\includegraphics[height=0.9cm]{figs/redbip.pdf}}} & \raisebox{-2mm}{\includegraphics[height=0.9cm]{figs/bluebip.pdf}} & Yes    &     Section~\ref{sec:K4-e1}                        \\ \hline
\multicolumn{1}{|c|}{\scalebox{1.5}{$\raisebox{-2pt}\CCextt$}}    & \scalebox{1.5}{$\raisebox{-2pt}\CCexttC$}   & $\mfrac{1}{27}$               & \multicolumn{1}{c|}{\raisebox{-2mm}{\includegraphics[height=0.9cm]{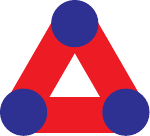}}}     & \raisebox{-2mm}{\includegraphics[height=0.9cm]{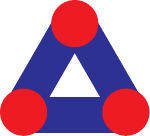}}     & Yes    & Theorem~\ref{theo:K4-e2}              \\ \hline
\end{tabular}
\vspace{0.5cm}
\label{table}
\caption{Summary of our main results. In the extremal graphs where edges of one colour form a bipartite or tripartite graph, the sizes of the parts differ by at most 1, except in the case $H=\RRBB$ where they differ by $\Theta(\sqrt{n})$. The extremal graphs for $\RRRB$ (resp. $\BBBR$) are quasirandom with red (resp. blue) density approaching $\frac{3}{4}$. 
} 
\end{table}  

In the next section, we discuss the context of our results and the connection of the semi-inducibility problem to several important problems in the literature, in particular the inducibility problem of Pippenger and Golumbic.

\section{Ramsey multiplicity and the inducibility problem}
\label{sec:relatedwork}

\subsection{Goodman's bound and Ramsey theory}

It is useful to extend our notation to formal linear combinations of graphs, sometimes called \emph{quantum graphs}, as follows. 
Given $k$-edge-coloured graphs $H_1,\ldots,H_t$, reals $c_1,\ldots,c_t$, 
and a $k$-edge-coloured complete graph $G$,
let
 \begin{equation*}
\fs\left(c_1H_1+\ldots + c_tH_t , G\right) \coloneqq 
\sum_{i \in [t]}c_i \cdot \fs(H_i,G).
\end{equation*}
Thus, given a quantum graph $Q=c_1H_1+\ldots+c_tH_t$ and an integer $n$, we can define 
\begin{align*}
\fmax(Q,n) &\coloneqq \max_{|V(G)|=n}\fs(Q,G)
\quad\text{and}\quad
{\rm min}(Q,n) \coloneqq \min_{|V(G)|=n}\fs(Q,G),
\end{align*}
where the maximum and minimum are taken over all $k$-edge-coloured complete graphs $G$ on $n$ vertices.
Note that while $\min(H,n)$ is trivially $0$ for an edge-coloured graph $H$ (with at least two colours), it is an interesting quantity for a quantum graph $Q$. Also observe that if $c_i > 0$ for each $i$, it suffices to consider complete graphs for the maximisation problem, as in the case of \cref{prob}.
For the minimisation problem however, considering non-complete graphs $G$ could decrease $\ff(Q,G)$, but $\min(H,n)$ would be trivially $0$ if we considered all $n$-vertex graphs $G$.

In 1959 Goodman~\cite{Goodman59} determined the minimum number of monochromatic
triangles in a red-blue complete graph on $n$ vertices, that is ${\rm min}(\RRR+\BBB \ , n)$.
This in fact solves a special case of Problem~\ref{prob}. Indeed, the number of monochromatic triangles in a red-blue complete graph $G$ on $n$ vertices
equals, on the one hand, the total number of triangles $\binom{n}{3}$ in $G$ minus the number of triangles containing both colours. On the other hand, if a triangle contains both colours, then it contains exactly two $\PP$ subpaths. Thus, $\binom{n}{3} - \ff(\RRR+\BBB \ , G ) =
\frac 12\,\ff(\PP \ , G)$ and Goodman's theorem can be stated as an instance of \cref{prob}.

\begin{theo}[Goodman~\cite{Goodman59}]\label{th:goodman}
For all positive integers $n \geq 3$,
$$
\mbinom{n}{3} - {\rm min}\left(\RRR+\BBB \ , n \right) =
\mfrac 12\,\ffmax\left(\PP \ , n\right) = 
\left\lfloor \mfrac{n}{2}\left\lfloor\mfrac{n-1}{2}\right\rfloor\left\lceil\mfrac{n-1}{2}\right\rceil\right\rfloor.
$$
\end{theo}
The number of $\PP$ in $G$, a red-blue $K_n$, can be easily counted by
$$
\ff\left(\PP \ , G\right) =
\sum_{x \in V(G)}d_R(x)(n-1-d_R(x)).
$$
Thus, \cref{th:goodman} implies that the graphs $G$ which maximise the number of $\PP$ are precisely those where, in the red subgraph $G_R$, 
every vertex is incident to $\lfloor \frac{n-1}{2}\rfloor$ or $\lceil \frac{n-1}{2}\rceil$ edges, with the exception of one vertex in the case when $n-3$ is divisible by $4$. Hence, our result Theorem~\ref{th:CSpathBound} is a direct generalisation of Theorem~\ref{th:goodman}.

Furthermore, we have the approximate equality $\ff(\PP \ , G) = (\frac{1}{4}+o(1))n^3$ 
if and only if all but $o(n)$ vertices of $G$ have red and blue degrees both equal to $(\frac{1}{2}+o(1))n$. Of course, many very different red subgraphs $G_R$ attain this approximate equality:
the Erd\H{o}s-R\'enyi random graph $G_{n,1/2}$
in which each pair of vertices independently forms an edge with probability $\frac{1}{2}$ is such a graph with high probability; 
the complete balanced bipartite graph; its complement, the union of two balanced cliques;
or $C_n^{\lfloor n/4\rfloor}$ the $\lfloor\frac{n}{4}\rfloor$-th power of a cycle of length $n$.

In contrast to Goodman's theorem, Theorems~\ref{th:2kCyclesBound} and \ref{th:RBRBbound} show that for alternating $4t$-cycles there is a unique extremal graph up to swapping colours (which is itself extremal for Goodman's theorem).
Whereas a graph is extremal for Goodman's theorem if every vertex has the right red degree, a graph is extremal for Theorems~\ref{th:2kCyclesBound} and~\ref{th:RBRBbound} if every vertex has the right red degree \emph{and} the right number of pairs of vertices have the right red codegree.

Given an uncoloured graph $F$ and a colour $i$,
write $F^{(i)}$ for the monochromatic $i$-coloured $F$.
Ramsey's theorem states that 
$$
M_k(F,n) \coloneqq {\rm min}(F^{(1)}+\ldots+F^{(k)},n)
$$
is positive whenever $n$ is sufficiently large.
Thus we call $M_k(F,n)$ the \emph{Ramsey multiplicity} of $F$. A central problem in combinatorics is to determine the \emph{Ramsey number} of $F$, which is the minimum $n$ for which $M_k(F,n)$ is positive. 
It is also a major open problem to determine $M_2(K_r,n)$ for $r \geq 4$.
Unlike the case $r=3$, the expression $\ff(K_r^{(1)}+K_r^{(2)},G)$ cannot be written only as a function of $\ff(H,G)$ for a single graph $H$ when $r \geq 4$.
Determining which edge-coloured $G$ attains the minimum more generally is also a well-studied question.
Given a positive integer $k \geq 2$, we say that an uncoloured graph $F$ is \emph{$k$-common} if $M_k(F,n)$ 
asymptotically equals, as $n \to \infty$, the expected value of $\ff(K_r^{(1)}+\ldots+K_r^{(k)},G)$
 for $G$ a uniform random $k$-edge-colouring of $K_n$.
Thus Goodman's theorem shows that $K_3$ is $2$-common. 

Erd\H{o}s~\cite{Erdos62} conjectured that every clique is $2$-common,
which was extended to all graphs $F$ by Burr and Rosta~\cite{BurrRosta80}.
This second conjecture was disproved by Sidorenko~\cite{Sidorenko89}, and eventually
Thomason~\cite{Thomason89,Thomason97} disproved the original conjecture of Erd\H{o}s for cliques of order $r \geq 4$.
Sidorenko~\cite{Sidorenko93} and Erd\H{o}s and Simonovits~\cite{Simonovits84} independently made what is usually known as `Sidorenko's conjecture', which implies that every bipartite $F$ is $2$-common; despite attracting a lot of attention and being proved in several cases, it remains open.

Cummings, Kr\'al', Pfender, Sperfeld, Treglown, and Young~\cite{Cummings13} obtained a $3$-coloured analogue of Goodman's theorem by determining 
$M_3(K_3,n)$
and the extremal graphs attaining the bound, for large $n$.
(And $K_3$ is not $3$-common, which was shown earlier in~\cite{Cummings11}.)

\subsection{The inducibility problem}\label{sec:ind}

The inducibility problem is a classical problem in extremal graph theory,
introduced by Pippenger and Golumbic~\cite{PippengerGolumbic75}.
Let $F$ be an \emph{uncoloured} graph
and let $J$ be an uncoloured graph on $n$ vertices.
Let $\ind(F,J)$ be the number of induced copies of $F$ in $J$;
that is,
$$
\ind(F,J) \coloneqq \left|\left\{A \in \mbinom{V(J)}{|V(F)|}: J[A] \cong F\right\}\right|.
$$
Here, as is standard, $\binom{X}{k}$ is the collection of subsets of $X$ of size $k$, $G[A]$ is the subgraph of a graph $G$ induced by $A \subseteq V(G)$, and $G \cong H$ denotes that the graphs $G$ and $H$ are isomorphic.

\begin{prob}[The inducibility problem]\label{prob:ind}
Let $F$ be a graph
and let $n$ be a positive integer.
Determine
$$
\ind(F,n) \coloneqq \max_{|V(J)|=n}\ind(F,J)
\quad\text{and/or}\quad\ind(F) \coloneqq \lim_{n \to \infty}\frac{\ind(F,n)}{\binom{n}{|V(F)|}}.
$$
\end{prob}

Problem~\ref{prob:si} (Problem~\ref{prob} for $k=2$ colours) in the case when $H$ is a complete graph
is equivalent to Problem~\ref{prob:ind}.
Indeed, given a graph $H$ whose edges are coloured red and blue, let $H_B$ be the blue graph. Then
$\ff(H,G)= \ind(H_B,G_B)$ if $H$ is complete. It would make no difference to consider the red graph $H_R$ instead,
since $\ind(F,J)=\ind(\ov{F},\ov{J})$ where $\ov{F}$ and $\ov{J}$ are the complement of $F$ and $J$ respectively, so
$\ind(\ov{F},n)=\ind(F,n)$.

It should now be clear why we call Problem~\ref{prob:si} the \emph{semi-inducibility problem}:
in the inducibility problem, a $|V(F)|$-vertex subset of a host graph $J$ is counted if some specified edges are `in' while the rest are `out'. In Problem~\ref{prob:si}, some specified edges are in, some are out, while the remainder are allowed to be either in or out. (Note that in Problem~\ref{prob:si} a $|V(F)|$-vertex subset may be counted multiple times if it is the vertex set of many copies of the graph $F$ we are interested in.)

We can also define inducibility for quantum graphs.

\begin{prob}[The inducibility problem for quantum graphs]\label{prob:indgen}
Let $F_1,\ldots,F_t$ be graphs on the same number of vertices,
let $c_1,\ldots,c_t \in \mb{R}$
and let $n$ be a positive integer.
Determine
\begin{equation}\label{eq:quantum}
\ind\left(c_1F_1+\ldots + c_tF_t , n\right) \coloneqq 
\max_{|V(J)|=n}\sum_{i \in [t]}c_i \cdot \ind(F_i,J).
\end{equation}
\end{prob}

Problem~\ref{prob:si} is in turn a special case of Problem~\ref{prob:indgen}.
Indeed, let $\mc{G}_h$ be the collection of all red-blue copies of $K_h$ up to isomorphism.
Then, for any red-blue $H$ on $h$ vertices and red-blue complete graph $G$ on $n$ vertices,
$$
\ff(H,G) = \sum_{K \in \mc{G}_h}\ff(H,K) \cdot \ind(K_B,G_B)
\quad\text{ and }\quad
\ffmax(H,n) = \ind\left(\textstyle\sum_{K \in \mc{G}_h}\ff(H,K) \cdot K_B, n\right).
$$
For example, if $H=\CC$, then $\BB$ and $\RR$ contain two copies of $\CC$, while $\BR$ contains one, and each other red-blue $K_4$ contains none.
Thus
\begin{equation}\label{eq:C4}
\ffmax\left(\CC \ ,n\right) = \ind\left(2\cdot\BBn+2\cdot\RRn+\BRn \ , n\right).
\end{equation}

The inducibility problem is in general wide open, and very difficult. 
It has been solved for various small graphs, including all graphs on at most four vertices, except, surprisingly, the $4$-vertex path $P_4$; there is not even a sensible conjecture here.
Another outstanding problem is to determine $\ind(C_m)$ for $m \geq 6$. Pippenger and Golumbic~\cite{PippengerGolumbic75} conjectured that an `iterated blow-up' of $C_m$ is extremal for $m \geq 5$, which would yield $I(C_m)=m!/(m^m-m)$:
this is obtained by taking a balanced blow-up of $C_m$ and then adding a smaller copy of the balanced blow-up inside each class, then repeating in each new class, and so on.
This was proved for $m=5$ by Balogh, Hu, Lidick\'y, and Pfender in~\cite{BaloghHuLidickyPfender16}, a result which was recently improved by Lidick\'y, Mattes, and Pfender~\cite{LidickyMattesPfender} who proved a sharp result for all $n$.
(Trivially $\ind(C_3,n)$ is attained by the complete graph and it was shown in~\cite{PippengerGolumbic75} that $\ind(C_4,n)$ is attained by the complete balanced bipartite graph.)

Extending results of Bollob\'as~\cite{Bollobas76} and Brown and Sidorenko~\cite{BrownSidorenko94},
Schelp and Thomason~\cite{SchelpThomason98} proved a general result about complete multipartite quantum graphs.
All three proofs use Zykov symmetrisation~\cite{Zykov52}, a simple tool originally used by Zykov to give a new proof of Tur\'an's theorem.

\begin{theo}[Schelp and Thomason \cite{SchelpThomason98}]\label{th:ST}
Let $F_1,\ldots,F_t$ be complete multipartite graphs,
let $c_1,\ldots,c_t$ be reals where $c_i \geq 0$
whenever $F_i$ is not a complete graph.
Then $\ind(\sum_{i \in [t]}c_i F_i,n)$ 
is attained by a complete multipartite graph.
\end{theo}

This theorem reduces the problem of determining $\ind(\sum_{i \in [t]}c_i F_i)$
to maximising a polynomial, where the variables are the part sizes in a complete multipartite graph. For example, in a complete $r$-partite graph with parts of size $x_1 n, \ldots, x_r n$,
every induced copy of $\BBn = C_4$
is obtained by choosing two distinct parts and then choosing two distinct vertices in each one. Thus
$$
\ind\!\left(\BBn\right) = \max\left\{ \sum_{1 \leq i < j \leq r} 6x_i^2x_j^2 :
x_1,\ldots,x_r > 0, \ x_1+\ldots+x_r=1, \ r \in \mb{N} \right\},
$$
which has the solution $\frac{3}{8}$ uniquely attained by $(\frac{1}{2},\frac{1}{2})$. We remark that Theorems~\ref{th:2kCyclesBound} and \ref{th:RBRBbound} do not follow from Theorem~\ref{th:ST}; for Theorem~\ref{th:RBRBbound} this can be seen from ~\eqref{eq:C4}. 
We do however use this reduction to prove Theorem~\ref{theo:K4-e2} on a certain red-blue $K_{1,1,2}$ which does correspond to a quantum graph as in Theorem~\ref{th:ST}.

In terms of behaviour of extremal graphs, complete multipartite graphs are unusual. Independently, Yuster~\cite{Yuster19} and Fox, Huang, and Lee~\cite{FoxHuangLee} showed that for almost all graphs $F$, the iterated blow-up of $F$ asymptotically attains $\ind(F,n)$ (and attains it exactly if $n$ is not much larger than $|V(F)|$).
For further results, see~\cite{BollobasEgawaHarrisJin95,BollobasNaraTachibana86,EvenZoharLinial15,FoxSauermannWei21,HatamiHirstNorin14,HefetzTyomkyn18,Hirst14}.

It is of interest to understand which graphs $J$ can be an extremal graph for some (quantum) graph $F$.
Given a graph $F$, current results show that the extremal graph of $F$ is either complete multipartite or the union of complete multipartite graphs (or their complements)~\cite{Bollobas76,BollobasEgawaHarrisJin95, BollobasNaraTachibana86,BrownSidorenko94,Hirst14, SchelpThomason98, LPSS},  a blow-up of $F$~\cite{HatamiHirstNorin14}, an iterated blow-up of $F$~\cite{BaloghHuLidickyPfender16, Yuster19,FoxHuangLee}, or an iterated blow-up of a supergraph of $F$~\cite{FoxSauermannWei21}.

\subsection{Random graphs as maximisers}\label{sec:randmax}

Even-Zohar and Linial~\cite[Section~4]{EvenZoharLinial15} posed the question of exploring the role
of random constructions in the study of inducibility. 
Their flag algebra calculations suggest randomlike graphs could be extremal (e.g.~for $F$ which is a $4$-cycle with a pendant edge, the numerical upper bound from flag algebra calculations is very close to the lower bound from the construction obtained by independently adding edges with probability $\frac{5}{6}$ between two disjoint vertex sets of equal size).

Liu, Mubayi, and Reiher~\cite{LMR} initiated the study of the `feasible region' of induced graphs, which generalises the inducibility problem in another direction.
Its `upper boundary' is defined as follows.
Given $\sigma \in [0,1]$ and a quantum graph $Q$, let
\begin{align*}
\ind(Q,\sigma) := \sup\{y: &\text{ there exists a sequence }J_1,J_2,\ldots\text{ of graphs with }a_n := |V(J_n)|\to\infty\\
&\text{and }\textstyle e(J_n)/\binom{a_n}{2} \to \sigma
\text{ and }\ind(Q,J_n)/\binom{a_n}{|V(Q)|} \to y\}.
\end{align*}
So $\ind(Q) = \sup_{\sigma \in [0,1]} \ind(Q,\sigma)$.
Its lower boundary $i(Q,\sigma)$ is defined as above with infimum replacing supremum.
We discuss the feasible region further in Section~\ref{sec:concind}.
Let 
$$
{\rm rand}(Q,\sigma) := \lim_{n\to\infty}\mb{E}(\ind(Q,G_{n,\sigma})/\tbinom{n}{|V(Q)|}). 
$$
For a non-quantum graph $F$, we have 
\begin{equation}\label{eq:rand}
{\rm rand}(F,\sigma) = \frac{|V(F)|!}{|\Aut(F)|}\sigma^{e(F)}(1-\sigma)^{\binom{|V(F)|}{2}-e(F)}
\end{equation}
and so ${\rm rand}(Q,\sigma)$ exists (and can be written explicitly in terms of its constituent parts). 
Recall Definition~\ref{def:quasirandom}. 
We say that a sequence $J_1,J_2,\ldots$ of graphs with $a_n := |V(J_n)| \to \infty$ is \emph{quasirandom} if there is a sequence $\eps_n \to \infty$ such that $J_n$ is $\eps_n$-quasirandom.
The Chernoff-Hoeffding inequality shows that a sequence of instances of the binomial random graph $G_{n,\sigma}$ has density approaching $\sigma$; is quasirandom; and has induced $Q$-density approaching ${\rm rand}(Q,\sigma)$ almost surely. Thus $I(Q,\sigma) \geq {\rm rand}(Q,\sigma)$.
Liu, Mubayi and Reiher asked whether there is a non-quantum graph $F$ and density $\sigma$ with $\ind(F,\sigma) = {\rm rand}(F,\sigma)$.
Jain, Michelen and Wei answered this negatively.

\begin{theo}[\cite{Jain2023binomial}]\label{th:JMW}
    For all graphs $F$ on at least three vertices and all densities $\sigma \in (0,1)$, we have
    $$
    \ind(F,\sigma) > {\rm rand}(F,\sigma).
    $$
\end{theo}
(Note that for $\sigma$ equal to $0$ or $1$, the empty and complete graphs show that the result cannot be extended to the closed interval.)
In fact, they proved the stronger result that $G_{n,\sigma}$ is not even a `local maximiser'. See~\cite{Jain2023binomial} for more details.

We show that this phenomenon does \emph{not} extend to the more general setting of quantum graphs with positive coefficients, where in fact there are graphs for which the binomial random graph is asymptotically extremal for many densities, and crucially is unique in the sense that every near-extremal graph is quasirandom.
It is the characterisation of near-extremal graphs which is important here, as it is easy to find a quantum graph with positive coefficients for which a random graph is asymptotically extremal:
if $Q$ is the sum of all graphs on $k$ vertices up to isomorphism, then every graph is extremal.

There are many examples of the binomial random graph being asymptotically extremal for a quantum graph whose coefficients are not necessarily positive.
Indeed, Sidorenko's conjecture is that for every bipartite $F$, the number of (not necessarily induced) copies of $F$ in an $n$-vertex graph of a given density is asymptotically minimised by the binomial random graph of that density.
That is, the quantum graph $Q = \sum_i F_i$ which `counts the number of copies of $F$' satisfies $i(Q,\sigma)={\rm rand}(Q,\sigma)$, so $\ind(-Q,\sigma) = -i(Q,\sigma) = -{\rm rand}(Q,\sigma)$.
The conjecture has been proved for many graphs, for example $C_4$, for which we have
$Q = 3\cdot\Kfour + \Kminus + \BBn$.

We obtain a quantum graph with positive coefficients for which there is an interval of edge densities where the binomial random graph is asymptotically extremal.

\begin{theo}\label{th:rand}
Let $Q := 2\cdot\Kminus + 2\cdot \Kplus + \BRn$.
    For all $\frac{1}{4}(1+\sqrt{2}) \leq \sigma \leq 1$, we have
    $$
    \ind(Q,\sigma) = {\rm rand}(Q,\sigma)
    $$
and also $\ind(Q) = \sup_{\sigma \in [0,1]}{\rm rand}(Q,\sigma)$.
Moreover, for all $0 < \delta < 10^{-6}$, whenever $n$ is sufficiently large and $J$ is an $n$-vertex graph with $\sigma\binom{n}{2}$ edges with $\ind(Q,J) > ({\rm rand}(Q,\sigma) - \delta)\binom{n}{4}$, we have that $J$ is $(3\delta^{1/8})$-quasirandom of density $\sigma$.
\end{theo}

Theorem~\ref{th:rand} is a corollary of the following density version of Theorem~\ref{th:RRRBbound} which shows that, for the semi-inducibility problem, there is a red-blue graph $H$, namely $\RRRB$, and an interval of densities $\sigma$ where the red-blue binomial random graph is a maximiser, and it is unique in a similar sense.

\begin{restatable}{theo}{RRRBprofile}
\label{th:RRRBprofile}
Let $\frac{1}{4}(1+\sqrt{2}) \leq \sigma \leq 1$.
For all positive integers $n$, a red-blue copy of $K_n$ with $\sigma\binom{n}{2}$ red edges has at most
$\frac{1}{2}\sigma^3(1-\sigma)n^4+O(n^3)$
copies of $\RRRB$. Furthermore, for all $0 < \delta < 10^{-6}$, whenever $n$ is sufficiently large and $G$ is a red-blue copy of $K_n$ with 
$\sigma\binom{n}{2}$ red edges and
\[\#(\RRRB,G) > \tfrac{1}{2}\sigma^3(1-\sigma)n^4 - \delta n^4,\]
the red graph $R$ of $G$ is $(4\delta^{1/8})$-quasirandom
of density $\sigma$.
\end{restatable}

We note that random extremal graphs are of interest in the directed setting too.
Burke, Lidick\'y, Pfender and Phillips~\cite{BurkeLidickyPfenderPhillips} found a $4$-vertex tournament in which all extremal constructions contain a quasirandom component and asked~\cite[Problem~1]{BurkeLidickyPfenderPhillips} for a characterisation of directed or undirected graphs whose extremal graphs have this structure.

\subsection{Further related work}
As far as we are aware, Problem~\ref{prob} has not previously been explicitly stated in full generality. Hence, there are few results other than those we have already stated in the context of related problems.

Erd\H{o}s and S\'os (see \cite{ErdosHajnal72}) made the conjecture that $f(n) \coloneqq \fmax(K_3^{\rm rb},n)$, where $K_3^{\rm rb}$ is the rainbow triangle,
satisfies the recurrence
$$
f(n) = f(a)+ f(b)+f(c)+f(d)+abc+abd+acd+bcd,
$$
where $a+b+c+d = n$ and $a, b, c, d$ are as equal as possible.
This comes from an iterated blow-up of a certain $3$-edge-coloured copy of $K_4$.
This was proved by Balogh, Hu, Lidick\'y, Pfender, Volec, and Young \cite{BaloghHuLidickyPfenderVolecYoung17}
for sufficiently large $n$ as well as for $n$ equal to a power of $4$.

De Silva, Si, Tait, Tun\c{c}bilek, Yang, and Young~\cite{DeSilva18} introduced the rainbow case of Problem~\ref{prob} (where every edge of $H$ has a different colour).
They investigated various $H$ for which $\max(H)$ is asymptotically attained by a randomly coloured clique, with high probability.
Cairncross and Mubayi~\cite{CairncrossMubayi} solved the case of a certain $3$-edge coloured clique: the $(s+t)$-clique $K^{s,t}_{s+t}$ with vertex partition $V_1 \cup V_2$ where $|V_1|=s$ and $|V_2|=t$ in which every edge in $V_1$ is red, every edge in $V_2$ is blue, and all edges between them are green.
As well as determining $\fmax(K^{s,t}_{s+t})$,
they asymptotically determined the maximum over all $3$-edge-coloured complete graphs $G$ with a given red, blue, and green density.
Recently, Cairncross, Mizgerd, and Mubayi~\cite{cairncross2024inducibility} determined $\max(K_k^{\rm rb})$ for $K_k^{\rm rb}$ a rainbow clique on $k \geq 11$ vertices. They showed that the asymptotically extremal graphs are balanced iterated blow-ups of $K_k^{\rm rb}$.
As stated above, this is not the case for $K_3^{\rm rb}$, and it can be shown~(see~\cite{cairncross2024inducibility}) that $\ind(C_5)=5!/(5^5-5)$ implies the result also holds for $K_5^{\rm rb}$. Thus the remaining open cases are $k \in \{4,6,7,8,9,10\}$.

Some rainbow cases of Problem~\ref{prob} have been investigated in ordered and directed graphs. This is mainly due to a close connection with a generalised Ramsey parameter which is the subject of an important $\$ 500$ conjecture of Erd\H{o}s and Hajnal~\cite{ErdosHajnal72};
for more details, see~\cite{MubayiRazborov21}.

A key tool that has been used in many of the results we have mentioned is the method of 
flag algebras developed by Razborov~\cite{Razborov07}.
This method uses semi-definite programming to find valid inequalities for subgraph densities.
In this paper we instead use elementary graph theoretic and optimisation methods.
However, it seems likely that flag algebras could be used to make further progress on Problem~\ref{prob}.

Inducibility and related problems have been studied in the directed setting. For example, Hu, Ma, Norin, and Wu~\cite{HuMaNorinWu} studied the directed inducibility problem for oriented stars.
Sah, Sawhney, and Zhao~\cite{SahSawhneyZhao} determined the maximum number of directed paths of given length in a tournament (an oriented complete graph).
In fact this last result has a more than tangential connection to our work and we use its ideas along the way to proving Theorem~\ref{th:CSpathBound}.

\subsection{Known results on the semi-inducibility problem}

We can now summarise what was known about the semi-inducibility problem (Problem~\ref{prob:si}) before our work.
Recall that finding $\max(H,n)$ is equivalent to an inducibility problem when $H$ is a red-blue complete graph, and that these have been solved for certain sporadic $H$. In particular, the inducibility problem has been solved for all graphs with at most four vertices except for the four-vertex path $\BRn$. Furthermore, $\ffmax(\PP \ , n)$ is determined by Theorem~\ref{th:goodman} (Goodman's theorem). So $\max(H,n)$ is known for all graphs $H$ with at most three vertices.
Our results on four-vertex graphs therefore constitute some of the smallest open cases of Problem~\ref{prob:si}.

\section{Preliminaries}
\label{sec:prelim}

\subsection{Organisation and proof ideas}

Our proofs use a variety of elementary methods. Section~\ref{sec:prelim} contains some general results and tools. One helpful tool is Lemma~\ref{lm:twinning} which states that every vertex in an extremal graph for $H$ lies in roughly the same number of copies of $H$. For our results where the extremal graph has one colour which induces a spanning complete bipartite graph, we first prove a stability result and then use a general lemma (Lemma~\ref{L:deltato0}) to derive the exact result. This lemma uses `edge flips' to show that an extremal graph can have no imperfections.

Section~\ref{sec:4cycle} proves Theorem~\ref{th:RBRBbound} on alternating $4$-cycles $\CC$, which uses an optimisation argument that counts cycles from antipodal vertices.

Our result on alternating walks (Theorem~\ref{th:CSpathBound}) is proved in Section~\ref{sec:altwalks} and uses a Cauchy-Schwarz argument of Sah, Sawhney, and Zhao from their work on maximising the number of oriented paths of a given length in a tournament~\cite{SahSawhneyZhao}.

In Section~\ref{sec:altcycles}, we use Theorem~\ref{th:CSpathBound} to determine the maximum number of alternating cycles of length $4t$ (Theorem~\ref{th:2kCyclesBound}) starting with the observation that in an extremal graph, almost every walk of length $4t-1$ must extend to an alternating cycle.

We prove Theorems~\ref{th:RRBBbound},~\ref{th:RRRBbound} and~\ref{th:RRRBprofile} on non-alternating $4$-cycles $\RRBB$ and $\RRRB$ in Section~\ref{sec:nonalt}. 
We also give the short derivation of Theorem~\ref{th:rand} from Theorem~\ref{th:RRRBprofile}.
For each of these first three results, we consider a numerical relaxation of the problem in terms of the (red) degrees and (red) codegrees of a graph, and solve the resulting problem using analytic and combinatorial techniques. This is described further at the beginning of Section~\ref{sec:nonalt}.

Section~\ref{sec:K4-e} contains results on red-blue colourings of $K_{1,1,2}$. The solution for four of these colourings follow easily from our results on red-blue $4$-cycles and a general lemma (Lemma~\ref{L:extendingH}), while to determine $\max(\CCextt,n)$ we use Theorem~\ref{th:ST} and a result of~\cite{LPSS} on `symmetrisable functions' and methods for analysing complete multipartite graphs.

We do not use the machinery of flag algebras, which nevertheless seems like it may be helpful to resolve Problem~\ref{prob:si} for certain small $H$.

\subsection{Notation}
We fix some notation for the rest of the paper.
Recall that a \emph{red-blue graph} is a graph in which each edge is either coloured red or blue.
We say that a red-blue graph $H$ is a \emph{subgraph} of a red-blue graph $G$ if the uncoloured $H$ is a subgraph of the uncoloured $G$ and the colours also match.
Usually $G$ denotes a red-blue $K_n$, with vertex set $V$ of size $n$, in which we are counting copies of a red-blue graph $H$ with $h$ vertices.
We take $R$ and $B$ to be the (uncoloured) graphs induced by the red edges and the blue edges of $G$ respectively.
Throughout we represent small red-blue graphs pictorially, with blue edges dashed to improve readability.

Let $L$ be a red-blue graph. Given (not necessarily disjoint) $A,B\subseteq V(L)$, we write $L[A]$ for the subgraph of $L$ induced on $A$ and $L[A,B]$ for the red-blue subgraph of $L$ induced by the edges of $L$ which have one endpoint in $A$ and the other in $B$. (In particular, if $A=B$, then $L[A,B]\cong L[A]$.) Given a red-blue subgraph $M$ of $L$, or set $M$ of edges of $L$, we write $L-M$ for the graph obtained from $L$ by deleting all the edges of $M$. Given a subset $S$ of $V(L)$, we write $L-S$ for the graph obtained from $L$ by deleting the vertices in $S$ and all the edges incident with them.
We say a vertex in a red-blue graph is \emph{bichromatic} if it is incident with at least one edge of each colour and is \emph{monochromatic} if it is incident with edges of only one colour.

Given a red-blue (quantum) graph $H$ and given $G$ a red-blue $K_n$, we say that $G$ is \emph{extremal (for $H$)} if $\#(H,G) = \max(H,n)$.

For convenience, we assume that $V$ is ordered, and write $\sum_{x < y}$ and $ \sum_{x \neq y}$ as shorthand for $\sum_{x \in V}\sum_{y \in V: x<y}$ (a sum over unordered pairs) and $\sum_{x \in V}\sum_{y \in V: x \neq y}$ (a sum over ordered pairs), respectively. When we sum over the entire vertex set $V$, we write $\sum_x$ as shorthand for $\sum_{x\in V}$. Similarly, when we sum over an entire index set $S$, we write $\sum_i$ as shorthand for $\sum_{i\in S}$.

Given two sets $A,B$, we write $A \mathrel\triangle B = (A \sm B) \cup (B \sm A)$ for the \emph{symmetric difference} of $A$ and $B$.

We write $c=a\pm b$ as short-hand for $a-b \leq c \leq a+b$. We write \textit{hierarchies} of the form $0 < a \ll b \ll c < 1$ in statements and proofs to mean that they hold whenever we choose the constants $a, b, c$ from right to left such that $c<1$, then $b$ is taken sufficiently small in terms of $c$, and then $a$ sufficiently small in terms of $b$.
More precisely, this means that there exist non-decreasing functions $f : (0, 1] \to (0, 1]$ and $g : (0, 1] \to (0, 1]$ such
that for all $c<1$, all $b \leq f (c)$, and all $a \leq  g(b)$ our statements hold and the calculations and arguments in our proofs are correct. Hierarchies with more terms are defined analogously.

\subsection{Inequalities etc.}
We will make use of
\begin{itemize}
    \item the AM-GM inequality, (a special case of) which states that for any two non-negative real numbers $a,b$ we have $\sqrt{ab}\leq \frac{a+b}{2}$; 
    \item the Cauchy-Schwarz inequality, which states that for any real $x_1,\ldots,x_n$ and $y_1,\ldots,y_n$, we have $\sum_i x_iy_i \leq \sqrt{(\sum_i x_i^2) (\sum_i y_i^2)}$;
    \item Jensen's inequality, which states that for any function $f$ and $x_1, \dots, x_n$ in its domain, we have
$\frac1n\sum_if(x_i)\leq f(\frac{1}{n}\sum_i x_i)$ if $f$ is concave, and $\frac1n\sum_if(x_i)\geq f(\frac{1}{n}\sum_i x_i)$ if $f$ is convex,
in both cases with equality if and only if $x_1=\ldots=x_n$ or $f$ is linear;
    \item the mean value theorem, which states that for any continuous function $f$ on $[a,b]$ which is differentiable on $(a,b)$, there exists $c\in (a,b)$ such that $f(b)-f(a)=(b-a)f'(c)$.
\end{itemize}

\subsection{General results}
As promised, we first show that the limit $\fmax(H)$ exists. Note that this is analogous to showing that $\ind(F)$ exists, which was first proved in \cite{PippengerGolumbic75}.

\begin{prop}\label{pr:limit}
For all $k \geq 1$ and every $k$-edge-coloured graph $H$, the limit $\fmax(H)$ exists.
\end{prop}

\begin{proof}
Fix a $k$-edge-coloured graph $H$ and let $h\coloneqq |V(H)|$. We will show that $\lim_{n \to \infty}\fmax(H,n)/(n)_{h}$ exists. This suffices to complete the proof, since it is clear that $\fmax(H)$ is equal to this limit. 
    
Let $n\geq h+1$ and let $G$ be a $k$-edge-coloured $K_n$ with exactly $\fmax(H,n)$ copies of $H$. For any $v\in V := V(G)$, the subgraph $G-\{v\}$ has at most $\fmax(H,n-1)$ copies of $H$, while each copy of $H$ in $G$ is contained in exactly $n-h$ such subgraphs. Therefore,
\[(n-h)\cdot \fmax(H,n)\leq n\cdot \fmax(H,n-1).\]
Dividing both sides by $(n)_{h+1}$, we obtain
\[\frac{\fmax(H,n)}{(n)_{h}}\leq \frac{\fmax(H,n-1)}{(n-1)_{h}}.\]
As $\max(H,n)\geq 0$ for all $n$, this implies that the limit $\lim_{n \to \infty}\fmax(H,n)/(n)_{h}$ exists.
\end{proof}

We begin by collecting some basic facts about the semi-inducibility problem.
\begin{fact}
\label{fact:basic}
Let $H$ be a red-blue graph on $h$ vertices and let $n$ be a positive integer.
Then the following hold:
\begin{enumerate}[label=\rm(\roman*)]
    \item A uniformly random red-blue $K_n$ contains $\Theta(n^h)$ copies of $H$ with probability $1 - o(1)$, so $\max(H, n) = \Omega(n^h)$ by the first moment method.\label{fact:basic-random}
    \item If $H$ is monochromatic, then the unique extremal graph is the complete graph in the colour of $H$.\label{itm:mono}
    \item If $H$ is complete, then $\max(H,n)=\ind(H_R,n)=\ind(H_B,n)$
    where $H_R$ is the red graph and $H_B$ the blue graph of $H$.
    \item\label{itm:swap} If $H'$ is obtained from $H$ by swapping colours, then $\max(H',n)=\max(H,n)$ and the extremal graphs of $H'$ are obtained from those of $H$ by swapping colours.
    \item \label{itm:quantum} If $Q$ is a quantum red-blue graph and $c>0$ is a constant, then $\max(c \cdot Q,n) = c \cdot \max(Q,n)$. Moreover, $c\cdot Q$ and $Q$ have the same extremal graphs.
\end{enumerate}
\end{fact}

In each shared row in Table~\ref{table}, the second result follows from the first together with Fact~\ref{fact:basic}\ref{itm:swap}.

Next, we observe that for any red-blue graph $H$, if $G$ is a large extremal graph for $H$, then each vertex of $G$ is contained in asymptotically the same number of copies of $H$.

\begin{lemma}\label{lm:twinning}
Let $0<\frac{1}{n}\ll \frac{1}{h}\leq \frac{1}{2}$. Let $H$ be a red-blue graph on $h$ vertices and $G$ be a red-blue $K_n$ such that $\#(H, G) = \max(H, n)$. 
Then there is a constant $c \coloneqq c(h)$ such that each vertex is in
$$
\left(1\pm \mfrac cn\right)\mfrac{h \cdot \max(H,n)}{n}
$$
copies of $H$ in $G$.
\end{lemma}

\begin{proof}
Let $N\coloneqq \max(H,n)$.
By \cref{fact:basic}\cref{fact:basic-random}, we have $N = \Omega(n^h)$. On the other hand, any two vertices appear together in $O(\binom{n-2}{h-2}) = O(n^{h-2})$ copies of $H$. Hence, we may choose $c > 0$ such that any two vertices appear together in at most $\frac{chN}{2n^2}$ copies of $H$.

Let $G$ be a red-blue $K_n$ with $N$ copies of $H$, i.e.,  $G$ achieves the maximum number of copies of $H$. 
On average, a vertex is in $\frac{hN}{n}$ copies of $H$. We now suppose for a contradiction that not every vertex is in $(1\pm \frac cn)\frac{h N}{n}$ copies of $H$. Then we can choose vertices $x_1, x_2 \in V$ in $N_1, N_2$ copies of $H$ respectively, such that $N_2 \leq \frac{h N}{n} \leq N_1$ and
\[N_1 - N_2 > \mfrac{c}{n} \cdot\mfrac{ h N}{n}.\]

Let $G'$ be the red-blue $K_n$ obtained from $G$ by replacing $x_2$ with a vertex $x_2'$ which is twin to $x_1$, i.e., with $N_R(x_2') = N_R(x_1)$, and colouring the edge $x_1x_2'$ arbitrarily. 
Note that the number of copies of $H$ not containing $x_2$ is unchanged. On the other hand, $x_2'$ is contained in at least $N_1 - \frac{chN}{2n^2}$ copies of $H$ that do not include $x_1$.
Thus, $G'$ has at least
\[N-N_2+N_1-\mfrac{chN}{2n^2}>N+\mfrac{chN}{2n^2}>N\]
copies of $H$, a contradiction.
\end{proof}

The next lemma shows that for certain red-blue graphs $H^- \subset H$, the extremal graphs of $H^-$ and $H$ are the same.

\begin{lemma}\label{L:extendingH}
Let $H$ be a red-blue graph, let $H^-$ be a spanning subgraph of $H$, and let $n$ be a positive integer. Suppose that there are $t$ ways to extend a fixed labelled copy of $H^-$ to a labelled copy of $H$. Then
\begin{equation}\label{eq:extendingH}
\max(H,n) \leq \mfrac{t \cdot \max(H^-,n)}{\#(H^-,H)}.
\end{equation}
Furthermore, let $\mathcal{M}$ be the set of all extremal $n$-vertex graphs for $H^-$ and let $\mathcal{M}_t$ be the set of all graphs $G$ in $\mathcal{M}$ such that every copy of $H^-$ in $G$ is a subgraph of exactly $t$ copies of $H$ in $G$. If $\mathcal{M}_t$ is nonempty then we have equality in~$(\ref{eq:extendingH})$ and $\mathcal{M}_t$ is precisely the collection of extremal $n$-vertex graphs for $H$.
\end{lemma}

\begin{proof}
Let $G$ be an arbitrary red-blue $K_n$. Each copy of $H$ in $G$ contains $s \coloneqq \#(H^-,H)$ copies of $H^-$ and each copy of $H^-$ in $G$ is a subgraph of at most $t$ copies of $H$ in $G$. Thus $\#(H,G) \leq \frac{t}{s}\#(H^-,G)$ with equality only if each copy of $H^-$ in $G$ is a subgraph of exactly $t$ copies of $H$ in $G$. Hence $\max(H,n) \leq \frac{t}{s}\max(H^-,n)$ and $\#(H,G) = \frac{t}{s}\max(H^-,n)$ if and only if $G \in \mathcal{M}_t$.
\end{proof}

\subsection{\texorpdfstring{$\delta$}{delta}-bipartite and partitioned graphs}

The remainder of this section contains tools for the case when the extremal graph of a red-blue $H$ has a bipartite structure: the edges of one colour form a spanning complete bipartite graph which is close to being balanced.

Towards our stability results, we introduce the notion of  $\delta$-bipartite graphs. Roughly speaking, these are red-blue complete graphs that have small edit distance to a red-blue complete graph where one colour forms a nearly balanced complete bipartite graph. 
\begin{defn}[$\delta$-bipartite]
Let $G$ be a red-blue $K_n$. For $0 \leq \delta \leq 1$, we say $G$ is \emph{$\delta$-bipartite} if there is a bipartition $V = X \cup Y$ such that the following hold:
\begin{enumerate}[label=(B\arabic*)]
\item $\lfloor\frac{n}{2}\rfloor - \delta n \leq |X|, |Y| \leq \lceil\frac{n}{2}\rceil + \delta n$.\label{item:Bsize}
\item Either $|E(R) \mathrel\triangle E(K_{X, Y})| \leq \delta \binom{n}{2}$ or $|E(B) \mathrel\triangle E(K_{X, Y})| \leq \delta \binom{n}{2}$, where $K_{X, Y}$ is the complete bipartite graph with parts $X$ and $Y$.\label{item:Bcolour}
\end{enumerate}
We refer to $\{X,Y\}$ as a \emph{$\delta$-bipartition} of $G$.
\end{defn}

\begin{defn}[Minority graph]
    Let $G$ be a $\delta$-bipartite graph with $\delta<\frac{1}{2}$.
    For a $\delta$-bipartition $\{X, Y\}$ of $G$, we define the \emph{minority graph} $M(X, Y)$ to be the (uncoloured) graph with edge set $E(R) \bigtriangleup E(K_{X,Y})$ if the number of such edges is at most $\delta\binom{n}{2}$,
    or 
    $E(B) \bigtriangleup E(K_{X,Y})$ if the number of such edges is at most $\delta\binom{n}{2}$
    (by \cref{item:Bcolour} one of these must hold).
    That is, $M(X, Y)$ consists of those edges that do not follow the colours given by the bipartition. We refer to the edges of $M$ as \emph{minority edges}.
\end{defn}

The following lemma is a version of the well-known fact that large (almost) Dirac graphs,~i.e.\ $n$-vertex graphs with minimum degree (almost) at least $\frac n2$, have a special structure: either every pair of sets which are not much smaller than half-sized have many edges between them, or the graph is close to the union of two disjoint cliques of size $\frac n2$, or is close to containing its complement (a complete balanced bipartite graph) as a subgraph.
When most degrees are roughly $\frac n2$, it is easy to see that the final case becomes that the graph is close to a complete balanced bipartite graph.

\begin{lemma}[Lemma~3.2 in~\cite{cheng2024stability}]\label{L:dirac}
Let $0 < \frac1n \ll \eps \ll 1$.
Let $G$ be a red-blue $K_n$ and let $C \in \{R, B\}$. Suppose that $d_C(y) = (\frac{1}{2}\pm\eps)n$ for all but at most $\eps n$ vertices $y \in V$, and suppose there are two sets $X,Y$ with $|X|,|Y| \geq (\frac{1}{2}-\eps)n$ and $e_C(X,Y) \leq \eps n^2$.
Then $G$ is $\eps^{1/4}$-bipartite.
\end{lemma}

For a number of graphs $H$ the extremal construction for the semi-inducibility problem is a red-blue graph where the edges of one colour induce a complete bipartite graph. We refer to such graphs as \emph{partitioned graphs}. 
\begin{defn}[$(a,b)$-partitioned graph]
We say that a red-blue complete graph $G$ is \emph{$(a,b)$-partitioned} if it has a vertex bipartition $V = X \cup Y$ and there is a colour $C \in \{R,B\}$
such that $C \cong K_{X,Y}$
and $|X|=a$ and $|Y|=b$. Note that if $|a-b| \leq 1$ then $G$ is $0$-bipartite.

We say two partitioned graphs have the \emph{same colour pattern} if the same colour $C\in \{R,B\}$ induces a complete bipartite graph in both.
\end{defn}

The following proposition is an immediate consequence of the definition of a partitioned graph.
\begin{prop}\label{prop:count}
    Let $H$ be a spanning connected subgraph of an $(a,b)$-partitioned graph $G_0$. 
    Let $G$ be an $(s,t)$-partitioned graph with the same colour pattern where $\min\{s,t\} \geq \max\{a,b\}$. Then 
    $$
    \#(H,G) = 
    \begin{cases}
        \binom{s}{a}\binom{t}{b}\cdot \#(H,G_0) &\mbox{ if }a=b\\
        \left(\binom{s}{a}\binom{t}{b}+\binom{s}{b}\binom{t}{a}\right) \cdot \#(H,G_0) &\mbox{ if }a \neq b.
    \end{cases}
    $$
\end{prop}

Next we observe that embeddings of a graph $H$ in a large $\delta$-bipartite graph have a very specific structure.

\begin{prop}\label{prop:embed}
    Let $0<\frac{1}{n}\ll \delta\ll  \frac{1}{h}\leq 1$. Let $G$ be a $\delta$-bipartite red-blue $K_n$ and let $\{X,Y\}$ be a $\delta$-bipartition of $G$. Let $H$ be a connected red-blue $h$-vertex graph which is a subgraph of a partitioned graph with the same colour pattern as $G$. Fix $i\in V(H)$ and $x\in V := V(G)$.
    \begin{enumerate}[label=\rm(\roman*)]
        \item Let $\mc{C}_x(i)$ be the set of copies of $H$ in $G$ where $i$ is mapped to $x$ and no edge is mapped to a minority edge. For each $j\in V(H)$, there exists a unique $P_j\in \{X,Y\}$ such that for any copy of $H$ in $\mc{C}_x(i)$, the image of $j$ lies in $P_j$.\label{prop:embed-nominority}
        \item Let $\mc{C}'_x(i)$ be the set of copies of $H$ in $G$ where $i$ is mapped to $x$ and an edge of $H$ is mapped to a minority edge if and only if it is incident to $i$. For each $j\in V(H)$, there exists a unique $Q_j\in \{X,Y\}$ such that for any copy of $H$ in $\mc{C}'_x(i)$, the image of $j$ lies in $Q_j$.\label{prop:embed-minority}
        \item If $H$ is a cycle, then any copy of $H$ in $G$ contains no or at least two minority edges. In particular, if $H$ is $2$-connected and $H_0$ is a copy of $H$ in $G$ in which $i$ is mapped to $x$ and any minority edge is the image of an edge incident to $i$, then $H_0\in \mc{C}_x(i)\cup \mc{C}'_x(i)$.\label{prop:embed-2}
    \end{enumerate}    

\end{prop}

\begin{proof}
    For \cref{prop:embed-nominority}, suppose for a contradiction that there are distinct collections $\{P_j\in \{X,Y\}: j\in V(H)\}$ and $\{P_j'\in \{X,Y\}: j\in V(H)\}$ such that there is a copy of $H$ whose $j$-th vertex is in $P_j$ for all $j \in V(H)$ and another copy of $H$ whose $j$-th vertex is in $P'_j$ for all $j \in V(H)$. 
    Note that the definition of $\mc{C}_x(i)$ implies that $P_i=P_i'$. 
    This induces an auxiliary vertex colouring of $H$ where any $j\in V(H)$ is coloured blue if $P_j=P_j'$ and red otherwise. By assumption, this colouring is not monochromatic and so, since $H$ is connected, it contains an edge incident to both colours. That is, there exists $jj'\in E(H)$ such that $P_j=P_j'$ but $P_{j'}\neq P_{j'}'$ 
    and so only one of $G[P_j,P_{j'}]- M(X,Y)$ and $G[P_j',P_{j'}']- M(X,Y)$ contain edges of the colour of $jj'$ in $H$, a contradiction.

    The proof of \cref{prop:embed-minority} is analogous.
       
    For \cref{prop:embed-2}, assume that $H$ is a cycle and suppose for a contradiction that $H_1=x_1\dots x_h$ is a copy of $H$ in $G$ which contains a unique minority edge, say $x_h x_1$.
    For each $j\in [h]$, let $j$ be the preimage of $x_j$.
    
    Denote by $G'$ the partitioned graph on $V$ obtained from $G$ by swapping the colour of each minority edge.
    Applying part \cref{prop:embed-nominority} with $1, x_1,G'$ playing the roles of $i,x,G$, we obtain collections $\{P_j\in \{X,Y\}: j\in [h]\}$ and $\{P_j'\in \{X,Y\}: j\in [h]\}$ such that in any copy of $H$ in $G'$ where $1$ is mapped to $x_1$, the image of $j$ lies in $P_j$ for each $j\in [h]$,
    and in any copy of the path $H-h1$ in $G'$ where $1$ is mapped to $x_1$, the image of $j$ lies in $P_j'$ for each $j\in [h]$. Note that such copies of $H$ and $H-h1$ in $G'$ exist since $H$ is a subgraph of a partitioned graph with the same colour pattern as $G$.
    Since $H-h1$ is a spanning subgraph of $H$, we have $P_j=P_j'$ for all $j \in [h]$.
    Since $H_1-x_hx_1$ is a subgraph of $G'$ and a copy of $H-h1$, we have $x_j \in P_j'=P_j$ for all $j \in [h]$.
    Thus $x_1\ldots x_h$ is a copy of $H$ in $G'$.
    But the edges $x_1x_2,\ldots,x_{h-1}x_h$ have the same colour in $G$ and $G'$ while $x_hx_1$ has a different colour, so the cycle $H_1=x_1\ldots x_h$ is not a copy of $H$ in $G$, a contradiction.

   The `in particular' part of \cref{prop:embed-2} follows since any two incident edges of a $2$-connected graph belong to a common cycle.
\end{proof}

Building on this exhibited structure of embeddings in $\delta$-bipartite graphs, the following definition introduces a general class of graphs $H$ which are good potential candidates for having a partitioned graph as extremal graph.

Let $H$ be a red-blue subgraph of a partitioned graph and $0\leq \alpha, \beta\leq 1$.
For each $i\in V(H)$, denote by $d_r^i$ and $d_b^i$ the red and blue degree of $i$ in $H$, respectively, and define
\[p_H^i(\alpha,\beta)\coloneqq (1-\alpha)^{d_b^i}(1-\beta)^{d_r^i}+ \alpha^{d_r^i}\beta^{d_b^i} \qquad \text{and} \qquad p_H(\alpha,\beta)\coloneqq\sum_{i\in V(H)}p_H^i(\alpha,\beta).\]
Roughly speaking, we will see in \cref{L:deltato0} that $p_H^i(\alpha,\beta)$ counts the number of copies of $H$ where $i$ is embedded into a vertex with an $\alpha$-proportion of red minority edges and a $\beta$-proportion of blue minority edges.

We say that an $h$-vertex subgraph $H$ of a partitioned graph is \emph{canonical} if for any $0<\delta\ll \eta \ll 1$ and $0\leq \alpha,\beta\leq 1$ such that $\eta\leq \alpha+\beta \leq 1+\delta$, we have $p_H(\alpha,\beta)\leq h-\frac{\eta}{2}$. Roughly speaking, a subgraph $H$  of a partitioned graph is canonical if a $\delta$-bipartite graph contains relatively few copies of $H$ with minority edges.
We prove this in the next lemma. We will then give a concrete class of graphs which are canonical (\cref{prop:canonical}) and mention an example which is not.

\begin{lemma}\label{L:deltato0}
    Let $0 < \frac{1}{n} \ll \delta \ll \frac{1}{h}\leq 1$.
    Let $H$ be a $2$-connected canonical red-blue $h$-vertex graph which is a subgraph of a partitioned graph.
    If $G$ is a $\delta$-bipartite red-blue $K_n$ such that $\#(H,G)=\max(H,n)$, then  $G$ is $(a,b)$-partitioned where $|a-b| \leq \delta n$.
\end{lemma}

\begin{proof}
    Choose a new constant $\eta$ with $\delta\ll \eta \ll \frac{1}{h}$.
Let $G$ be a $\delta$-bipartite red-blue $K_n$ such that $\#(H,G)=\max(H,n)$. Recall that $R$ and $B$ denote the red and blue subgraphs of $G$, respectively.

Let $\{S^*, T^*\}$ be a $\delta$-bipartition of $G$ and let $M^* \coloneqq M(S^*, T^*)$ be the associated minority graph. Suppose without loss of generality that $e_R(S^*,T^*) > e_B(S^*,T^*)$, otherwise we may interchange the role of red and blue in the argument below.
So most edges between $S^*$ and $T^*$ are red, while most edges inside these parts are blue.
We claim that the red subgraph of $H$ is bipartite. (Since $H$ is a subgraph of a partitioned graph, either this is true, or the blue subgraph of $H$ is bipartite.)
Indeed, suppose otherwise. Then every copy of $H$ in $G$ contains a minority edge. Since any minority edge can be embedded into a copy of $H$ in at most $2e(H)$ ways, this implies that
$$\#(H,G) \leq 2e(H) \cdot \delta \mbinom{n}{2} \cdot n^{h-2} < \delta h^2 n^{h}. $$
On the other hand, $\max(H,n)=\Omega(n^h)$ by \cref{fact:basic}\cref{fact:basic-random}. This contradicts the maximality of $G$ since $\delta \ll \frac{1}{h}$. 

We claim that there is a bipartition $\{S, T\}$ of $V$ such that the associated minority graph $M\coloneqq M(S,T)$ satisfies $\Delta(M)\leq \frac{n-1}{2}$. Let $V_0 \coloneqq \{x \in V:d_{M^*}(x) > \frac{n-1}{3}\}$ and note that $|V_0| \leq 3\delta n$. Let $\{S,T\}$ be a bipartition of $V$ that induces the minimum number of minority edges subject to the constraints that $S^* \setminus V_0 \subseteq S$ and $T^* \setminus V_0 \subseteq T$. We claim that $\{S, T\}$ has the required property. If $x \in V_0$, then $d_M(x) \leq \frac{n-1}{2}$ follows because otherwise we could move $x$ to the other part and contradict the minimality of $\{S,T\}$. If $x \notin V_0$, this follows because 
\[d_{M}(x) \leq d_{M^*}(x)+|V_0| \leq \mfrac{n-1}{3}+3\delta n < \mfrac{n-1}{2}.\]
Furthermore, $\{S,T\}$ is a $4\delta$-bipartition of $G$ since $e(M) \leq e(M^*)$, $\left||S|-|S^*|\right| \leq 3\delta n$, and $\left||T|-|T^*|\right| \leq 3\delta n$. This proves the claim.

Next, we claim that $\Delta(M)\leq \eta n$. 
Suppose otherwise that there is a vertex $x \in V$ such that $d_M(x) > \eta n$ and suppose without loss of generality that $x \in S$. We will show that $G$ contains many
fewer labelled copies of $H$ containing $x$ than the partitioned graph obtained from $G$ by flipping the colours of the minority edges of $G$, which will contradict the maximality of $G$. 

We first count the number of labelled copies of $H$ in $G$ containing $x$.
Let $\mathcal{D}_x$ be the set of all labelled copies of $H$ containing $x$ and let $\mc{D}^\dag_x$ be the set of labelled copies in $\mc{D}_x$ that contain a minority edge not incident with $x$. Then
\[|\mc{D}^\dag_x| \leq  h\cdot 2e(H) \cdot 4\delta \mbinom{n}{2} \cdot n^{h-3} < 4\delta h^3 n^{h-1}. \]

To compute $|\mc{D}_x\setminus \mc{D}^\dag_x|$, set $S' \coloneqq S \cap N_R(x)$ and $T' \coloneqq T \cap N_B(x)$,
so $N_M(x) = S' \cup T'$,
and let $\alpha \coloneqq |S'|/|S|$ and $\beta \coloneqq |T'|/|T|$.
Our assumption on $x$ implies that $\alpha+\beta\geq \eta$. 
Moreover, we have
$\alpha|S|+\beta|T|\leq \frac{n-1}{2}$, which implies that
$\frac{(\alpha+\beta)n}{2}\leq \frac{n}{2}+||S|-\frac{n}{2}||\alpha-\beta|$
and so $\alpha+\beta\leq 1+8\delta$.
For $i\in V(H)$, denote by $d_r^i$ and $d_b^i$ the red and blue degree of $i$ in $H$, and let $\mc{D}_x(i)$ be the set of labelled copies of $H$ containing $x$ in which every minority edge is incident to $x$ and $x$ is the image of $i$.

Let $i\in V(H)$.
Then, the `in particular' part of \cref{prop:embed}\cref{prop:embed-2} implies that any copy of $H$ in $\mc{D}_x(i)$ either contains no minority edge, or maps all edges incident to $i$ to a minority edge.
Moreover, \cref{prop:embed}\cref{prop:embed-nominority,prop:embed-minority} imply there exist collections $\{P_j\in \{S,T\}: j\in V(H)\}$ and $\{Q_j\in \{S,T\}: j\in V(H)\}$ such that, in any copy of the first type, $j$ is mapped to a vertex in $P_j$ for all $j\in V(H)$, while in any copy of the second type, $j$ is mapped to a vertex in $Q_j$ for all $j\in V(H)$.
Thus, putting $U \coloneqq V(H) \sm (N_H(i) \cup \{i\})$,
\begin{align*}
|\mc{D}_x(i)| &\leq |S \sm S'|^{d_b^i}|T \sm T'|^{d_r^i}\prod_{j \in U}|P_j|\;
+  \;|S'|^{d_r^i}|T'|^{d_b^i}\prod_{j \in U}|Q_j|\\
&\leq \left(\mfrac{1}{2}+4\delta\right)^{h-1}n^{h-1}p_H^i(\alpha,\beta).
\end{align*}

Since $H$ is canonical, we altogether obtain that
\begin{align*}
|\mc{D}_x| &= \sum_{i \in V(H)}|\mc{D}_x(i)| + |\mc{D}_x^{\dagger}|
\leq \left(\mfrac{1}{2}+4\delta\right)^{h-1} n^{h-1}p_H(\alpha,\beta)+4\delta h^3 n^{h-1}\\
&\leq \left(\mfrac{1}{2}+4\delta\right)^{h-1}\left(h-\mfrac{\eta}{2}\right) n^{h-1}+4\delta h^3 n^{h-1}
\leq \left(\mfrac{1}{2}-6\delta\right)^{h-1}\left(h-\mfrac{\eta}{3}\right) n^{h-1}.
\end{align*}
Together with \cref{lm:twinning}, we get that
\begin{equation}\label{eq:canonical}
    \#(H,G)\leq \left(h-\mfrac{\eta}{3}\right)\left(\mfrac{1}{2}-5\delta\right)^{h-1} \mfrac{n^h}{h}.
\end{equation}

Now let $G^*$ be the $(|S|, |T|)$-partitioned graph obtained by flipping the colours of all minority edges in $G$. That is, all edges in $G^*[S,T]$ are red, and all edges in $G^*[S]$ and $G^*[T]$ are blue. Given $i\in V(H)$, $y\in V$, and writing $\mc{D}_y^*(i)$ for the set of labelled copies of $H$ in $G^*$ in which $y$ is the image of $i$, we have
$$|\mc{D}_y^*(i)| = \prod_{j \in V(H) \sm \{i\}}|P_j|-O\bigl(n^{h-2}\bigr) 
\geq \left(\mfrac{1}{2}-5\delta\right)^{h-1}n^{h-1}.$$
So, for any $y\in V$, there are at least $\sum_{i \in V(H)}|\mc{D}_y^*(i)| \geq h(\frac{1}{2}-5\delta)^{h-1}n^{h-1}$ labelled copies of $H$ in $G^*$ containing $y$, and hence the total number of labelled copies of $H$ in $G^*$ is at least $h(\frac{1}{2}-5\delta)^{h-1}\frac{n^h}{h}$. By \cref{eq:canonical}, this is a contradiction to the maximality of $G$.
Hence, $\Delta(M)\leq \eta n$, as claimed. 

Finally, we claim that $M$ is empty. Suppose otherwise that $M$ contains an edge $xy$. The `in particular' part of \cref{prop:embed}\cref{prop:embed-2} implies that any labelled copy of $H$ containing $xy$ must contain at least one other minority edge.
The number of labelled copies of $H$ containing $xy$ and a minority edge not adjacent to $xy$ is at most ${4e(H)^2 \cdot 4  \delta  \binom{n}{2} \cdot n^{h-4} < 2\delta h^4 n^{h-2}}$. Since $\Delta(M)\leq \eta n$, there are at most $\eta h^3 n^{h-2}$ labelled copies of $H$ containing $xy$ and an edge adjacent to $xy$ in $M$. Altogether this implies that there are at most $2\eta h^3 n^{h-2}$ labelled copies of $H$ containing the edge $xy$.

Now consider the graph $G'$ obtained by swapping the colour of $xy$ (so $xy$ is not a minority edge in $G'$). Recall that $|S|, |T| \geq (\frac{1}{2} - 4\delta) n$ and $\Delta(M)\leq \eta n$. By greedily embedding $H$ in $G$ one vertex at a time avoiding minority edges, this implies that the number of labelled copies of $H$ in $G'$ containing $xy$ is at least $((\frac{1}{2} - 4\delta-h\eta) n )_{h-2} > 2 \eta h^3 n^{h-2}$.
Since $G$ and $G'$ have the same number of copies of $H$ not containing $xy$, this implies that $G'$ has more copies than $G$, a contradiction.
Thus $M$ is empty, which completes the proof of the lemma.
\end{proof}

We now observe that subgraphs of a partitioned graph which have no monochromatic vertex or are invariant under swapping colours are canonical.

\begin{prop}\label{prop:canonical}
    Let $H$ be a red-blue $h$-vertex graph which is a subgraph of a partitioned graph and suppose that $H$ has no isolated vertex and at least one bichromatic vertex. Let $H'$ be obtained from $H$ by swapping the colours. Suppose that $H$ has no monochromatic vertex or $H\cong H'$. Then $H$ is canonical.
\end{prop}

\begin{proof}
    Let $0<\delta\ll \eta\ll 1$ and $0\leq \alpha, \beta\leq 1$ such that $\eta\leq \alpha+\beta\leq 1+ \delta$. 
    For each $i\in V(H)$, denote by $d_r^i$ and $d_b^i$ the red and blue degree of $i$ in $H$, respectively.
    Define
    \begin{align*}
        X&\coloneqq \{i\in V(H) : \min\{d_r^i,d_b^i\}\geq 1\},\\
        Y&\coloneqq \{i\in V(H) : d_r^i=0, d_b^i\geq 1\},\\
        Z&\coloneqq \{i\in V(H) : d_b^i=0, d_r^i\geq 1\}.
    \end{align*}
    Since $H$ has no isolated vertex, $X,Y$, and $Z$ partition $V(H)$.
    Define a function $f\colon x \mapsto 1+\frac{x^2}{2}-x$ and observe that $f(x)$ is a quadratic with positive leading coefficient.
    For any $i\in X$, we have
    \begin{align*}
        p_H^i(\alpha,\beta)&= (1-\alpha)^{d_b^i}(1-\beta)^{d_r^i}+ \alpha^{d_r^i}\beta^{d_b^i} \leq (1-\alpha)(1-\beta)+ \alpha\beta
        \leq 1+ \mfrac{(\alpha+\beta)^2}{2}-(\alpha+\beta)\\
        &\leq \max\{f(\eta),f(1+\delta)\}
        \leq 1-\mfrac{\eta}{2}.
    \end{align*}
     Since $Y=\emptyset=Z$ or $H\cong H'$, there is a bijection $g\colon Y\to Z$ such that for any $i\in Y$, we have $d_r^i=0=d_b^{g(i)}$ and $d_b^i=d_r^{g(i)}\geq 1$. For each $i\in Y$, we have
     \begin{align*}
        p_H^i(\alpha,\beta) + p_H^{g(i)}(\alpha,\beta) &= \left((1-\alpha)^{d_b^i}(1-\beta)^{d_r^i}+ \alpha^{d_r^i}\beta^{d_b^i}\right)+
        \left((1-\alpha)^{d_b^{g(i)}}(1-\beta)^{d_r^{g(i)}}+ \alpha^{d_r^{g(i)}}\beta^{d_b^{g(i)}}\right)\\
        &= \left((1-\alpha)^{d_b^i}+\beta^{d_b^i}\right)+\left((1-\beta)^{d_b^i}+\alpha^{d_b^i}\right)\\
        &\leq (1-\alpha)+\beta+(1-\beta)+ \alpha =2.
    \end{align*}
    Note also that
    $X\neq \emptyset$ since $H$ has a bichromatic vertex. Therefore
    \begin{align*}
        p_H(\alpha,\beta)\leq |X|\left(1-\mfrac{\eta}{2}\right)+2|Y|=h-\mfrac{\eta|X|}{2}\leq h-\mfrac{\eta}{2},
    \end{align*}
    as desired.
\end{proof}

Observe that if there is no correspondence between $Y$ and $Z$ as above, then the vertices in $Y\cup Z$ may incur too large terms which may not be balanced off if $X$ is too small. In this case, $H$ may not be canonical. In fact, we now give an example of a small ($2$-connected) non-canonical graph.

Consider a red $K_{2,3}$ on vertex classes $S$ and $T$ with $|S|=2$ and $|T|=3$, and let $H$ be obtained by adding a blue edge inside $T$. 
Then,
\[p_H(0.99,0.01)=2(2\cdot 0.99^3)+(2\cdot 0.99^2) + 2(2\cdot 0.01^1\cdot 0.99^2)=5.8806> 5,\]
so $H$ is not canonical.

Note that the above $H$ has six red edges and one blue, so we might expect its extremal graph to a have a similar colour distribution and hence that a $\delta$-bipartite graph would not be extremal.
However, it is not true that a graph is canonical if and only if it has a $\delta$-bipartite extremal graph. For example, the red-blue $K_7$ whose red edges induce a $K_{2,5}$ is canonical (as it has no monochromatic vertex) but it is known that its extremal graph is a partitioned graph with approximate part ratios $\frac{1}{3},\frac{2}{3}$ (see \cite[Theorem~1.3]{LPSS}). One can therefore wonder whether a broader definition of canonical which accounts for unbalanced part sizes could help us characterise graphs whose extremal graphs are partitioned.

\section{Alternating 4-cycles}\label{sec:4cycle}
In this section, we prove Theorem~\ref{th:RBRBbound} which gives a sharp upper bound on the number of alternating $4$-cycles in a red-blue $K_n$ for \emph{any} positive integer $n$. For convenience, we repeat the statement.
\RBRBbound*
The proof of Theorem~\ref{th:RBRBbound} is different to that of Theorem~\ref{th:2kCyclesBound} which replaces $4$ by any given multiple of $4$ but applies only to sufficiently large $n$.
The key step is obtaining the tight inequality
\[ \ff\left(\CC \ ,G\right)\leq \mfrac{n-2}{8} \cdot \ff\left(\PP \ ,G\right) \]
by considering antipodal pairs in copies of $\CC$. The right-hand side can in turn be easily upper-bounded as in Goodman's theorem (Theorem~\ref{th:goodman}).

We begin by proving a simple lemma about optimising sums of products.

\begin{lemma} \label{lm:optimization combined}
Let $M, s$ and $n$ be positive integers.
For $1\le i\leq n$, suppose that $a_i,b_i$ are non-negative integers with $a_i+b_i\leq s$,
and assume that $\sum_{i=1}^n(a_i+b_i)=M$. 
Then \[ \sum_{i=1}^na_ib_i\leq \mfrac{Ms}4, \]
with equality if and only if $a_i=b_i=\frac{s}{2}$ for $\frac{M}{s}$ values of $i$ and $a_i=b_i=0$ for all other values of $i$.
\end{lemma}

\begin{proof}
Subject to $a_ib_i=c_i^2$ for a non-negative real number $c_i$, the sum $a_i+b_i$ is minimised  when $a_i=b_i=c_i$. Thus it suffices to maximise $\sum_{i=1}^nc_i^2$ subject to $2\sum_{i=1}^nc_i=M$, where $0 \leq c_i \leq \frac{s}{2}$ for all $1 \leq i \leq n$. However, this sum is maximised when $c_i \in \{0,\frac{s}{2}\}$ for all but at most one value of $i$. Hence
\[ \sum_{i=1}^na_ib_i\leq \left(\mfrac{M}{2} \Big/ \mfrac{s}{2}\right) \cdot\mfrac{s^2}{4}=\mfrac{Ms}4\]
with equality if and only if $a_i=b_i=\frac{s}{2}$ for $\frac{M}{s}$ values of $i$ and $a_i=b_i=0$ for all other values of $i$.
\end{proof}

Given vertices $x,y\in V$ and a positive integer $k$, let $w_k^R(x, y)$ denote the number of alternating walks of length $k$ from $x$ to $y$ such that the first edge traversed is red. We have the key identity
\begin{equation*}
\ff\left(\PP \ , G\right)=\sum_{x \neq y}w_2^R(x,y).
\end{equation*}
We can use Lemma~\ref{lm:optimization combined} together with this identity to bound the number of alternating 4-cycles $\CC$ in terms of the number of alternating $2$-paths $\PP$.

\begin{lemma} \label{4 cycle count}
Let $G$ be a red-blue complete graph.
If for all vertices $x,y\in V$ we have $w_2^R(x,y)+w_2^R(y,x)\le s$, then
\[ \ff\left(\CC \ ,G\right)\leq \mfrac s8 \cdot \ff\left(\PP \ ,G\right). \]
Moreover, equality only holds if $w_2^R(x,y)=w_2^R(y,x)=\frac s2$ for $\frac 1s\cdot \ff(\PP,G)$ pairs $x,y$ and $w_2^R(x,y)=w_2^R(y,x)=0$ for all other pairs.
\end{lemma}

\begin{proof}
For vertices $x, y\in V$, let $c(x,y)$ be the number of 4-cycles $xuyv$ such that $xu,yv$ are red and $vx,uy$ are blue. Thus $c(x,y)=w_2^R(x,y)w_2^R(y,x)$ since there are $w_2^R(x,y)$ choices for $u$ and $w_2^R(y,x)$ choices for $v$, and these choices are independent of each other.

Moreover, $\sum_{x<y}c(x,y) = 2\cdot \ff(\CC \ ,G)$, since each cycle is counted once for each  of the two choices of antipodal vertices. Thus
\[ \ff\left(\CC \ ,G\right) = \mfrac12 \sum_{x<y}w_2^R(x,y)w_2^R(y,x). \]
On the other hand, we have
\[ \sum_{x<y}\left(w_2^R(x,y)+w_2^R(y,x)\right)=\sum_{x\neq y}w_2^R(x,y)=\ff\left(\PP \ , G\right). \]
The assertion now follows from Lemma~\ref{lm:optimization combined}.
\end{proof}

\begin{proof}[Proof of Theorem~\ref{th:RBRBbound}]
Let
\[ a(n)\coloneqq2\mbinom{\lceil n/2\rceil}{2}\mbinom{\lfloor n/2\rfloor}{2}=\mfrac 12 \left\lceil \mfrac n2\right\rceil  \left\lfloor \mfrac{n}2\right\rfloor \left\lceil \mfrac{n-2}{2}\right\rceil \left\lfloor \mfrac{n-2}2\right\rfloor=\mfrac 12 \Bigl\lfloor \mfrac{n^2}4\Bigr\rfloor \Bigl\lfloor\mfrac{(n-2)^2}{4}\Bigr\rfloor. \]
We want to show that 
$\ffmax(\CC,n)=a(n)$. This value is achieved by partitioning the vertex set into two parts $X,Y$ that are as equal in size as possible, and giving every edge between $X$ and $Y$ one colour, and every other edge the other colour. 
Indeed, every copy of $\CC$ has two vertices in $X$ and two vertices in $Y$, and for each fixed choice of these vertices, we obtain two copies of $\CC$.

We will use induction on $n$ to show that $\ffmax(\CC,n)\le a(n)$ and that there are no other constructions achieving this value. This will complete the proof.

For this, suppose first that $n \leq 3$. Then clearly $\ffmax(\CC,n)=0=a(n)$.
If $n=4$, then $\BB$ and $\RR$ contain two copies of $\CC$, while $\BR$ contains one, and all other red-blue $K_4$ have none, so $\ffmax(\CC,4)=2=a(4)$ and we have the required structure for equality.
If $n=5$ and $G$ is a red-blue $K_5$ with at least $a(5)=6$ copies of $\CC$, then since there are five vertex-subsets of size four, there is one with at least two copies of $\CC$. Thus, without loss of generality, $G$ contains an induced $\BB$ with vertices $x,y,z,w$ such that $xy$ and $zw$ are red. Let $v$ be the fifth vertex of $G$. Then $v$ lies in at least four copies of $\CC$. 
But each red edge incident with $v$, say $xv$, is in at most two $\CC$ since the other red edge in such a cycle must be $wz$. Thus $v$ is incident with at least two red edges. Similarly each blue edge incident with $v$, say $zv$, is in at most two $\CC$, since such a cycle must now contain the red edge $zw$. So $v$ is incident with exactly two red and two blue edges. If the red edges of $G$ form a $P_5$, then we have exactly four $\CC$. So the red edges form the vertex-disjoint union of $K_3$ and $K_2$.

Now suppose $n \geq 6$ and let $G$ be a red-blue $K_n$. 
Since $n-2 \geq 4$, the induction hypothesis allows us to assume that for all distinct $x,y\in V$, we have $\ff(\CC, G-\{x,y\}) \leq a(n-2)$ with equality if and only if $V \setminus \{x,y\}$ has a partition into two parts $X',Y'$ that are as equal in size as possible, where every edge between $X'$ and $Y'$ has one colour, and every other edge the other colour.

Observe that for any $x,y\in V$, we have $w_2^R(x,y)+w_2^R(y,x)\le n-2$ since only vertices distinct from $x,y$ can be the midpoint of a walk of length 2 between $x$ and $y$. If equality holds we will call the edge $xy$ {\em full}.

\medskip
\noindent
{\bf Case 1: There is no full edge.}

\noindent
It follows that $w_2^R(x,y)+w_2^R(y,x)\le n-3\eqcolon s$ for every edge $xy$. Combining Lemma~\ref{4 cycle count} and Theorem~\ref{th:goodman} we obtain

\begin{align*}\ffmax\left(\CC \ ,n\right)&\leq \mfrac s8 \cdot \ffmax\left(\PP \ ,n\right)=\mfrac{n-3}4 \left\lfloor \mfrac{n}{2}\left\lceil \mfrac{n-1}{2}\right\rceil \left\lfloor \mfrac{n-1}{2}\right\rfloor\right\rfloor\le \mfrac{n(n-3)}{8}  \left\lceil \mfrac{n-1}{2}\right\rceil \left\lfloor \mfrac{n-1}{2}\right\rfloor  \\
&< \mfrac 12 \cdot\left\lceil \mfrac n2\right\rceil  \left\lfloor \mfrac{n}2\right\rfloor \left\lceil \mfrac{n-2}{2}\right\rceil \left\lfloor \mfrac{n-2}2\right\rfloor = a(n).\end{align*}
To see the last inequality, one can consider even and odd $n$ separately. 
Specifically, observe that we cannot have equality in this case.

\medskip
\noindent
{\bf Case 2: There is a full edge $xy$.}

\noindent
We may assume without loss of generality that $xy$ is red. Let $P\coloneqq N_B(x)$ and $Q\coloneqq N_R(x)-\{y\}$. Since $xy$ is full it follows that $N_B(y)=Q$ and $N_R(y)=P\cup\{x\}$. 
Let $p\coloneqq |P|$, $q\coloneqq |Q|$, and $r$ be the number of red edges $uv$ with $u\in P$ and $v\in Q$.
To complete the proof, it suffices to prove firstly that the number of $\CC$ involving $x$ and/or $y$ is at most 
\[ a(n)-a(n-2)=\mfrac 12  \Bigl\lfloor\mfrac{(n-2)^2}{4}\Bigr\rfloor\Bigl( \Bigl\lfloor \mfrac{n^2}4\Bigr\rfloor-\Bigl\lfloor \mfrac{(n-4)^2}4\Bigr\rfloor\Bigr)=\Bigl\lfloor\mfrac{(n-2)^2}{4}\Bigr\rfloor(n-2), \]
and secondly, that we have both equality here and $\ff(\CC,G-\{x,y\})=a(n-2)$ only if $G$ has the required bipartite structure.

We proceed with the first assertion.
There are exactly $pq$ copies of $\CC$ that have $x$ and $y$ as antipodal vertices, and $r$ copies in which $x$ and $y$ are adjacent.  
Every $\CC$ involving $x$ but not $y$ is of the form $uxvw$  with $u\in P$ and $v\in Q$. Thus precisely one of $uw,vw$ has one endpoint in $P$ and the other in $Q$. If this edge is red, then $w\in Q$, so this edge is $uw$ and there are at most $q-1$ choices for $v$. If this edge is blue, then $w\in P$, and this edge is $vw$ and there are at most $p-1$ choices for $u$. Overall the number of $\CC$ involving $x$ but not $y$ is at most $r(q-1)+(pq-r)(p-1)$.

By symmetry, the number of $\CC$ involving $y$ but not $x$ is at most $r(p-1)+(pq-r)(q-1)$. Thus the total number of $\CC$ involving $x$ or $y$ is at most
\begin{align*}pq+&r+r(q-1)+(pq-r)(p-1)+r(p-1)+(pq-r)(q-1)\\ &=pq(p+q-1)+r=pq(n-3)+r\le pq(n-2)\le a(n)-a(n-2),
\end{align*}
where we are using the facts that $p+q=n-2$, that $r\le pq$, and that $pq\leq \lfloor\frac{(n-2)^2}{4}\rfloor$ since $p+q=n-2$. 

For the second assertion, we have that the number of $\CC$ involving $x$ and/or $y$ equals $a(n)-a(n-2)$ only if $r=pq$ and $\{p,q\}=\{\lfloor\frac{n-2}{2}\rfloor,\lceil\frac{n-2}{2}\rceil\}$.
So $G-\{x,y\}$ contains a spanning complete balanced bipartite red graph (with parts $P,Q$).
We have $\ff(\CC,G-\{x,y\})=a(n-2)$ only if one of the red or blue graph of $G-\{x,y\}$ is a spanning complete bipartite graph which is as close to being balanced as possible. These statements can only hold simultaneously if it is the red graph, and neither $P$ nor $Q$ contain any red edges.

Let $X \coloneqq P \cup\{x\}$ and $Y \coloneqq Q\cup\{y\}$.
Then $X,Y$ is a partition of $V$ into two parts as equal in size as possible,
and the red graph of $G$ is precisely the complete bipartite graph between $X$ and $Y$, as desired.
\end{proof}

\section{Alternating walks}\label{sec:altwalks}

In this section, we are interested in maximising the number of alternating walks of a given length $t$ in a red-blue $K_n$: that is, the number of ordered tuples $(v_0,v_1,\ldots,v_t)$  of vertices where $v_{i-1}v_i$ and $v_iv_{i+1}$ have different colours for each $i \in [t-1]$. 

In large complete graphs, almost all walks of fixed length are paths, and so we also obtain a result for alternating paths $x_0x_1\ldots x_t$ (where all $x_i$ are distinct and $x_{i-1}x_i$ and $x_ix_{i+1}$ have different colours).
To be precise, given $k$, we let $P_{2k}$ be the alternating path on $2k$ edges, let $P^R_{2k+1}$ be the alternating path on $2k+1$ edges whose first and last edges are both red, and define $P^B_{2k+1}$ analogously.
Counting alternating walks of length $t$ is equivalent to a case of the quantum version of Problem~\ref{prob:si}, as follows. 
Define the $(t+1)$-vertex quantum graph $Q_t$ by
setting $Q_t = 2\cdot P_t$ if $t$ is even and $Q_t = 2 \cdot P^R_t + 2 \cdot P^B_t$ if $t$ is odd.
Then the number of alternating walks of length $t$ in $G$, a red-blue $K_n$, is equal to $\#(Q_t+Q'_1+\ldots + Q'_s,G)$ for some not necessarily distinct red-blue graphs $Q'_1,\ldots,Q'_s$ on at most $t$ vertices, for some $s=s(t)$ (corresponding to walks with repeated vertices), and is hence equal to $\#(Q_t,G)+O(n^t)$.
Thus, using Fact~\ref{fact:basic}\ref{itm:quantum}, the maximum number of alternating walks of length $t$ in a red-blue $K_n$ equals
\begin{equation}
\label{eq:walkpath}
2\cdot \max(P_t,n) + O(n^t)
\text{ if }t\text{ is even }
\quad\text{and}\quad
2\cdot \max(P^R_t+P^B_t,n) + O(n^t)
\text{ if }t\text{ is odd.}
\end{equation}
That is, the number of alternating walks is approximately twice the number of alternating paths.

We first prove Theorem~\ref{th:CSpathBound}, repeated below.
The proof is essentially that of a result of Sah, Sawhney and Zhao~\cite{SahSawhneyZhao} on the maximum number of walks/paths of a given length in a tournament, where the role of an arc $\ova{xy}$ is replaced by the red-blue path $\PP$ from $x$ to $y$. 

We first collect some notation and facts about walks that we will use throughout this and the next section. Let $G$ be a red-blue $K_n$. Let $W_t$ be the number of alternating walks of length $t \geq 1$.
Recall from \cref{sec:4cycle} that given an integer $k \geq 1$ and $x,y\in V$, we denote by $w_k^R(x,y)$ the number of alternating walks of length $k$ from $x$ to $y$ whose first edge is red. Also let $w_k(x)$, $w^R_k(x)$, and $w^B_k(x)$ be the number of alternating walks of length $k$ starting at vertex $x$ in which the first edge traversed is, respectively, of either colour, red, and blue. Note that  $w^R_1(x) = d_R(x)$. We also set $w^R_0(x) \coloneqq 1 \eqqcolon w^B_0(x)$.
For $k\geq 0$, let \[ \rho_k \coloneqq \sum_{x}w^R_k(x)^2 d_B(x) \quad\text{and}\quad \beta_k \coloneqq \sum_{x}w^B_k(x)^2 d_R(x). \]
Note that $\rho_0 = \sum_{x} d_B(x) = 2e(B)$ and $\beta_0 =\sum_{x} d_R(x)= 2e(R)$.

\CSpathBound*
\begin{proof}
That $W_1 \leq n(n-1)$ is immediate, since this is twice the number of edges.

For $x,y\in V$, let $r_{xy}$ (respectively $b_{xy}$) equal 1 if $x \neq y$ and the edge $xy$ is red (respectively blue) and equal 0 otherwise.
Then, for each $k \geq 1$ and $x\in V$, we have
\begin{equation*}
    w^R_k(x) = \sum_{y} r_{xy} w^B_{k-1}(y) \quad\text{ and }\quad w^B_k(x) = \sum_{y} b_{xy} w^R_{k-1}(y).
\end{equation*}
Hence, for $k \geq 1$, we have
\begin{align*}
\rho_k & = \sum_{x}w^R_k(x)^2 d_B(x) = \sum_{x} \left( \sum_{y} r_{xy} w^B_{k-1}(y) \right)^2 d_B(x)\\
& = \sum_{x, y, y'} r_{xy} w^B_{k-1}(y) r_{xy'} w^B_{k-1}(y') d_B(x)\\
& \leq \sum_{x, y, y'} \left( \frac{w^B_{k-1}(y)^2 + w^B_{k-1}(y')^2}{2} \right) r_{xy} r_{xy'} d_B(x) && \text{[by the AM-GM inequality]}\\
& = \sum_{x, y, y'} w^B_{k-1}(y)^2 r_{xy} r_{xy'} d_B(x)
&& \text{[by changing variables]}\\
& = \sum_{x, y} w^B_{k-1}(y)^2 r_{xy} d_R(x) d_B(x)\\
& \leq \left(\mfrac{n-1}{2}\right)^2 \sum_{x, y} w^B_{k-1}(y)^2 r_{xy} && \text{[since } d_R(x)+d_B(x)=n-1 \text{]}\\
& = \left(\mfrac{n-1}{2}\right)^2 \sum_{x} w^B_{k-1}(x)^2 d_R(x) =  \left(\mfrac{n-1}{2}\right)^2 \beta_{k-1}.
\end{align*}
A similar argument gives \[ \beta_k \leq \left(\mfrac{n-1}{2}\right)^2 \rho_{k-1}. \]
In particular, for $k$ even,
\begin{equation}\label{eq:even-rho/beta}
    \rho_k \leq\left(\mfrac{n-1}{2}\right)^{2k-2}\beta_1\leq \left(\mfrac{n-1}{2}\right)^{2k} 2 e(B) \quad\text{ and }\quad \beta_k\leq\left(\mfrac{n-1}{2}\right)^{2k-2}\rho_1 \leq \left(\mfrac{n-1}{2}\right)^{2k} 2 e(R),
\end{equation}
while, for $k$ odd,
\begin{equation}\label{eq:odd-rho/beta}
    \rho_k \leq\left(\mfrac{n-1}{2}\right)^{2k-2}\rho_1\leq \left(\mfrac{n-1}{2}\right)^{2k} 2 e(R) \quad\text{ and }\quad \beta_k\leq\left(\mfrac{n-1}{2}\right)^{2k-2}\beta_1 \leq \left(\mfrac{n-1}{2}\right)^{2k} 2 e(B).
\end{equation}

Now, for each $t \geq 1$,
\begin{equation}
\label{eq:t-walk-count}
W_t = \sum_x w^B_{t-1}(x)d_R(x) + \sum_x w^R_{t-1}(x) d_B(x).
\end{equation}

When $t$ is even, there is an obvious bijection between alternating $t$-walks starting with a red edge and alternating $t$-walks starting with a blue edge. Hence, for even $t$,
\begin{equation}
\label{eq:t-walk-count-even}
W_t = 2\sum_x w^B_{t-1}(x)d_R(x) \leq 2\sqrt{\sum_x w^B_{t-1}(x)^2 d_R(x)\sum_x d_R(x)} = 2\sqrt{\beta_{t-1}\beta_0},
\end{equation}
where the second step follows from the Cauchy-Schwarz inequality. Now, using \cref{eq:odd-rho/beta} and the AM-GM inequality, we obtain 
\begin{align*}
W_t & \leq 2 \sqrt{\beta_{t-1}\beta_0} 
\leq 2\left(\mfrac{n-1}{2}\right)^{t-1} \sqrt{4e(B)e(R)} = 4\left(\mfrac{n-1}{2}\right)^{t-1} \sqrt{e(B)e(R)}\\
& \leq 4\left(\mfrac{n-1}{2}\right)^{t-1} \mfrac{e(B) + e(R)}{2} = 2n\left(\mfrac{n-1}{2}\right)^{t}.
\end{align*}

When $t$ is odd, \eqref{eq:t-walk-count} along with the Cauchy-Schwarz inequality implies
\begin{equation}
\label{eq:t-walk-count-odd}
W_t \leq \sqrt{\sum_x w^B_{t-1}(x)^2 d_R(x)\sum_x d_R(x)} + \sqrt{\sum_x w^R_{t-1}(x)^2 d_B(x)\sum_x d_B(x)} = \sqrt{\beta_{t-1}\beta_0} + \sqrt{\rho_{t-1}\rho_0}.
\end{equation}
Similarly to above, using \cref{eq:even-rho/beta}, we obtain
\[ W_t \leq 2\left(\mfrac{n-1}{2}\right)^{t-1} \left(e(R) + e(B)\right) = 2n\left(\mfrac{n-1}{2}\right)^{t}. \]

For equality, it suffices to observe that if the red and blue degrees are both $\frac{n-1}{2}$ for each vertex, then for an alternating walk of length $t$ we have $n$ choices for the first vertex, $n-1$ for the second, and then $\frac{n-1}{2}$ for each subsequent vertex. 
\end{proof}

\begin{proof}[Proof of Corollary~\ref{co:CSpathBound}]
This follows immediately from Theorem~\ref{th:CSpathBound} and~(\ref{eq:walkpath}).
\end{proof}

Next we prove a stability version of Theorem~\ref{th:CSpathBound}, which states that any graph with close to the maximum number of alternating walks $W_t$ must have most vertices with similar red and blue degrees. 
This will be needed for the proof of \cref{th:2kCyclesBound}.
We first collect some definitions and lemmas.

\begin{defn}[$\epsilon$-balanced]
Let $G$ be a red-blue graph on $n$ vertices and $0 \leq \epsilon \leq 1$. We say $G$ is \emph{$\epsilon$-balanced} if
\[ \sum_{x} \left|d_R(x) - \mfrac{n-1}{2}\right| \leq \epsilon\mbinom{n}{2}, \]
and \emph{$\epsilon$-unbalanced} otherwise.    
\end{defn}

The case when $t$ is odd (the first and last edge of $W_t$ have the same colour) is harder due to lack of symmetry, and will require the following lemma.

\begin{lemma}
\label{cl:stability-aux}
Let $n$ be a positive integer, let $x \in (0,1]$ and let $x_1,\ldots,x_n \in [-1,1]$ such that $\frac{1}{n}\sum_{i} x_i = x$. Then $\frac{1}{n}\sum_{i} (1-x_i^2)(1+xx_i) \leq 1-x^4$.
\end{lemma}
\begin{proof}
Define $f(y) = (1-y^2)(1+yx)$ for all $y \in \mb{R}$.
Note that $f$ is cubic with negative leading term and its roots are $-\frac{1}{x}$, $-1$ and $1$. Routine analysis shows that $f$ has a local maximum at
$x_{\rm max}:=(\sqrt{1+3x^2}-1)/(3x)$
and that $0<x_{\rm max}<x$. So $f$ is strictly increasing on $[-1,x_{\rm max})$ and strictly decreasing on $(x_{\rm max},1]$.

Suppose that $x_1,\ldots,x_n$ maximise $\frac{1}{n}\sum_{i}f(x_i)$ subject to $\frac{1}{n}\sum_{i}x_i=x$. We must have $x_j \geq x > x_{\rm max}$ for some $j \in [n]$. If there were some $x_k<x_{\rm max}$, then increasing $x_k$ by $\min\{x_{\rm max}-x_k,x_j-x_{\rm max}\}$ and decreasing $x_j$ by the same amount would increase the sum $\frac{1}{n}\sum_{i}f(x_i)$, since both the corresponding terms would increase, contradicting our assumption. So it must be the case that $x_i \geq x_{\rm max} > 0$ for all $i \in [n]$. Now $f''(y)=-2-6x y$ which is negative for $y \geq 0$. So $f$ is concave on the interval $[0,1]$ in which all the $x_i$ lie and hence $x_i=x$ for all $i \in [n]$ 
by Jensen's inequality. The result follows. 
\end{proof}

\begin{theo}
\label{th:CSpathStability}
Let $0 \leq \epsilon \leq 1$, let $t \geq 2$ be an integer, and let $G$ be a red-blue $K_n$.
If $G$ is $\epsilon$-unbalanced, then $G$ has at most $(1 - \frac{\epsilon^4}{4}) 2n(\frac{n-1}{2})^{t}$ alternating walks of length $t$.
\end{theo}
\begin{proof}
We first consider the case when $t = 2$. Note that
\begin{align}
W_2 & = 2\sum_{x} d_R(x) d_B(x) = 2\sum_x d_R(x) (n - 1 - d_R(x)) \label{E:W2}\\
& = 2\sum_x \left( \left(\mfrac{n-1}{2}\right)^2 -  \left(\mfrac{n-1}{2} - d_R(x)\right)^2 \right)\nonumber\\
& \leq 2 n \left(\mfrac{n-1}{2}\right)^2 - \mfrac{2}{n} \left( \sum_{x} \left|\mfrac{n-1}2 - d_R(x)\right|\right)^2 && \text{[by Jensen's inequality]}\nonumber\\
& \leq 2n \left(\mfrac{n-1}{2}\right)^2 - \mfrac{2}{n}  \left(\epsilon \mbinom{n}{2}\right)^2 && \text{[by assumption]}\nonumber\\
& = (1 - \epsilon^2) 2 n \left(\mfrac{n-1}{2}\right)^2. \label{E:A2stab}
\end{align}

Suppose now that $t > 2$ is even. 
From \eqref{eq:t-walk-count-even} and \cref{eq:odd-rho/beta}, we obtain
\begin{align*}
    W_{t}^2 &\leq 4 \beta_{t-1}\beta_0 \leq 4\left(\mfrac{n-1}{2}\right)^{2t-4} \beta_{1}\beta_0.
\end{align*}
Recalling that $\beta_1 = \sum_{x}d_B(x)^2d_R(x)$ and $\beta_0 = \sum_x d_R(x) = 2e(R)$, we have
\begin{equation}
\label{eq:stability-evenpathbound1}
W_{t}^2 \leq 8e(R) \left(\mfrac{n-1}{2}\right)^{2t-4} \sum_{x}d_B(x)^2d_R(x).
\end{equation}
A symmetric argument gives 
\begin{equation}
\label{eq:stability-evenpathbound2}
W_{t}^2 \leq 4\rho_{t-1}\rho_0 \leq 4\left(\mfrac{n-1}{2}\right)^{2t-4} \rho_{1}\rho_0 \leq 8e(B) \left(\mfrac{n-1}{2}\right)^{2t-4} \sum_{x}d_R(x)^2d_B(x). \end{equation}
Taking a weighted average of \eqref{eq:stability-evenpathbound1} and \eqref{eq:stability-evenpathbound2} with weights $e(B)/(e(R) + e(B))$ and $e(R)/(e(R) + e(B))$ respectively, and using the AM-GM inequality, we obtain
\begin{align*}
W_{t}^2 \stacklEq{} 8\left(\mfrac{n-1}{2}\right)^{2t-4} \left(\mfrac{e(B)e(R)}{e(B) + e(R)}\right)  \left( \sum_{x} d_B(x)^2d_R(x) + \sum_{x} d_R(x)^2d_B(x) \right) \\
\stacklEq{} 8 \left(\mfrac{n-1}{2}\right)^{2t-4} \left(\mfrac{e(B) + e(R)}{4}\right)  \left( \sum_{x} d_B(x)d_R(x) \bigl(d_B(x) + d_R(x)\bigr) \right) \\
\stacklEq{} 2 \left(\mfrac{n-1}{2}\right)^{2t-4} \mbinom{n}{2}  \left( (n-1) \sum_{x} d_B(x)d_R(x) \right) =\left(\mfrac{n-1}{2}\right)^{2t-4} \mbinom{n}{2}(n-1) W_2\\
\stacklEq{(\ref{E:A2stab})} (1 - \epsilon^2) 4n^2 \left(\mfrac{n-1}{2}\right)^{2t} \leq \left( \left(1 - \mfrac{\epsilon^2}{2}\right) 2 n \left(\mfrac{n-1}{2}\right)^t \right)^2.
\end{align*}

We now consider the case when $t > 2$ is odd. Here the lack of symmetry forces the proof to be more involved. Suppose without loss of generality that $e(B) \leq e(R)$ and let $0 \leq \delta \leq \frac{1}{2}$ such that $e(R) = (\frac{1}{2}+\delta)\binom{n}{2}$.
We have
\begin{align*}
W_t^2 \stacklEq{(\ref{eq:t-walk-count-odd})} \left(\sqrt{\beta_{t-1} \beta_0} + \sqrt{\rho_{t-1} \rho_0}\right)^2 \stackrel{(\ref{eq:even-rho/beta})}{\leq} \left(\mfrac{n-1}{2}\right)^{2t-4}\left(\sqrt{\rho_1\beta_0}+\sqrt{\beta_1\rho_0}\right)^2\\
\stackEq{} \left(\mfrac{n-1}{2}\right)^{2t-4} \left(2\sqrt{\beta_1 \rho_1 \beta_0 \rho_0}  + \rho_1 \beta_0 + \beta_1 \rho_0\right)\\
\stacklEq{(\ref{eq:odd-rho/beta})} \left(\mfrac{n-1}{2}\right)^{2t-4} \left(8 \left(\mfrac{n-1}{2}\right)^2 e(R) e(B)  + 2\rho_1 e(R) + 2\beta_1 e(B)\right)\\
\stackEq{} \left(\mfrac{n-1}{2}\right)^{2t-4} \left(2n^2 \left(\mfrac{n-1}{2}\right)^4 (1-4\delta^2) + \mbinom{n}{2}\sum_{x}d_R(x)d_B(x) \bigl((1+2\delta) d_R(x) + (1-2\delta)d_B(x)\bigr) \right).
\end{align*}

Suppose first that $2\delta \leq \eps^2$. Then the previous equation implies
\begin{align*}
W_t^2 \stacklEq{} \left(\mfrac{n-1}{2}\right)^{2t-4} \left(2n^2 \left(\mfrac{n-1}{2}\right)^4(1-4\delta^2) + (1+2\delta)\mbinom{n}{2}\sum_{x}d_R(x)d_B(x) \left(d_R(x) + d_B(x)\right) \right)\\
\stackEq{(\ref{E:W2})}  \left(\mfrac{n-1}{2}\right)^{2t-4} \left(2n^2 \left(\mfrac{n-1}{2}\right)^4(1-4\delta^2) + (1+2\delta)\mbinom{n}{2}\frac{W_2}{2}(n-1) \right) 
\\
\stacklEq{(\ref{E:A2stab})} \left(\mfrac{n-1}{2}\right)^{2t-4} \left(2n^2 \left(\mfrac{n-1}{2}\right)^4 + 2(1+2\delta) (1 - \epsilon^2) n^2 \left(\mfrac{n-1}{2}\right)^4\right)\\
\stackEq{} 2n^2\left(\mfrac{n-1}{2}\right)^{2t}\bigl(1+ (1+2\delta)(1 - \epsilon^2)\bigr)\\
\stacklEq{} \left(1-\mfrac{\epsilon^4}{2}\right)4n^2\left(\mfrac{n-1}{2}\right)^{2t}
\quad\quad [\text{using }2\delta \leq \eps^2].
\end{align*}

Suppose instead that $2\delta \geq \epsilon^2$.
Define $\delta_x$ for all $x \in V$ by writing $d_R(x)=(\frac{1}{2}+\delta_x)(n-1)$ for each $x \in V$.
So $\sum_x 2\delta_x = 2\delta n$ and $2\delta_x \in [-1,1]$ for all $x \in V$.
Then
\begin{align*}
W_t^2 & \leq \left(\mfrac{n-1}{2}\right)^{2t-4} \left(2n^2 \left(\mfrac{n-1}{2}\right)^4 (1-4\delta^2) + \mbinom{n}{2}\sum_{x}d_R(x)d_B(x) \left((1+2\delta) d_R(x) + (1-2\delta)d_B(x)\right) \right) \\
& =\left(\mfrac{n-1}{2}\right)^{2t-4} 2n^2 \left(\mfrac{n-1}{2}\right)^4 \left( 1-4\delta^2 + \mfrac{1}{n}\sum_{x}(1-(2\delta_x)^2) \left(1+2\delta\cdot 2\delta_x\right) \right) \\
& \leq \left(\mfrac{n-1}{2}\right)^{2t} 2n^2 \left( 1-4\delta^2 + 1-(2\delta)^4 \right)  \quad\quad [\text{by Lemma~\ref{cl:stability-aux}}]\\
& \leq \left(1-\mfrac{\epsilon^4}{2}\right)4n^2\left(\mfrac{n-1}{2}\right)^{2t}
\quad\quad [\text{using }2\delta \geq \eps^2].
\end{align*}
In both cases, we have 
\begin{align*}
W_t^2 & \leq \left(1-\mfrac{\epsilon^4}{2}\right) 4 n^2 \left(\mfrac{n-1}{2}\right)^{2t} \leq \left( \left( 1 - \mfrac{\epsilon^4}{4} \right) 2 n \left(\mfrac{n-1}{2}\right)^t \right)^2.\qedhere
\end{align*}
\end{proof}

Note that \cref{th:CSpathStability} and~(\ref{eq:walkpath}) imply a stability version of \cref{co:CSpathBound}.

\begin{cor}\label{cor:CSpathBound}
Let $0 \leq \epsilon \leq 1$ and let $t \geq 2$ be an integer. Any red-blue $K_n$ with more than $(1 - \frac{\epsilon^4}{4}) n\left(\frac{n-1}{2}\right)^{t}$ alternating paths of length $t$ is $\varepsilon$-balanced.
\end{cor}

\section{Alternating even cycles}\label{sec:altcycles}

In this section, we prove Theorem~\ref{th:2kCyclesBound}, recalled below.
\EvenCyclesBound*
Theorem~\ref{th:2kCyclesBound} follows, in part, from the following stability result.

\begin{theo}
\label{th:2kCyclesStability}
Let $0<\frac{1}{n}\ll \varepsilon\ll \frac{1}{t}\leq \frac{1}{2}$.
Any red-blue $K_n$ containing more than $(1-\epsilon)\frac{n}{4t}(\frac{n-1}{2})^{4t-1}$ alternating $4t$-cycles
is $\delta$-bipartite where $\delta = (4\epsilon)^{1/32}$.
\end{theo}

We first derive bounds on the number of alternating paths starting from a generic vertex in an $\epsilon$-balanced graph.
\begin{lemma}\label{L:almostRegWalkCounts}
Let $i$ be a positive integer and $0 < \frac{1}{n} \ll \eps \ll \frac{1}{i}$. Suppose $G$ is an $\epsilon$-balanced red-blue $K_n$ and let $x_0 \in V$ such that 
\[ \left|d_R(x_0)-\mfrac{n-1}{2}\right| \leq 3\epsilon\left(\mfrac{n-1}{2}\right). \]
Then the following hold.
\begin{enumerate}
\item For each $y \in V$,
$ w_{i}^R(x_0,y) \leq (1+3(i-1)\epsilon)\left(\mfrac{n-1}{2}\right)^{i-1}. $\label{item:almostRegWalkCounts1}
\item $ 
(1-3i\epsilon)\left(\mfrac{n-1}{2}\right)^{i} \leq w^R_{i}(x_0) \leq \left(1+3i\epsilon\right)\left(\mfrac{n-1}{2}\right)^{i}.$\label{item:almostRegWalkCounts2}
\end{enumerate} 
\end{lemma}

\begin{proof}
We proceed by induction on $i$. When $i=1$, \cref{item:almostRegWalkCounts1} holds trivially and \cref{item:almostRegWalkCounts2} holds by our hypothesis on $d_R(x_0)$. Suppose \cref{item:almostRegWalkCounts1,item:almostRegWalkCounts2} hold for $i = k$ for some $k \geq 1$. 

Note that $w_{k+1}^R(x_0,y) \leq w_{k}^R(x_0)$ since each walk contributing to $w_{k}^R(x_0)$ can be extended to at most one walk contributing to $w_{k}^R(x_0,y)$. So \cref{item:almostRegWalkCounts1} holds with $i=k+1$ because \cref{item:almostRegWalkCounts2} holds with $i=k$.

Let $C=B$ if $k$ is odd and $C=R$ if $k$ is even. For $x \in V$, let $\delta_x$ be such that $d_C(x) = (1+\delta_x)\frac{n-1}{2}$. Now
\begin{equation}\label{E:walksFromAVert}
w^R_{k+1}(x_0) = \sum_{y}w^R_{k}(x_0,y)d_C(y) =\mfrac{n-1}{2}\left(w^R_{k}(x_0)+\sum_{y}w^R_{k}(x_0,y)\delta_y\right)
\end{equation}
where the second equality follows by substituting $d_C(x)=(1+\delta_x)\frac{n-1}{2}$ and simplifying. Now, using \cref{item:almostRegWalkCounts1} with $i=k$ and our hypothesis that $\sum_{y}|\delta_y| \leq \epsilon n$, we have 
\begin{equation}
\label{E:abswalksbound}
\left|\sum_{y}w^R_{k}(x_0,y)\delta_y\right| \leq \epsilon n(1+3(k-1)\epsilon)\left(\mfrac{n-1}{2}\right)^{k-1} < 3\epsilon\left(\mfrac{n-1}{2}\right)^{k}
\end{equation}
where in the last inequality we use the facts that $\epsilon \ll \frac{1}{k}$ and $n$ is large. Substituting \cref{item:almostRegWalkCounts2} with $i=k$ and \eqref{E:abswalksbound} into \eqref{E:walksFromAVert} establishes that \cref{item:almostRegWalkCounts2} holds for $i=k+1$. 
\end{proof}

We are almost ready to prove Theorem~\ref{th:2kCyclesStability}, which states that, in any large red-blue $K_n$ with almost the maximum number of alternating $4t$-cycles, one colour must be close to a balanced complete bipartite graph.
In view of Theorem~\ref{th:CSpathBound}, almost every alternating $(4t-1)$-walk is contained in an alternating $4t$-cycle, and therefore there is a vertex $x_0$ for which almost every such walk starting at $x_0$ is contained in such a cycle.
The set $Y$ of vertices $y$ for which there are many alternating $(4t-2)$-walks from $x_0$ to $y$ starting with a red edge must therefore send almost entirely blue edges to $N_R(x_0)$. Moreover, by Lemma~\ref{L:almostRegWalkCounts}, $Y$ contains almost half the vertices. 
Along with Lemma~\ref{L:dirac}, this implies the required structure.

In contrast, we believe that the bound in Theorem~\ref{th:2kCyclesBound} is not achievable for alternating $(4t+2)$-cycles. Indeed, note that a partitioned graph does not contain any alternating $(4t+2)$-cycles and hence every alternating $(4t+1)$-walk is `wasted'. On the other hand, it is straightforward to see that the uniformly random red-blue complete graph contains $\frac{1}{4t+2}(\frac{n}{2})^{4t+2} + O(n^{4t+1})$
alternating $(4t+2)$-cycles with high probability, and so precisely half of the alternating $(4t+1)$-walks are `wasted'. We suspect  that this is optimal.

\begin{proof}[Proof of Theorem~\ref{th:2kCyclesStability}]
Suppose that $G$ is a red-blue $K_n$ that contains more than $(1-\epsilon)\frac{n}{4t}\left(\frac{n-1}{2}\right)^{4t-1}$ alternating $4t$-cycles. For each alternating $4t$-cycle in $G$ there are $8t$ alternating $(4t-1)$-walks that use only edges of that cycle (each using all but one of the edges) and hence $G$ has more than $(1-\epsilon)2n\left(\frac{n-1}{2}\right)^{4t-1}$ such alternating $(4t-1)$-walks.
Then, by Theorem~\ref{th:CSpathStability}, $G$ is $\epsilon_1$-balanced,  
where $\epsilon_1\coloneqq(4\epsilon)^{1/4}$. Furthermore, by Theorem~\ref{th:CSpathBound}, $G$ contains at most $2n\left(\frac{n-1}{2}\right)^{4t-1}$ alternating $(4t-1)$-walks and, hence, less than $2\epsilon n\left(\frac{n-1}{2}\right)^{4t-1}$ alternating $(4t-1)$-walks are not contained in an alternating $4t$-cycle. We refer to these as \emph{wasted walks}. 

Since $G$ is $\epsilon_1$-balanced, less than half of the vertices $x \in V$ can satisfy $\left|d_R(x)-\frac{n-1}{2}\right| > 2\epsilon_1\left(\frac{n-1}{2}\right)$. Similarly, less than half of the vertices can be the start of more than $4\epsilon \left(\frac{n-1}{2}\right)^{4t-1}$ wasted walks.
So we can choose a vertex $x_0$ such that $\left|d_R(x_0)-\frac{n-1}{2}\right| \leq 2\epsilon_1\left(\frac{n-1}{2}\right)$ and at most $4\epsilon\left(\frac{n-1}{2}\right)^{4t-1}$ wasted walks start at $x_0$.

Set \[ Y \coloneqq \left\{y\in V: w^R_{4t-2}(x_0,y) > \epsilon_1\left(\mfrac{n-1}{2}\right)^{4t-3} \right\}. \]
We will show that the sets $Y$ and $N_R(x_0)$ witness the fact that $G$ is $\delta$-bipartite, via Lemma~\ref{L:dirac}. Our choice of $x_0$ ensures $|N_R(x_0)| \geq (\frac{1}{2}-\sqrt{\epsilon_1})n$ and, since $G$ is $\varepsilon_1$-balanced, fewer than $\sqrt{\varepsilon_1}n$ vertices $x\in V$ satisfy $\left|d_R(x)-\frac{n}{2}\right| > \sqrt{\varepsilon_1}n$. Thus, it suffices to show that $Y$ is not much smaller than half-sized and there are few red edges between $N_R(x_0)$ and $Y$.

By Lemma~\ref{L:almostRegWalkCounts}\cref{item:almostRegWalkCounts1}, for each $y \in Y$,
\[ w^R_{4t-2}(x_0,y) < (1+12t\epsilon_1)\left(\mfrac{n-1}{2}\right)^{4t-3}. \]
Together with Lemma~\ref{L:almostRegWalkCounts}\cref{item:almostRegWalkCounts2} and the definition of $Y$, this implies that
\begin{equation*}
(1-12t\epsilon_1)\left(\mfrac{n-1}{2}\right)^{4t-2} < w_{4t-2}^R(x_0) < |Y|(1+12t\epsilon_1)\left(\mfrac{n-1}{2}\right)^{4t-3}+(n-|Y|)\epsilon_1\left(\mfrac{n-1}{2}\right)^{4t-3}. \end{equation*}
It follows that
\[ |Y| > \mfrac{1 - (12t + 3)\epsilon_1}{1 + (12t-1)\epsilon_1} \left(\mfrac{n-1}{2}\right) > (1 - 26t\epsilon_1)\left(\mfrac{n-1}{2}\right) > \left(\mfrac{1}{2}-\sqrt{\epsilon_1}\right)n. \]

Now, let \[ E_R \coloneqq \{ yz \in E(R): y\in Y, z \in N_R(x_0) \}. \] Note that extending a $(4t-2)$-walk from $x_0$ to a vertex in $Y$ with an edge in $E_R$ gives a wasted walk. Hence, by the definition of $Y$, every edge in $E_R$ contributes more than $\epsilon_1 \left(\frac{n-1}{2}\right)^{4t-3}$ wasted walks starting at $x_0$. It follows that $|E_R| < \epsilon_1 \left(\frac{n-1}{2}\right)^{2}$, since otherwise there are more than $\epsilon_1^2 \left(\frac{n-1}{2}\right)^{4t-1} > 4\epsilon \left(\frac{n-1}{2}\right)^{4t-1}$ wasted walks starting at $x_0$, contradicting the choice of $x_0$. In particular, $|E_R| < \sqrt{\epsilon_1}n^2$.

It follows, by Lemma~\ref{L:dirac}, that $G$ is $\delta$-bipartite with $\delta = (\epsilon_1)^{1/8}=(4\epsilon)^{1/32}$.
\end{proof}

Theorem~\ref{th:2kCyclesBound} follows almost immediately from Theorem~\ref{th:2kCyclesStability} along with Lemma~\ref{L:deltato0}.

\begin{proof}[Proof of Theorem~\ref{th:2kCyclesBound}]

Note that each alternating $(2t-1)$-path can be completed into at most one alternating $2t$-cycle. On the other hand, exactly $2t$ alternating $(2t-1)$-paths can be obtained from an alternating $2t$-cycle by deleting one of its edges. By  Corollary~\ref{co:CSpathBound}, this implies that the number of alternating $2t$-cycles is at most
\[ \mfrac{1}{2t} n\left(\mfrac{n-1}{2}\right)^{2t-1} = \mfrac{1}{t}\left(\mfrac{n}{2}\right)^{2t} + O(n^{2t-1}). \]

Now choose constants $\eps,\delta,n_0$ such that the conclusion of Theorem~\ref{th:2kCyclesStability} with the given $t$ and chosen $\eps$ holds for $n \geq n_0$, and the conclusion of Lemma~\ref{L:deltato0} with $h \coloneqq 4t$ and the chosen $\delta$ holds for $n \geq n_0$. 
Let $G$ be a red-blue copy of $K_n$ with the maximum number of alternating $4t$-cycles.
Now $G$ must contain at least as many alternating $4t$-cycles as a $0$-bipartite red-blue $K_n$,  and hence Theorem~\ref{th:2kCyclesStability} implies that
$G$ is $\delta$-bipartite.
By \cref{prop:canonical}, we may apply \cref{L:deltato0} to obtain that $G$ is $(a,n-a)$-partitioned where $|a-\frac{n}{2}| \leq \frac{\delta}{2} n$.
Thus
 the number of alternating $4t$-cycles in $G$ is exactly $\frac{1}{2t}(a)_{2t}(n-a)_{2t}=\frac{1}{2t}\prod_{i=0}^{2t-1}(a-i)(n-a-i)$.
This is uniquely maximised when $\{a,n-a\}=\{\lfloor\frac{n}{2}\rfloor,\lceil\frac{n}{2}\rceil\}$.
\end{proof}

\section{Non-alternating 4-cycles}\label{sec:nonalt}

Recall that Problem~\ref{prob} is trivial for monochromatic $H$ (see Fact~\ref{fact:basic}\ref{itm:mono}). Thus, up to switching colours, the only non-alternating red-blue 4-cycles of interest are $RRBB$-cycles and $RRRB$-cycles.

Let $G$ be a red-blue $K_n$ with vertex set $[n]$. To closely approximate the number of 4-cycles of a particular colour pattern in $G$, it is sufficient to know only the red degree $d_R(x)$ of each vertex $x$ and the red \emph{codegree} $|N_R(x) \cap N_R(y)|$ for each  pair of vertices $x$ and $y$. We capture this information in a \emph{degree-codegree vector} $(\bm{d},\bm{z})$ of $R$ where $\bm{d}=(d_x)_{1 \leq x \leq n}$ with $d_x \coloneqq d_R(x)/n$ and $\bm{z}=(z_{xy})_{1 \leq x < y \leq n}$ with $z_{xy} \coloneqq |N_R(x) \cap N_R(y)|/n$. Such a vector must obey some obvious requirements. 
\begin{enumerate}[label=(P\arabic*)]
    \item\label{it:1}
$0 \leq d_x \leq 1$ for all $1 \leq x \leq n$.
    \item\label{it:2}
$\max\{0,d_x+d_y-1\} \leq z_{xy} \leq \min\{d_x,d_y\}$ for all $1 \leq x < y \leq n$.
    \item\label{it:3}
$\sum_{x < y}z_{xy}=\frac{n}{2}\sum_x d_x^2-\frac{1}{2}\sum_x d_x$.
\end{enumerate}
Indeed,~\ref{it:1} is obvious, the pigeonhole principle implies~\ref{it:2}, and~\ref{it:3} can be seen by observing that both $\sum_{x < y}|N_R(x) \cap N_R(y)|$ and $\sum_x \binom{d_R(x)}{2}$ count the number of monochromatic red 2-paths in $G$. To obtain our results for both $RRBB$-cycles and $RRRB$-cycles we first approximate the number of the appropriate cycles in $G$ by a function of the degree-codegree vector of $R$. We then show that the upper bounds we desire can be obtained as the maximum of this function over all real vectors $(\bm{d},\bm{z})$ obeying~\ref{it:1}--\ref{it:3}, irrespective of whether they are realisable by a red-blue $K_n$ or not. In fact, for $RRBB$-cycles we maximise over all vectors obeying just~\ref{it:1} and~\ref{it:2}, and for $RRRB$-cycles we maximise over all vectors obeying just~\ref{it:1} and~\ref{it:3}.

\subsection{\texorpdfstring{$RRBB$}{RRBB}-cycles}

In this subsection, we prove \cref{th:RRBBbound}, recalled below. 
\RRBBbound*
The core of the proof is to show that any graph maximising the number of $\RRBB$ must be $\delta$-bipartite for some small $\delta$.
\begin{lemma}
\label{lem:RRBBdeltabipartite}
Let $0 <\frac{1}{n}\ll \delta < 1$. Let $G$ be a red-blue $K_n$ with
\[\#(\RRBB,G)=\max(\RRBB,n).\] 
Then $G$ is $\delta$-bipartite.
\end{lemma}

We first assume this lemma holds and derive the theorem by explicitly calculating, using \cref{prop:count,L:deltato0}, which $\delta$ maximises the number of $\RRBB$.

\begin{proof}[Proof of Theorem~\ref{th:RRBBbound}]
We note that, by \cref{prop:canonical}, $\RRBB$ satisfies the hypotheses of Lemma~\ref{L:deltato0}, in particular it is a subgraph of a $(3,1)$-partitioned graph. 
Let $n_0,\delta$ be such that 
Lemma~\ref{L:deltato0} holds with $h=4$
and Lemma~\ref{lem:RRBBdeltabipartite} holds,
whenever $n \geq n_0$.
Let $G$ be a red-blue $K_n$ and suppose that $\#(\RRBB,G) = \max(\RRBB,n)$. By Lemma~\ref{lem:RRBBdeltabipartite}, $G$ is $\delta$-bipartite.
Then Lemma~\ref{L:deltato0} implies that $G$ is $(a,b)$-partitioned for some $|a-b| \leq \delta n$. Thus, writing $n=2m$ and $a=m+a_0$ (where $m$ and $a_0$ are not necessarily integers), we have $b=n-a=m-a_0$ and 
Proposition~\ref{prop:count} implies that
\begin{equation}\label{E:RRBBsplit}
\#(\RRBB,G) = 3\left(\mbinom{a}{3}b + a\mbinom{b}{3}\right)
= (m + a_0) (m - a_0) (m^2 - 3 m + a_0^2 + 2).
\end{equation}
Without loss of generality, we may suppose that $a_0 \geq 0$.
By differentiating this function of $a_0$, we see that the unique maximum for real $a_0 \geq 0$ is at $a^* \coloneqq \sqrt{(3m-2)/2} = \frac{1}{2}\sqrt{3n-4}$,
and the function is strictly increasing in $[0,a^*)$ and strictly decreasing in $(a^*,\infty)$.
Thus the middle expression in \cref{E:RRBBsplit} achieves its integer maximum only if $a \in \{\lfloor \frac{n}{2}+a^* \rfloor, \lceil \frac{n}{2}+a^* \rceil\}$.
\end{proof}

To prove \cref{lem:RRBBdeltabipartite}, we rely on Lemma~\ref{L:dirac}. Hence, it suffices to show that most red degrees are close to $\frac{n}{2}$ and that there exist two sets $X$ and $Y$ which are not much smaller than half-sized with few blue edges in between.

We will show this using a careful count of copies of $\RRBB$ where given vertices $x$ and $y$ both appear as bichromatic or monochromatic vertices. For this, we introduce the following functions which help us count these copies in terms only of red degrees and red codegrees.

Given $p,q,r\in \mb{R}$,
let \begin{gather*}
b(p,q,r) \coloneqq r(1-p-q+r),
\quad
m(p,q,r) \coloneqq\mfrac{1}{2}(p-r)^2+\mfrac{1}{2}(q-r)^2,\\
\text{and}\quad t(p,q,r) \coloneqq b(p,q,r)+m(p,q,r).
\end{gather*}
We say that the triple $(p,q,r)$ is \emph{graphical}
if $\max\{0,p+q-1\} \leq r \leq \min\{p,q\}$,
and note that if $(p,q,r)$ is graphical and $p,q\in[0,1]$,
then $r \in [0,1]$.

 Given a red-blue $K_n$ and distinct vertices $x$ and $y$, we write $\bic(x,y)$ for the number of copies of $\RRBB$ in which $x$ and $y$ are (nonadjacent) bichromatic vertices, and
$\mon(x,y)$ for the number of copies of $\RRBB$ in which $x$ and $y$ are (nonadjacent) monochromatic vertices.

\begin{lemma}
\label{L:RRBBcounts}
Let $G$ be a red-blue $K_n$ with red graph $R$ and blue graph $B$, and let $(\bm{d},\bm{z})$ be the degree-codegree vector of $R$.
Then, for all $1 \leq x < y \leq n$,
$(d_x,d_y,z_{xy})$ is graphical and
$$
\bic(x,y) = b(d_x, d_y, z_{xy})n^2 + O(n) \quad\text{and}\quad
\mon(x,y) = m(d_x, d_y, z_{xy})n^2 + O(n).
$$
\end{lemma}
\begin{proof}
Let $x$ and $y$ be distinct vertices of $G$. Let $b_{xy}$ equal 1 if $xy$ is blue and 0 if $xy$ is red and let $r_{xy}=1-b_{xy}$. 
The fact that every $(d_x,d_y,z_{xy})$ is graphical is equivalent to stating that~\ref{it:2} holds.
We have
\begin{align*}
\bic(x,y) &= |N_R(x) \cap N_R(y)| |N_B(x) \cap N_B(y)| \\
& =  |N_R(x) \cap N_R(y)| \left( n - d_R(x) - d_R(y) + |N_R(x) \cap N_R(y)| -b_{xy} \right) \\
& = b(d_x, d_y, z_{xy})n^2+O(n)
\end{align*}
and
\begin{align*}
\mon(x,y) &= \mbinom{w_2^R(x, y)}{2} + \mbinom{w_2^R(y, x)}{2} \\
& =  \mbinom{d_R(x) - |N_R(x) \cap N_R(y)| -r_{xy}}{2} + \mbinom{d_R(y) - |N_R(x) \cap N_R(y)| - r_{xy}}{2} \\
& = m(d_x, d_y, z_{xy})n^2 + O(n),
\end{align*}
as desired.
\end{proof}

Observe that, for any $p \in [0,1]$, $(p,p,p)$ and $(p,1-p,0)$ are graphical, and
$$
b(p,p,p)=t(p,p,p)=p(1-p),
\quad\text{while}\quad
b(p,1-p,0)=0\quad\text{and}\quad t(p,1-p,0)=\mfrac{1}{2}-p(1-p).
$$ 

We next prove a numerical optimisation result that shows, for any graphical $(p,q,r) \in [0,1]^3$, that $t(p,q,r)$ is bounded above by the linear function of $b(p,q,r)$ that passes through the two points given above.

\begin{lemma}\label{L:RRBBtradeoff}
For any graphical $(p,q,r) \in [0,1]^3$, we have 
\begin{equation}
\label{eq:RRBBtradeoff}
t(p,q,r) \leq \mfrac{1}{2} - p(1-p)-b(p,q,r)\left(\mfrac{1}{2p(1-p)}-2\right).
\end{equation}

\end{lemma}

\begin{proof}
Fix $p,q \in [0,1]$.
By subtracting the LHS of \eqref{eq:RRBBtradeoff} from its RHS and simplifying, we obtain 
\begin{equation*}
f_{p,q}(r) \coloneqq \mfrac{1}{2p(1-p)}\left(-r^2+r(3p+q-2p^2-1)+p(1-p)((1-p)^2-q^2)\right).
\end{equation*}
To prove the lemma, it suffices to show that $f_{p,q}(r) \geq 0$ for all $r \in [\max\{0,p+q-1\},\min\{p,q\}]$.
Since $f_{p,q}$ is a concave quadratic in $r$, under the specified constraints, it must achieve its minimum at some $r^* \in \{\max\{0, p+q-1\},\min\{p,q\}\}$. 
We consider four cases accordingly.
\begin{itemize}
\item If $r^* = 0 \geq p + q - 1$, then 
$f_{p,q}(r^*) = \frac{1}{2}((1-p)^2-q^2) \geq 0$
since $q \leq 1 - p$.

\item If $r^*=p+q-1 \geq 0$, then
$f_{p,q}(r^*) = \frac{1}{2}(p^2-(1-q)^2) \geq 0$
since
$p \geq 1 -q$. 

\item If $r^*=p \leq q$, then
$$
f_{p,q}(r^*) = \mfrac{1}{2(1-p)}(-(1-p)q^2+q-p(1-p+p^2))
$$
which is a concave quadratic in $q$ for $q \in [p,1]$ and so, under the specified constraints, must achieve its minimum at $q=p$ or at $q=1$. At the former it is $0$ and at the latter it is $\frac{1}{2}p^2 \geq 0$. 

\item If $r^*=q \leq p$, then
$$
f_{p,q}(r^*) = \mfrac{1}{2p}(-pq^2-(1-2p)q+p(1-p)^2)
$$
which is a concave quadratic in $q$ for $q \in [0,p]$ and so, under the specified constraints, must achieve its minimum at $q=0$ or at $q=p$. At the former it is $\frac{1}{2}(1-p)^2 \geq 0$ and at the latter it is $0$.
\end{itemize}
Thus for every $r$ such that $(p,q,r)$ is graphical, we have $f_{p,q}(r) \geq f_{p,q}(r^*) \geq 0$ and the assertion follows.
\end{proof}

Essentially, \cref{L:RRBBcounts,L:RRBBtradeoff} imply that if two vertices $x$ and $y$ appear in many common copies of $\RRBB$ as bichromatic vertices, then the total number of $\RRBB$ in which  $x$ and $y$ are nonadjacent vertices is not too large. Overall, this implies that if a vertex $x$ appears as a bichromatic vertex in many $\RRBB$ then it cannot appear in too many $\RRBB$ in total.

We now discuss our next lemma, making use of the notation defined in its statement. Let $G$ have the maximum number of $\RRBB$. 
We already know from \cref{lm:twinning} that each vertex is, in total, roughly in the same number of $\RRBB$. In our next lemma, we consider a fixed vertex $x_0$ (which fixes $p = d_{x_0}$) which is a bichromatic vertex in a substantial number of $\RRBB$,
and consider vertices $y$ via the pairs of values $(q,r) = (d_y,z_{x_0y})$. Using the above, we show that this implies that most vertices (corresponding to pairs in $A\setminus A_0$) are in roughly the same number of copies of $\RRBB$ containing $x_0$.
A careful analysis of the bound obtained in \cref{L:RRBBtradeoff} allows us to show that this can only happen if those vertices have roughly as many red as blue edges (\cref{L:RRBBsets}\cref{L:RRBBsets-q}) and are either almost antitwins of $x_0$ (if their corresponding pair is in $S$) or almost twins of $x_0$ (if their corresponding pair is in $T$).

Recall that to prove \cref{lem:RRBBdeltabipartite}, by \cref{L:dirac}, it suffices to show that almost all vertices have roughly as many red as blue edges (which is the case by the above) and find two large sets with few blue edges in between. The above partition will provide us with exactly that: the (almost) twins of $x_0$ have many red neighbours in $N_R(x_0)$ and so there are few blue edges between $N_R(x_0)$ and $T$.

\begin{lemma}\label{L:RRBBsets}
Let $0 < \frac1m \ll \eps\ll\eta \ll 1$.
Let $\frac{1}{10} \leq p \leq \frac{9}{10}$
and let $A$ be a collection of $m$
pairs $(q,r) \in [0,1]^2$
such that $(p,q,r)$ is graphical. Suppose that
$$
\sum_{(q,r) \in A}b(p,q,r) \geq \left(\mfrac{1}{8}-\eps\right)m
\quad\text{and}\quad
\sum_{(q,r) \in A}t(p,q,r) \geq \left(\mfrac{1}{4}-\eps\right)m.
$$
Then $|p-\frac{1}{2}| < \eta$, and there is a partition $A = A_0 \cup S \cup T$ such that
\begin{enumerate}[label=\rm(\roman*)]
    \item $|A_0| \leq \eta m$;\label{L:RRBBsets-A0}
    \item $|q-\frac{1}{2}|<\eta$ for every $(q,r) \in S \cup T$; \label{L:RRBBsets-q}
    \item $r \leq \eta$ for all $(q,r) \in S$; \label{L:RRBBsets-S}
    \item $r \geq \frac{1}{2}-\eta$ for all $(q,r) \in T$. \label{L:RRBBsets-T}
\end{enumerate}
\end{lemma}

We postpone the proof of \cref{L:RRBBsets} to the end of this subsection and first derive \cref{lem:RRBBdeltabipartite} with the above mentioned strategy.

\begin{proof}[Proof of \cref{lem:RRBBdeltabipartite}]
Choose constants $\eps,\eta$ such that $\frac1n \ll \eps \ll \eta \ll \delta$.
Let $G$ be a red-blue $K_n$ and suppose that $\#(\RRBB,G)=\max(\RRBB,n)$. 
Let $(\bm{d},\bm{z})$ be the degree-codegree vector of $R$.
Since  $\RRBB$ is a subgraph of a $(3,1)$-partitioned graph, 
a $0$-bipartite graph shows, via Proposition~\ref{prop:count}, that $\max(\RRBB,n) \geq \frac{1}{16}n^4 - O(n^3)$. 
So, by Lemma~\ref{lm:twinning}, each vertex is contained in at least $\frac{1}{4}n^3-O(n^2)$ copies of $\RRBB$.
Thus, by Lemma~\ref{L:RRBBcounts},
for all $x \in V$ we have
\begin{align*}
\sum_{y \in V\sm\{x\}}t(d_x,d_y,z_{xy}) &= \mfrac{1}{n^2}\sum_{y \in V\sm\{x\}}(\bic(x,y)+\mon(x,y))+ O(1) \geq \mfrac{n^3/4}{n^2} + O(1)
> \left(\mfrac{1}{4} - \eps\right)(n-1).
\end{align*}
Since each copy of $\RRBB$ has one pair of bichromatic vertices, we have
$$
\#(\RRBB,G) = \sum_{x < y}\bic(x,y) = \mfrac{1}{2}\sum_{x}\sum_{y \in V \sm \{x\}}\bic(x,y).
$$
By averaging, there is a vertex $x_0 \in V$ 
with $\sum_{y \in V\sm\{x_0\}}\bic(x_0,y) \geq \frac{2}{n} \cdot \#(\RRBB,G) \geq \frac{1}{8}n^3+O(n^2)$. Thus
$$
\sum_{y \in V\sm\{x_0\}}b(d_{x_0},d_y,z_{x_0y}) = \mfrac{1}{n^2}\sum_{y \in V\sm\{x_0\}}\bic(x_0,y) + O(1)
> \left(\mfrac{1}{8}-\eps\right)(n-1).
$$
Furthermore, $\sum_{y \in V\sm\{x\}}\bic(x,y) \leq d_R(x)d_B(x)n \leq d_x(1-d_x)n^3$ for any vertex $x$, so 
\[ \mfrac{1}{10} \leq d_{x_0} \leq \mfrac{9}{10}. \] 
Lemma~\ref{L:RRBBsets} applied with
$A \coloneqq \{(d_y,z_{x_0y}): y \in V \sm \{x_0\}\}$ implies that $|d_R(x_0)-\frac n2| \leq \eta n$ and
there is a partition $V = \{x_0\} \cup V_0 \cup S \cup T$ such that $|V_0| \leq \eta n$, $|d_R(y)-\frac n2| < \eta n$ for all $y \in S \cup T$, every $y \in S$ has $|N_R(x_0) \cap N_R(y)| \leq \eta n$, and every $y \in T$ has $|N_R(x_0) \cap N_R(y)| \geq \frac n2-\eta n$.

We will show that the sets $N_R(x_0)$ and $T$ witness the fact that $G$ is $\delta$-bipartite, via Lemma~\ref{L:dirac}.
For this, we need to show that $T$ is not much smaller than half-sized and there are few blue edges between $N_R(x_0)$ and $T$.
Using the fact that every $v \notin V_0$ satisfies $|d_B(v) - \frac{n}{2}| \leq 2\eta n$, we have
\begin{align*}
|S| \left(\mfrac{1}{2} - 2\eta \right) n 
&\leq \sum_{y \in S}d_B(y) 
= \sum_{v \in N_R(x_0)}|N_B(v)\cap S| + \sum_{y \in S}|N_B(y) \sm N_R(x_0)|\\
&\leq |V_0|n + \sum_{v \in N_R(x_0) \sm V_0}d_B(v) + \sum_{y \in S}(d_B(y) - d_R(x_0) + |N_R(y) \cap N_R(x_0)| + 1)\\
&\leq |V_0|n + |N_R(x_0) \sm V_0|\left(\mfrac{n}{2}+\eta n\right) + n\cdot (3\eta n+1)\\
&\leq \eta n^2 + \left(\mfrac{n}{2}+\eta n\right)^2 + 3\eta n^2 + n
\leq \left(\mfrac{1}{4} + 6 \eta\right)n^2
\end{align*}
which implies
$|S| \leq ( \frac{1}{2} + 15\eta) n$.
So $|T| = n-1 - |S| - |V_0| \geq n-1-\left( \frac{1}{2} + 15\eta\right) n-\eta n \geq (\frac{1}{2}-17\eta)n$.
Finally, we have
$
e_B(N_R(x_0),T) \leq \sum_{y \in T}|N_R(x_0) \cap N_B(y)| \leq 2\eta n|T| \leq 2\eta n^2.
$
Lemma~\ref{L:dirac} implies that $G$ is $\delta$-bipartite.
\end{proof}

We conclude the proof of \cref{th:RRBBbound} by deriving the remaining lemma.

\begin{proof}[Proof of \cref{L:RRBBsets}]
Lemma~\ref{L:RRBBtradeoff} implies that
\begin{align}
    \mfrac{1}{m}\sum_{(q,r) \in A}t(p,q,r)& \leq \mfrac{1}{2} - p(1-p)-\left(\mfrac{1}{8}-\varepsilon\right)\left(\mfrac{1}{2p(1-p)}-2\right)\nonumber\\
    &\leq \mfrac{3}{4}-p(1-p)-\mfrac{1}{16p(1-p)} + \epsilon\left( \mfrac{1}{2p(1-p)}-2 \right) \nonumber\\
    &\leq \mfrac{3}{4}-p(1-p)-\mfrac{1}{16p(1-p)}+\mfrac{32 \varepsilon}{9}\label{eq:RRBBStability}
\end{align}
where to obtain the last inequality, we use the fact that $\frac{1}{2p(1-p)}-2$ is maximised, under our constraints, when $p \in \{\frac{1}{10}, \frac{9}{10}\}$. 
Therefore
\[\mfrac{1}{4} - \eps  \leq \mfrac{3}{4}-p(1-p)-\mfrac{1}{16p(1-p)}+\mfrac{32 \varepsilon}{9}
\]
and hence 
we have
\begin{equation*}
        p(1-p)+\mfrac{1}{16p(1-p)} \leq \mfrac{1}{2} + 5\epsilon.
\end{equation*}
Solving the quadratic in $w$ obtained from rearranging $w+\frac{1}{16w} \leq a$, we see that the latter is equivalent to $\frac{a}{2}-\frac{1}{4}\sqrt{4a^2-1} \leq w \leq \frac{a}{2}+\frac{1}{4}\sqrt{4a^2-1}$. Thus
\begin{equation}\label{E:RRBBp(1-p)}
    \left|\mfrac{1}{4}- p(1-p)\right| \leq \mfrac{5\varepsilon}{2}+\mfrac{1}{4}\sqrt{4\left(\tfrac{1}{2}+5\varepsilon\right)^2-1}\leq 2\sqrt{\varepsilon}.
\end{equation}
Writing $p=\frac{1}{2}+p_0$, this implies that $p_0^2 \leq 2\sqrt{\eps}$, so
\begin{equation}\label{E:RRBBp}
    \left|p-\mfrac{1}{2}\right| =|p_0|<\varepsilon^{1/5}<\eta,
\end{equation}
as desired.

Thus Lemma~\ref{L:RRBBtradeoff} implies, using \cref{E:RRBBp(1-p)} and $\frac{1}{2p(1-p)}\geq 2$, that for each $(q,r) \in A$,
\[t(p, q, r) \leq \mfrac{1}{2} - \left(\mfrac{1}{4}-2\sqrt{\varepsilon}\right) = \mfrac{1}{4} + 2\sqrt\eps.\]
Let $$
A_0 \coloneqq \left\{(q,r) \in A: t(p,q,r) < \mfrac{1}{4}-\eps^{1/3}\right\}.
$$
We have
$$
\left(\mfrac{1}{4}-\eps\right)m \leq \sum_{(q,r) \in A}t(p,q,r) \leq \left(\mfrac{1}{4}+2\sqrt\eps\right)(m-|A_0|) + \left(\mfrac{1}{4}-\eps^{1/3}\right)|A_0|
$$
and hence \[|A_0| \leq \mfrac{3\sqrt\eps m}{2\sqrt\eps+\eps^{1/3}}<\eta m,\] as desired for \cref{L:RRBBsets-A0}.

Let $(q,r) \in A\setminus A_0$.
Using \cref{E:RRBBp}, we have 
\begin{align*}
8t(p,q,r)& =8r(1-p-q+r) + 4(p-r)^2+4(q-r)^2\\
& \leq 8r\left(\mfrac{1}{2} + \varepsilon^{1/5} - q + r\right) + 4\left(\mfrac{1}{2} + \varepsilon^{1/5}-r\right)^2+4(q-r)^2\\
& = (4r-2q)^2 + 1 + 4(\varepsilon^{1/5}+\varepsilon^{2/5})\\
& \leq (4r-2q)^2 + 1 + 5\varepsilon^{1/5}.
\end{align*}
Since $(q,r) \notin A_0$, we have $8t(p,q,r) \geq 2 - 8\eps^{1/3}$ and hence
$(4r - 2q)^2 \geq 1 - 6\varepsilon^{1/5}$.
Denote $\delta \coloneqq \sqrt{1-6\varepsilon^{1/5}}$. This implies
\begin{equation}
\label{E:RRBBtradeoffstability}
4r \leq 2q - \delta \quad\text{ or }\quad 4r \geq 2q + \delta.
\end{equation}
Note that $1 - 4\varepsilon^{1/5} \leq \delta \leq 1 - 3\varepsilon^{1/5}$ and recall \cref{E:RRBBp}.
When $2q < \delta \leq 1-3\varepsilon^{1/5}<2p$, we have $2q - \delta<0$ and $2q + \delta>4q=4\min\{p,q\}$, so \cref{E:RRBBtradeoffstability} contradicts the fact that $0\leq r\leq \min\{p,q\}$. 
Similarly, when $2q > 4\max\{1-p,p\}-\delta\geq 4\max\{1-p,p\}-(1-3\varepsilon^{1/5})>2p$, we have $2q-\delta<4(p+q-1)$ and $2q+\delta>4p=4\min\{p,q\}$, so \cref{E:RRBBtradeoffstability} contradicts the fact that $p+q-1\leq r\leq \min\{p,q\}$.
Therefore, we have \[1-4\varepsilon^{1/5}\leq \delta\leq 2q\leq 4\max\{1-p,p\}-\delta\leq 2+4\varepsilon^{1/5}-(1-4\varepsilon^{1/5})= 1+8\varepsilon^{1/5},\]
proving \cref{L:RRBBsets-q}.
Along with \eqref{E:RRBBtradeoffstability}, this implies that either
\[ 4r  \leq 2q  - \delta \leq 1 + 8\varepsilon^{1/5} - (1 - 4\varepsilon^{1/5}) \leq 4\eta,\]
in which case we add $(q,r)$ to $S$,
or 
\[ 4r  \geq 2q  + \delta \geq 1 - 4\varepsilon^{1/5} + (1 - 4\varepsilon^{1/5})\geq 2 - 4\eta,\]
in which case we add $(q,r)$ to $T$.
In particular, this gives a partition of $A\setminus A_0$ such that \cref{L:RRBBsets-S,L:RRBBsets-T} are satisfied and this completes the proof of the lemma.
\end{proof}

\subsection{\texorpdfstring{$RRRB$}{RRRB}-cycles}

In this section, we prove  Theorems~\ref{th:RRRBbound},~\ref{th:rand} and~\ref{th:RRRBprofile}. We recall the statements of the first and third of these statements, which concern $\RRRB$. We derive Theorem~\ref{th:rand} from Theorem~\ref{th:RRRBprofile} at the end of the section.
\RRRBbound*
\RRRBprofile*

As discussed at the start of this section, we prove Theorems~\ref{th:RRRBbound} and~\ref{th:RRRBprofile} by relaxing the problem of determining $\max(\RRRB,n)$ to an optimisation problem whose feasible set consists of pairs $(\bm{d},\bm{z})$ of $n$- and $\binom{n}{2}$-tuples, which includes all those $(\bm{d},\bm{z})$ which are the degree-codegree vectors of a graph $R$ on $n$ vertices. 
We first define the feasible set of vectors we will consider, which are pairs $(\bm{d},\bm{z})$ satisfying~\ref{it:1} and~\ref{it:3}, and the function we will maximise. We then motivate these definitions in Lemma~\ref{lm:sequences}.

\begin{defn}\label{D:Ssigma}
Given $n \in \mathbb{N}$ and $\sigma \in [0,1]$,
let $S(\sigma)$ be the set of $(\bm{d},\bm{z})$ with $\bm{d}=(d_i)_{1 \leq i \leq n}$ and $\bm{z}=(z_{ij})_{1 \leq i< j\leq n}$ with every $d_i \in [0,1]$ and $z_{ij} \in \mathbb{R}$ such that
$$
\mfrac{1}{n}\sum_{i}d_i = \sigma  \quad\text{and}\quad
\sum_{i<j}z_{ij}=\mfrac{n}{2}(\tau n - \sigma)
\quad\text{where}\quad
\tau \coloneqq \mfrac{1}{n}\sum_{i}d_i^2.
$$
Let $S \coloneqq \bigcup_{\sigma \in [0,1]}S(\sigma)$.
Let $f(p, q,r)\coloneqq r(p+q-2r)$ and
define $f: S \to \mathbb{R}$ by
\[f(\bm{d},\bm{z})\coloneqq \mfrac{1}{n^2}\sum_{i<j}f(d_i,d_j,z_{ij}). \]
\end{defn}

We suppress the dependence on $n$ in the notation $S(\sigma)$; it will be clear from the context. Note that pairs in $S$ do not always correspond to degree-codegree sequences, since for example values $z_{ij}$ can be unbounded and/or negative.
Observe that
$(\sigma n)^2 = (\sum_i d_i)^2 = n \cdot \sum_i d_i^2 - \sum_{i<j}(d_i-d_j)^2 = n \cdot \tau n - \sum_{i<j}(d_i-d_j)^2$ so
\begin{equation}\label{eq:tausigma}
\tau = \sigma^2 + \mfrac{1}{n^2}\sum_{i<j}(d_i-d_j)^2 \geq \sigma^2.
\end{equation}

The first lemma states that the pair $(\bm{d},\bm{z})$ consisting of degrees and codegrees in the red graph of a red-blue $K_n$ lies in $S$, and $f(\bm{d},\bm{z})$ counts the proportion of copies of $\RRRB$ in the red-blue $K_n$.

\begin{lemma}\label{lm:sequences}
Let $G$ be a red-blue $K_n$ with red graph $R$ and let $(\bm{d},\bm{z})$ be the degree-codegree vector of $R$.
Then
$$
(\bm{d},\bm{z}) \in S
\quad\text{and}\quad
\#(\RRRB,G) = \mfrac{n^4}{2} f(\bm{d},\bm{z})+O(n^3).
$$
\end{lemma}

\begin{proof}
Let $\sigma\coloneqq \frac{1}{n}\sum_xd_x$.
We have $0 \leq d_x \leq \tfrac{n-1}{n}$ for all $x \in [n]$ (so \ref{it:1} holds) and
\[
 \sum_{x<y} z_{xy} = \mfrac{1}{n}\sum_{x<y} |N_R(x) \cap N_R(y)| = \mfrac{1}{n}\sum_{v}\mbinom{|N_R(v)|}{2} 
 = \mfrac{n}{2}\left(\sum_{v} d_v^2 - \mfrac{1}{n}\sum_{v}d_v\right) = \mfrac{n}{2}(\tau n - \sigma)
\]
(this is~\ref{it:3}).
So $(\bm{d},\bm{z}) \in S$.
Let $x$ and $y$ be distinct vertices of $G$. Let $m_2^R(x,y)$ be the number of monochromatic red 2-paths between $x$ and $y$, let $T(x,y)$ be the number of copies of $\RRRB$ in which $x$ and $y$ are non-adjacent vertices, and let $r_{xy}$ equal 0 if $xy$ is blue and 1 if $xy$ is red. Then
\begin{align*}
T(x, y) & = m_2^R(x, y)\left(w_2^R(x, y) + w_2^R(y, x)\right)\\
& = \left|N_R(x) \cap N_R(y)\right| \bigl( \left( d_R(x) - \left|N_R(x) \cap N_R(y)\right| -r_{xy}\right) + \left( d_R(y) - \left|N_R(x) \cap N_R(y)\right| -r_{xy}\right) \bigr)\\
& =  z_{xy} ( d_x + d_y - 2 z_{xy})n^2+O(n).
\end{align*}
Since each copy of $\RRRB$ has two pairs of non-adjacent vertices, the number of copies of $\RRRB$ in $G$ is 
\[ \#(\RRRB,G) =
\mfrac{1}{2} \sum_{x<y} T(x, y) = \mfrac{n^2}{2} \sum_{x<y}  z_{xy} ( d_x + d_y - 2z_{xy})+ O(n^3)
= \mfrac{n^4}{2} f(\bm{d},\bm{z}) + O(n^3).\qedhere\]
\end{proof}

The remainder of our proofs of Theorems~\ref{th:RRRBbound} and~\ref{th:RRRBprofile} consists only of analysing the values of $f(\bm{d},\bm{z})$ over $(\bm{d},\bm{z}) \in S(\sigma)$. The next lemma records some straightforward analytic properties of two single variable functions which will appear when evaluating $f$.

\begin{lemma}\label{L:BRRRnum}
Let $\sigma \in [0,1]$.
\begin{enumerate}
\item Suppose $\sigma \geq \frac{1}{2}$. Then\label{item:BRRRnum1}
$$
\sigma^3(1-\sigma) \leq
  \mfrac{27}{256} - \mfrac{1}{4}\left(\mfrac{3}{4}-\sigma\right)^2 \leq \mfrac{27}{256}.
$$
\item Let $\tau \in [\sigma^2,\infty)$.\label{item:BRRRnum2}
Then 
$$
g_{\sigma}(\tau) \coloneqq -\mfrac{1}{8}\left(8\tau^2-(8\sigma+1)\tau+\sigma^2\right)
$$
is uniquely maximised at some $\tau^*$ where
\begin{align*}
    \tau^* &= \mfrac{\sigma}{2}+\mfrac{1}{16}
    \quad\text{and}\quad g_{\sigma}(\tau^*) 
< \mfrac{23}{256} &&\text{if}\quad
\sigma < \tfrac{1}{4}(1+\sqrt{2}),\\
\tau^* &= \sigma^2
\quad\text{and}\quad g_{\sigma}(\tau^*)=\sigma^3(1-\sigma) &&\text{if}\quad 
\sigma \geq \tfrac{1}{4}(1+\sqrt{2}).
\end{align*}
Moreover, in the second case, we have $g_{\sigma}(\tau) \leq \sigma^3(1-\sigma) - \frac{1}{2}(\tau-\sigma^2)^2$ for all $\tau \in [\sigma^2,\infty)$.
\end{enumerate}
\end{lemma}

\begin{proof}
We first prove~\cref{item:BRRRnum2}.
Let
$
\lL \coloneqq 2\sigma^2 - \tfrac{1}{8} - \sigma
$.
We have $\lL \geq 0$ if and only if $\sigma \geq \frac{1}{4}(1+\sqrt{2})$.
The derivative of $g_{\sigma}$ is $\frac{d}{d\tau} g_{\sigma}(\tau) = \frac{1}{8}+\sigma - 2\tau$
which equals $0$ when $\tau = \tau_0 \coloneqq \frac{\sigma}{2}+\frac{1}{16}$.
Since $g_{\sigma}(\tau)$ is a quadratic whose leading coefficient is negative, it attains its maximum at some unique $\tau^*$, where $\tau^* =\frac{\sigma}{2}+\frac{1}{16}$ if $\sigma^2 \leq \tau_0$ (i.e.\ if~$\lL \leq 0$) and at $\tau^* =\sigma^2$, with $g_{\sigma}(\tau^*)=\sigma^3(1-\sigma)$ otherwise (i.e.\ if~$\lL \geq 0$).
All but the `moreover' part now follows from routine calculation. 

The `moreover' part is clear if $\tau=\sigma^2$.
So suppose $\tau>\sigma^2$.
Let $\tau_1 \coloneqq \frac{1}{2}(\tau+\sigma^2)$ be the average of $\tau$ and $\sigma^2$, so $\sigma^2 < \tau_1 < \tau$.
The mean value theorem implies there is $\tau'$ with $\tau_1 < \tau' < \tau$ with
\begin{align*}
g_{\sigma}(\tau) - g_{\sigma}(\tau_1)
&= (\tau-\tau_1) \cdot \mfrac{d g_{\sigma}}{d\tau}(\tau') = (\tau-\tau_1)\left(\mfrac{1}{8}+\sigma - 2\tau'\right)\\
&= -2(\tau-\tau_1)(\tau'-\sigma^2) -\lL(\tau-\tau_1) \leq -2(\tau-\tau_1)(\tau_1-\sigma^2)
=-\mfrac{1}{2}(\tau-\sigma^2)^2.
\end{align*}
But $g_{\sigma}(\tau_1) < g_{\sigma}(\sigma^2)=\sigma^3(1-\sigma)$, giving the required.

For~\cref{item:BRRRnum1}, we use a similar mean value theorem argument. Let $h(\sigma) \coloneqq \sigma^3(1-\sigma)$.
So $\frac{d}{d\sigma}h(\sigma) = 4(\frac{3}{4}-\sigma)\sigma^2$.
We see that $h(\sigma)$ attains its unique maximum at $\sigma=\frac{3}{4}$.
Suppose $\sigma \in [\frac{1}{2},1]$.
We will suppose that $\sigma \neq \frac{3}{4}$ since otherwise the inequality is trivial.
Let $\sigma_1 \coloneqq \frac{1}{2}(\frac{3}{4}+\sigma)$. 
Then $\sigma_1 \neq \sigma$ and $(\frac{3}{4}-\sigma_1)(\sigma_1-\sigma) = \frac{1}{4}(\frac{3}{4}-\sigma)^2$
and $(\frac{3}{4}-\sigma)(\sigma_1-\sigma) = \frac{1}{2}(\frac{3}{4}-\sigma)^2$.
There is $\sigma'$ strictly between $\sigma$ and $\sigma_1$ (which implies $\sigma'\geq \frac{1}{2}$) with
\begin{align*}
    h\left(\mfrac{3}{4}\right) - h(\sigma)
    &\geq h(\sigma_1) - h(\sigma) 
    = 4\left(\mfrac{3}{4}-\sigma'\right)(\sigma_1-\sigma)(\sigma')^2 \\ 
    &\geq 4\left(\mfrac{3}{4}-\max\{\sigma,\sigma_1\}\right)(\sigma_1-\sigma)(\sigma')^2 \geq \mfrac{1}{4}\left(\mfrac{3}{4}-\sigma\right)^2,
\end{align*}
as required.
\end{proof}

Let $F(\sigma) \coloneqq \sup_{(\bm{d},\bm{z}) \in S(\sigma)}f(\bm{d},\bm{z})$
and define $\bm{t} = (t_{ij})_{1 \leq i<j \leq n}$ via $t_{ij} \coloneqq d_i+d_j-4z_{ij}$ 
and let $\ell \coloneqq \binom{n}{2}^{-1}\sum_{i<j}t_{ij}$.
So
\begin{equation}\label{eq:l}
\ell = \mfrac{2n}{n-1}\left(\sigma-2\tau+\mfrac{\sigma}{n}\right) = 2\sigma - 4\tau + O\left(\mfrac1n\right). 
\end{equation}
For all $\eps \geq 0$, let $S_\eps(\sigma)$ be the subset of $S(\sigma)$ with $\sum_{i<j}|t_{ij}-\ell| \leq \eps n^2$.
The next lemma shows that if $f(\bm{d},\bm{z})$ is close to being maximal, then almost every entry of $\bm{t}$ is close to $\ell$.

\begin{lemma}\label{lm:iteration}
Let $\delta \geq 0$.
Given $(\bm{d},\bm{z}) \in S(\sigma)$ we have $f(\bm{d},\bm{z}) \geq F(\sigma) - \delta$ only if $(\bm{d},\bm{z}) \in S_{20\sqrt{\delta}}(\sigma)$.
In particular, $F(\sigma)$ is attained, and $f(\bm{d},\bm{z}) = F(\sigma)$ only if $(\bm{d},\bm{z}) \in S_0(\sigma)$.
\end{lemma}

\begin{proof}
For any $\eps \in \mathbb{R}$, we have
$$
f(d_i,d_j,z_{ij}+\eps)-f(d_i,d_j,z_{ij})=\eps(d_i+d_j-4z_{ij})-2\eps^2
=\eps t_{ij} - 2\eps^2.
$$
Thus for any pairs $ij,i'j'$,
$$
f(d_i,d_j,z_{ij}+\eps)-f(d_i,d_j,z_{ij})
+ f(d_{i'},d_{j'},z_{i'j'}-\eps)-f(d_{i'},d_{j'},z_{i'j'})
 = \eps (t_{ij}-t_{i'j'} - 4\eps).
$$
Suppose there are $ij,i'j'$ with $t_{ij}\geq t_{i'j'}+\eps$.
Then we can replace $z_{ij}$ with $z_{ij}-\frac{\eps}{5}$ and $z_{i'j'}$ with $z_{i'j'}+\frac{\eps}{5}$ to obtain a new element of $(\bm{d},\bm{z}')$ of $S(\sigma)$
with
\begin{equation}\label{eq:fchange}
f(\bm{d},\bm{z}') - f(\bm{d},\bm{z}) 
=  \mfrac{\eps}{5n^2}\left(t_{ij}-t_{i'j'}-\mfrac{4\eps}{5}\right)
\geq \mfrac{\eps}{5n^2}\left(\eps - \mfrac{4\eps}{5}\right) = \mfrac{\eps^2}{25n^2}.
\end{equation}
This implies the required statement for $\delta=0$: 
if $(\bm{d},\bm{z}) \notin S_0(\sigma)$, then $f(\bm{d},\bm{z})$ can be increased. The subset $S_0(\sigma)$ is closed and bounded and hence a compact subset of $[0,1]^n \times \mathbb{R}^{\binom{n}{2}}$. Thus the continuous function $f$ attains its maximum.

Thus we may suppose that $\delta>0$. We assume that $f(\bm{d},\bm{z})\geq F(\sigma)-\delta$ and we will prove that $(\bm{d},\bm{z})\in S_{20\sqrt{\delta}}(\sigma)$.
We will define an iterative process which will modify $\bm{z}$ by making its entries more equal, leaving $\bm{d}$ unchanged, and such that $f(\bm{d},\bm{z})$ increases.
Since this increase is bounded above, we will be able to conclude that the entries in $\bm{z}$ must not be too far apart.
To initialise, write $\bm{z}^0 \coloneqq \bm{z}$ and $\bm{t}^0 = (t^0_{ij})_{ij}$ where $t^0_{ij} \coloneqq d_i+d_j-4z^0_{ij}$, and $\Sigma^0 \coloneqq \sum_{i<j}|t_{ij}-\ell|$.
Let $\gamma \coloneqq 5\sqrt{\delta}$.

\medskip
For $k=1,2,3,\ldots$, we do the following.

If there are $ij$ and $i'j'$ in $\binom{[n]}{2}$ such that $t_{ij}^{k-1} \geq \ell+\gamma$ and $t_{i'j'}^{k-1} \leq \ell-\gamma$
we set
$$
z^k_{ij} \coloneqq z^{k-1}_{ij}+\mfrac{\gamma}{5}
\quad\text{and}\quad
z^k_{i'j'} \coloneqq z^{k-1}_{i'j'}-\mfrac{\gamma}{5}
$$
and all other coordinates of $\bm{z}^k$ are the same as in $\bm{z}^{k-1}$.
Define $\bm{t}^k=(t^k_{i''j''})_{i''j''}$ where $t^k_{i''j''} \coloneqq d_{i''}+d_{j''}-4z^k_{i''j''}$.
Let $\Sigma^{k} \coloneqq \sum_{i''<j''} |t^k_{i''j''}-\ell|.$

If there are no such $ij$ and $i'j'$, the process terminates and $k-1$ is the final step of the iteration.

\medskip
 
We claim that, for all steps $k \geq 1$,
\begin{enumerate}
    \item $(\bm{d},\bm{z}^k) \in S(\sigma)$;\label{item:iteration1}
    \item
    $\Sigma^k = \Sigma^{k-1} - \frac{8\gamma}{5}$;\label{item:iteration2}
    \item $f(\bm{d},\bm{z}^k) \geq f(\bm{d},\bm{z}^{k-1}) + \frac{\gamma^2}{25n^2}$.\label{item:iteration3}
\end{enumerate}

Part~\cref{item:iteration1} is immediate since $\sum_{i<j}z^k_{ij} = \sum_{i<j}z^{k-1}_{ij}$, and part~\cref{item:iteration2} follows from our choices of $ij$ and $i'j'$ at step $k$ of the process and the definition of $\bm{z}^k$. 
We have $t^{k-1}_{ij}-t^{k-1}_{i'j'}=(t^{k-1}_{ij}-\ell)+(\ell-t^{k-1}_{i'j'}) > \gamma$, so~\cref{item:iteration3} follows from~(\ref{eq:fchange}).
Thus~\cref{item:iteration1,item:iteration2,item:iteration3} do indeed hold for all steps $k \geq 1$.

The process is finite by~\cref{item:iteration2}.
Thus it terminates at $k=r$ for some integer $r \geq 0$. We have 
$$
0 \leq \Sigma^{r} = \Sigma^0 -\frac{8r\gamma}{5}.
$$
Since the process terminated at step $r$, either $t^r_{ij}>\ell-\gamma$ for all $ij \in \binom{[n]}{2}$ or $t^r_{ij}<\ell+\gamma$ for all $ij \in \binom{[n]}{2}$. In either case, because $\sum_{i<j}(t_{ij}-\ell)=0$, 
we have that $0 \leq \Sigma^r \leq 2\gamma n^2$.
Thus 
$$
\frac{8r\gamma}{5} = \Sigma^0 - \Sigma^r \geq \Sigma^0 - 2\gamma n^2.
$$

Since, by~\cref{item:iteration3},
$
\frac{r\gamma^2}{25n^2} \leq f(\bm{d},\bm{z}^r) - f(\bm{d},\bm{z}^0) \leq F(\sigma) - (F(\sigma)-\delta) = \delta
$
we have
$$
\Sigma^0 \leq \frac{8r\gamma}{5} + 2\gamma n^2 \leq 
\frac{40n^2\delta}{\gamma} + 2\gamma n^2
\leq 4\gamma n^2.
$$
But, $\Sigma^0 = \sum_{i<j}|t^0_{ij}-\ell|$
so this completes the proof of the lemma.
\end{proof}

The final lemma states that for $\sigma \geq \frac{1}{4}(1+\sqrt{2}) \approx 0.60$ and any $(\bm{d},\bm{z}) \in S(\sigma)$ for which $f(\bm{d},\bm{z})$ is almost maximal, almost all entries of $\bm{z}$ are close to $\sigma^2$.

\begin{lemma}\label{L:BRRROptimisation}
Let $0<\frac{1}{n}\ll \delta < 10^{-6}$. Let $\sigma \in [0,1]$ and $(\bm{d},\bm{z}) \in S(\sigma)$.
Then the following hold.
\begin{enumerate}
    \item Suppose $\sigma \geq \frac{1}{4}(1+\sqrt{2})$. Then $f(\bm{d},\bm{z}) \leq \sigma^3(1-\sigma)+O(\frac{1}{n})$. Moreover, 
$f(\bm{d},\bm{z}) \geq \sigma^3(1-\sigma)- \delta$
only if
$\sum_{i<j}|z_{ij}-\sigma^2| < 3\delta^{1/8} n^2$.
\label{item:BRRROptimasation1}
\item We have $f(\bm{d},\bm{z}) \leq \frac{27}{256}+O(\frac{1}{n})$.
Moreover, $f(\bm{d},\bm{z}) \geq \frac{27}{256} - \delta$
only if $|\sigma-\frac{3}{4}| < 3\delta^{1/4}$ and
$\sum_{i<j}|z_{ij}-\sigma^2| < 3\delta^{1/8} n^2$.
\label{item:BRRROptimasation2}
\end{enumerate}
\end{lemma}

\begin{proof}
For each $1 \leq i < j \leq n$, define $\eps_{ij}$ via $t_{ij} = \ell - 4\eps_{ij}$.
So given $(\bm{d},\bm{z}) \in S(\sigma)$, we have $(\bm{d},\bm{z}) \in S_\eps(\sigma)$ if and only if $\sum_{i<j}|\eps_{ij}| \leq \eps n^2/4$.
We have
$$
f(d_i,d_j,z_{ij})=f\bigl(d_i,d_j,\tfrac{1}{4}(d_i+d_j-\ell)+\eps_{ij}\bigr)=\mfrac{1}{8}\left((d_i+d_j)^2-\ell^2\right) + \eps_{ij}\ell - 2\eps_{ij}^2
$$
and
$$
\mfrac{1}{8}\sum_{i<j}\left((d_i+d_j)^2-\ell^2\right)
= \mfrac{1}{8}\sum_{i<j}\left(d_i^2+2d_id_j+d_j^2\right) - \mfrac{1}{8}\sum_{i<j}\ell^2 \stackrel{(\ref{eq:l})}{=} g_{\sigma}(\tau)n^2 + O(n)
$$
where, as in Lemma~\ref{L:BRRRnum},
$$
g_{\sigma}(\tau) \coloneqq -\tau^2+(\sigma+\tfrac{1}{8})\tau-\tfrac{1}{8}\sigma^2.
$$
Thus
\begin{equation}\label{eq:fdz}
f(\bm{d},\bm{z}) = g_\sigma(\tau) + \mfrac{1}{n^2}\sum_{i<j}\left(\eps_{ij}\ell-2\eps_{ij}^2\right) + O\left(\mfrac1n\right).
\end{equation}
We have
\begin{equation}\label{eq:lbd}
\ell \stackrel{(\ref{eq:l})}{=} 2\sigma - 4\tau + O\left(\mfrac1n\right) \stackrel{\eqref{eq:tausigma}}{\leq} 2\sigma(1 - 2\sigma) + O\left(\mfrac1n\right) \leq \mfrac{1}{4} + O\left(\mfrac1n\right).
\end{equation}
First we prove~\cref{item:BRRROptimasation1}.
So suppose $\sigma \geq \frac{1}{4}(1+\sqrt{2})$.
Lemma~\ref{L:BRRRnum}\cref{item:BRRRnum2} implies that
\begin{equation}\label{eq:gsigmatau}
    g_{\sigma}(\tau) \leq \sigma^3(1-\sigma) - \mfrac{1}{2}\left(\tau-\sigma^2\right)^2
    \quad\text{for all}\quad
    \tau \in [\sigma^2,\infty).
\end{equation}
By Lemma~\ref{lm:iteration}, there is $(\bm{d}',\bm{z}') \in S_0(\sigma)$ (so their corresponding $\eps_{ij}'$ is identically $0$)
with
\[
F(\sigma) = f(\bm{d}',\bm{z}') \stackrel{\cref{eq:fdz}}{=} g_\sigma(\tau) + O\left(\mfrac1n\right) \leq g_{\sigma}(\sigma^2) + O\left(\mfrac1n\right) = \sigma^3(1-\sigma) + O\left(\mfrac1n\right)
\]
by Lemma~\ref{L:BRRRnum}\cref{item:BRRRnum2}.
By taking $\bm{d}''$ with every entry equal to $\sigma$ and $\bm{z}''$ with every entry equal to $\sigma^2+O(\frac1n)$, with $(\bm{d}'',\bm{z}'') \in S(\sigma)$, we in fact have
\begin{equation}\label{eq:Fsigma}
F(\sigma) = \sigma^3(1-\sigma) + O\left(\mfrac1n\right).
\end{equation}
(Observe that this would correspond to an $n$-vertex $\sigma n$-regular graph whose codegrees all differ from $\sigma^2 n$ by a constant but it is not clear whether such a graph exists.)

Next, let $\delta>0$ be given, and let $n$ be sufficiently large. 
We define $\gamma \coloneqq 5\sqrt{2\delta}$.
Suppose $(\bm{d},\bm{z}) \in S(\sigma)$ satisfies $f(\bm{d},\bm{z}) \geq \sigma^3(1-\sigma)-\delta$. Then~(\ref{eq:Fsigma}) implies that $f(\bm{d},\bm{z}) \geq F(\sigma)-2\delta$, so Lemma~\ref{lm:iteration} implies that $(\bm{d},\bm{z}) \in S_{4\gamma}(\sigma)$, so
\begin{equation}\label{eq:firstterm}
\mfrac{1}{n^2}\sum_{i<j}|d_i+d_j - 4z_{ij} -\ell| 
= \mfrac{4}{n^2}\sum_{i<j}|\eps_{ij}| \leq 4\gamma
\end{equation}
and hence
$$
\sigma^3(1-\sigma) - \delta \leq f(\bm{d},\bm{z}) \stackrel{(\ref{eq:fdz}),(\ref{eq:gsigmatau})}{\leq} \sigma^3(1-\sigma) - \mfrac{1}{2}\left(\tau-\sigma^2\right)^2 + O\left(\mfrac1n\right) + \ell\gamma.
$$
Thus, using $\delta < 10^{-6}$,
\begin{equation}\label{eq:tausigma2}
\tau-\sigma^2 \leq \sqrt{2\ell\gamma+2\delta+O\left(\mfrac{1}{n}\right)} \stackrel{(\ref{eq:lbd})}{\leq} \sqrt{\gamma}.
\end{equation}
Combined with~(\ref{eq:tausigma}), this implies that
$$
\mfrac{1}{n^2}\sum_{i<j}(d_i-d_j)^2 \leq \sqrt{\gamma}
$$
and hence, via multiple applications of the triangle inequality, and the Cauchy-Schwarz inequality,
\begin{align*}
    \sum_{i}|d_i-\sigma| = \mfrac{1}{n}\sum_{i}\biggl|\sum_{j}(d_i-d_j)\biggr|
    \leq \mfrac{2}{n}\sum_{i<j}|d_i-d_j|
    \leq \mfrac{2}{n} \biggl(\mbinom{n}{2}\sum_{i<j}(d_i-d_j)^2\biggr)^{1/2} \leq \sqrt{2}\gamma^{1/4}n.
\end{align*}
Therefore
\begin{align*}
    \sum_{i<j}\left|d_i+d_j - \ell-4\sigma^2\right|
    &\stackrel{(\ref{eq:l})}{=} \sum_{i<j}\left|d_i+d_j - 2\sigma + 4\tau -4\sigma^2 + O\left(\mfrac1n\right)\right|\\
    &\stackrel{(\ref{eq:tausigma2})}{\leq} \sum_{i<j}\left|d_i+d_j - 2\sigma + 4\sqrt{\gamma} + O\left(\mfrac1n\right)\right|\\
    &\leq (n-1)\biggl(3\sqrt{\gamma}n+\sum_{i}|d_i-\sigma|\biggr)
    \leq 5\gamma^{1/4}n^2.
\end{align*}
Thus, using the triangle inequality,
\begin{align*}
    4\sum_{i<j}\left|z_{ij}-\sigma^2\right| \leq \sum_{i<j}\left|4z_{ij}+\ell-d_i-d_j\right| + \sum_{i<j}\left|d_i+d_j-\ell-4\sigma^2\right| \stackrel{(\ref{eq:firstterm})}{\leq} 6\gamma^{1/4}n^2 \leq 10\delta^{1/8}n^2,
\end{align*}
as required to prove part~\cref{item:BRRROptimasation1}.

To prove part~\cref{item:BRRROptimasation2}, we again have for any $\sigma \in [0,1]$ and $(\bm{d},\bm{z}) \in S(\sigma)$ that~(\ref{eq:fdz}) holds.
Since the set $[0,1]$ of possible $\sigma$ is closed and bounded, $F \coloneqq \max_{\sigma \in [0,1]}F(\sigma)$ exists and by Lemma~\ref{lm:iteration} we have $f(\bm{d},\bm{z}) = F$ only if $(\bm{d},\bm{z}) \in S_0(\sigma_{\rm max})$ for some $\sigma_{\rm max} \in [0,1]$. 
Thus $F = g_{\sigma_{\rm max}}(\tau^*) + O(\frac1n)$ by
Lemma~\ref{L:BRRRnum}\cref{item:BRRRnum2}. Further, we have either that $\sigma_{\rm max} < \frac{1}{4}(1+\sqrt{2})$ and $F(\sigma_{\rm max})<\frac{23}{256}+O(\frac1n)$ or that $\sigma_{\rm max} \geq \frac{1}{4}(1+\sqrt{2}) \geq \frac{1}{2}$ and $F(\sigma_{\rm max}) = \sigma_{\rm max}^3(1-\sigma_{\rm max}) + O(\frac1n)$.
By Lemma~\ref{L:BRRRnum}\cref{item:BRRRnum1} we have $\sigma_{\rm max}^3(1-\sigma_{\rm max}) \leq \frac{27}{256}$ with equality if and only if $\sigma_{\rm max}=\frac{3}{4}$.
Thus $F = \frac{27}{256}+O(\frac1n)$.

Suppose that
$f(\bm{d},\bm{z}) \geq \frac{27}{256}-\delta$.
As $\delta < 10^{-6}$, we must have $\sigma \geq \frac{1}{4}(1+\sqrt{2}) \geq \frac{1}{2}$ by \cref{L:BRRRnum}\cref{item:BRRRnum2}.
Then Lemma~\ref{L:BRRRnum} implies that
\begin{align*}
\mfrac{27}{256}-\delta &\leq f(\bm{d},\bm{z}) \stackrel{(\ref{eq:fdz})}{\leq} g_{\sigma}(\tau) + \frac{\ell}{n^2}\sum_{i<j}\eps_{ij} + O\left(\mfrac1n\right) 
\stackrel{(\ref{eq:lbd}),(\ref{eq:firstterm})}{\leq} \sigma^3(1-\sigma) + \mfrac{5\sqrt{\delta}}{4} + O\left(\mfrac1n\right)\\
&\leq \mfrac{27}{256} - \mfrac{1}{4}\left(\mfrac{3}{4}-\sigma\right)^2 + \mfrac{3\sqrt{\delta}}{2}.
\end{align*}
So we have $2\sqrt{\delta} > \frac{3}{2}\sqrt{\delta}+\delta \geq \frac{1}{4}(\frac{3}{4}-\sigma)^2$
and hence that $|\sigma-\frac{3}{4}| < 3\delta^{1/4}$. In particular, $\sigma \geq \frac{1}{4}(1+\sqrt{2})$ since $\delta<10^{-6}$.
The conclusion on $\bm{z}$ now follows from~\cref{item:BRRROptimasation1}.
\end{proof}

The proofs of our theorems for $\RRRB$ now follow easily.

\begin{proof}[Proof of Theorems~\ref{th:RRRBbound} and~\ref{th:RRRBprofile}]
Let $n$ be sufficiently large and let $G$ be a red-blue $K_n$. Let $(\bm{d},\bm{z})$ be the degree-codegree vector of $R$, and let $\sigma \coloneqq \frac{1}{n}\sum_{i}d_i$. Then $(\bm{d},\bm{z})$ of $G$ lies in $S(\sigma)$ by Lemma~\ref{lm:sequences}, and $\#(\RRRB,G) = \frac{n^4}{2} f(\bm{d},\bm{z})+O(n^3)$.

Lemma~\ref{L:BRRROptimisation}\cref{item:BRRROptimasation2} implies that 
$\#(\RRRB,G) \leq \frac{27}{512}n^4 + O(n^3)$. Moreover, it implies that
$\#(\RRRB,G) > \frac{27}{512} n^4 - \delta n^4$ only if $|\sigma-\frac{3}{4}| < 3(3\delta)^{1/4}$ and $\sum_{i<j}|z_{ij}-\sigma^2| < 3(3\delta)^{1/8}n^2$.
So $G$ must be $4\delta^{1/8}$-quasirandom
of density in $[\frac{3}{4}-4\delta^{1/4},\frac{3}{4}+4\delta^{1/4}]$.
This proves Theorem~\ref{th:RRRBbound}.

Lemma~\ref{L:BRRROptimisation}\cref{item:BRRROptimasation1} implies that $\#(\RRRB,G)\leq \frac{1}{2}\sigma^3(1-\sigma) n^4+O(n^3)$. Moreover, it implies that $\#(\RRRB,G) > \frac{1}{2}\sigma^3(1-\sigma) n^4 - \delta n^4$ only if $\sum_{i<j}|z_{ij}-\sigma^2| < 3(3\delta)^{1/8}n^2$.
So $G$ must be $4\delta^{1/8}$-quasirandom
of density $\sigma$.
This proves Theorem~\ref{th:RRRBprofile}.
\end{proof}

\begin{proof}[Proof of Theorem~\ref{th:rand}]
Let $Q = 2 \cdot \Kminus +  2\cdot \Kplus + \BRn$.
For any $J,G_J$ where $J$ is an uncoloured graph and $G_J$ is obtained from it by colouring all edges red and all non-edges blue, we have
\begin{equation}\label{eq:RRRBind}
\#(\RRRB,G_J) = \ind(Q,J),
\end{equation}
as in $\Kminus,\Kplus,\BRn$ there are two, two and one ways respectively to make a $4$-cycle with one missing edge.

For all $\sigma \in [0,1]$, we have
$$
{\rm rand}(Q,\sigma) \stackrel{(\ref{eq:rand})}{=} 4!\left(2\cdot\mfrac{1}{4}\cdot\sigma^5(1-\sigma) + 2\cdot\mfrac{1}{2}\cdot \sigma^4(1-\sigma)^2 + 1\cdot\mfrac{1}{2}\cdot \sigma^3(1-\sigma)^3\right)
= 12\sigma^3(1-\sigma).
$$
Theorem~\ref{th:RRRBbound} implies that for any integer $n$ we have
$$
\max(\RRRB,n) = \mfrac{27}{512}n^4 + O(n^3) = \mfrac{1}{2}\left(\mfrac{3}{4}\right)^3\left(1-\left(\mfrac{3}{4}\right)\right)n^4 + O(n^3) = {\rm rand}\left(Q,\mfrac{3}{4}\right)\binom{n}{4} + O(n^3).
$$
By~(\ref{eq:RRRBind}), we have $\ind(Q,n)=\max(\RRRB,n)$
and so $\ind(Q) = {\rm rand}(Q,\frac{3}{4})$.
As observed in Section~\ref{sec:randmax}, $\ind(Q) \geq {\rm rand}(Q,\sigma)$ for all $\sigma \in [0,1]$, so we have $\ind(Q) = \sup_{\sigma \in [0,1]}{\rm rand}(Q,\sigma)$, as required.

Now let $\frac{1}{4}(1+\sqrt{2}) \leq \sigma \leq 1$.
Theorem~\ref{th:RRRBprofile} implies that if $G$ is a red-blue graph on $n$ vertices with $\sigma\binom{n}{2}$ red edges, then
$$
\#(\RRRB,G) \leq \tfrac{1}{2}\sigma^3(1-\sigma)n^4 + O(n^3) = {\rm rand}(Q,\sigma)\binom{n}{4} + O(n^3),
$$
and for all $0 < \delta' < 10^{-6}$, whenever $n$ is a sufficiently large integer and $G$ is a red-blue graph with $\sigma\binom{n}{2}$ red edges with $\#(\RRRB,G) > \frac{1}{2}\sigma^3(1-\sigma)n^4 - \delta' n^4$,
we have that the red graph $R$ of $G$ is $(4\delta'^{1/8})$-quasirandom of density $\sigma$.

The first assertion and~(\ref{eq:RRRBind}) imply that $I(Q,\sigma) = {\rm rand}(Q,\sigma)$.

For the final part of the theorem, let $0 < \delta < 10^{-6}$ and let $n$ be a sufficiently large integer.
Let $J$ be an $n$-vertex graph with $\sigma\binom{n}{2}$ edges with $I(Q,J) > ({\rm rand}(Q,J)-\delta)\binom{n}{4}$.
Then~(\ref{eq:RRRBind}) implies that $\#(\RRRB,G_J) > \frac{1}{2}\sigma^3(1-\sigma)n^4 - \delta n^4/24 + O(n^3)$.
So the second assertion applied with $\delta' = \delta/23$, say, implies that $J$ is $(4(\delta'/23)^{1/8})$-quasirandom and hence $(3\delta^{1/8})$-quasirandom of density $\sigma$.
This completes the proof of the theorem.
\end{proof}

\section{4-cycles plus an edge}\label{sec:K4-e}

In this section we prove some results on certain $H$ which are red-blue $K_{1,1,2}$.
Each such $H$ contains a unique red-blue $4$-cycle $H^-$.
When $H^-$ is an $RBRB$- or $RRBB$-cycle, the results follow easily from our results on $H^-$.
When $H^-$ is monochromatic (and, to avoid trivialities, the remaining edge has the other colour), the problem is equivalent to the inducibility problem for a quantum graph whose components are complete multipartite graphs and we can therefore use a method of Liu, Pikhurko, Sharifzadeh and Staden~\cite{LPSS} which implies a stability version of Theorem~\ref{th:ST}.

\subsection{\texorpdfstring{$RBRB$}{RBRB} plus an edge and \texorpdfstring{$RRBB$}{RRBB} plus an edge}\label{sec:K4-e1}
Applying Lemma~\ref{L:extendingH} with $H^-=\CC$, ${H=\CCext}$, $\#(H^-,H)=1$ and $t=2$ and considering Theorem~\ref{th:RBRBbound}, we see that $\max(\CCext,n)=\lfloor \frac{n^2}{4} \rfloor\lfloor\frac{(n-2)^2}{4}\rfloor$ for all positive integers $n$. Furthermore, the extremal graphs are exactly those for which the red edges induce a balanced bipartite graph. 

Now let $H \in \{\RRBBexta,\RRBBextb\}$. Applying Lemma~\ref{L:extendingH} with $H^-=\RRBB$, $\#(H^-,H)=1$ and $t=1$ and considering Theorem~\ref{th:RRBBbound}, we see that $\max(H,n)=\max(\RRBB,n)$ for all sufficiently large $n$. Furthermore, the extremal graphs are exactly those given in Theorem~\ref{th:RRBBbound} for which the red edges form the complete bipartite graph if $H=\RRBBexta$ and for which the blue edges form the complete bipartite graph if $H=\RRBBextb$.

\subsection{\texorpdfstring{$RRRR$}{RRRR} plus a blue edge}\label{sec:K4-e2}

\begin{theo}\label{theo:K4-e2}
    For all sufficiently large $n$, a red-blue $K_n$ has at most
    $$
    \sum_{i \in [3]}\mbinom{n_i}{2}\mbinom{n-n_i}{2} = \mfrac{1}{27}n^4+O(n^3)
    $$ copies of $\CCextt$ where $(n_1,n_2,n_3)$ is the unique vector of positive integers with $n_1+n_2+n_3=n$ and $n_1 \geq n_2 \geq n_3 \geq n_1-1$. Moreover, this is achieved if and only if the red edges induce a tripartite Tur\'an graph $T_3(n)$.
\end{theo}

Recall from \cref{sec:ind} that determining $\max(\CCextt,n)$ is equivalent to a classical (uncoloured) inducibility problem. More precisely, setting
$$
Q \coloneqq 2\cdot\BBn+\Kminus,
$$
we have $\#(\CCextt,G)=\ind(Q,R)$ for any red-blue graph $G$ and its red subgraph $R$.
We will solve this inducibility problem using a method of \cite{LPSS} as follows. The main result in that paper (\cite[Theorem 1.1]{LPSS}) can be applied to show that if a quantum graph $Q'$ whose constituents are small complete multipartite graphs which satisfy some specific properties, then the extremal graphs for $Q'$ are complete multipartite whose part sizes are asymptotically determined. 
In other words, to determine $\ind(Q',n)=\max_{|V(J)|=n}\ind(Q',J)$ for large $n$, we only need to maximise the number of copies of $Q'$ across $n$-vertex complete multipartite graphs $J$.

In our case, observe that if $J\cong K_{n_1,\dots,n_t}$ where $n_1+\ldots+n_t=n$ and $n_i=x_i n + O(1)$ for each $i \in [t]$,
then $\ind(Q,J)$ is the number of ordered pairs $(xy,uv)$ where $x,y,u,v\in V(J)$ are distinct, $xy \notin E(J)$ and $\{u,v\} \subseteq N_{J}(x) \cap N_{J}(y)$.
That is, there is $i$ such that $x,y$ both lie in the $i$-th part, while neither $u$ nor $v$ lie in the $i$-th part. Thus
\begin{align*}
\lim_{n \to \infty}\frac{\ind(Q,J)}{n^4} 
&= \lim_{n \to \infty}\frac{\sum_{i \geq 1}\binom{n_i}{2}\binom{n-n_i}{2}}{n^4}
= \mfrac{1}{4}\sum_{i \geq 1}x_i^2(1-x_i)^2.
\end{align*}
We will therefore consider the following optimisation problem: among all $\bm{x}=(x_1,x_2,\ldots)$ with $x_1 \geq x_2 \geq \ldots \geq 0$ and $\sum_{i \geq 1}x_i \leq 1$, maximise $\lambda_Q(\bm{x}) \coloneqq \sum_{i \geq 1}x_i^2(1-x_i)^2$. Finding $\bm{x}$ maximising $\lambda_Q(\bm{x})$ gives us part ratios which maximise $\ind(Q, J)$ across $n$-vertex complete multipartite graphs $J$.
According to \cite[Theorem 1.1]{LPSS}, if the set of such maximisers $\bm{x}$ satisfies some additional `strictness' properties, then the extremal graphs for $\ind(Q,n)$ are all complete multipartite whose part ratios are asymptotically equal to one of these $\bm{x}$. 
Below is a special case of those `strictness' properties. The original definition differs in that \cite[Theorem 1.1]{LPSS} allows for other structures of maximisers. But for simplicity, we state here a version which only includes the structure of the (unique) maximiser for our quantum graph $Q$ of interest.

A complete multipartite graph $G$ on $n$ vertices is \emph{$Q$-strict} if there exists $c>0$ such that the following hold.
    \begin{enumerate}[label=(S\arabic*)]
        \item For all distinct $x,y \in V$ we have $\ind(Q,G) - \ind(Q,G \oplus xy) \geq c n^{|V(Q)|-2}$, where $G \oplus xy$ is the graph with vertex set $V(G)$ and edge set $E(G) \mathrel\triangle \{xy\}$. \label{def:strict1}
        \item Given $A \subseteq [t]$, let $G_A$ be the graph obtained from $G$ by adding a new vertex $v$ which is adjacent to each existing vertex $x$ if and only if $x$ is in the $i$-th part of $G$ for some $i \in A$. Then $\ind(Q,G_{A^*}) - \ind(Q,G_A)\geq cn^{|V(Q)|-1}$ whenever $|A^*|=t-1\neq |A|$.\label{def:strict2}
    \end{enumerate}

(Condition \cref{def:strict1} corresponds to \cite[Theorem 1.1(i)]{LPSS}, while \cref{def:strict2} is a version of \cite[Theorem 1.1(ii)]{LPSS} adjusted to our setting.)
We are now ready to state a special case of \cite[Theorem 1.1]{LPSS} for the above quantum graph $Q$. The fact that $\lambda_Q$ satisfies the desired `symmetrisability' condition to apply \cite[Theorem 1.1]{LPSS} follows from \cite[Lemma 1.2]{LPSS}. The conclusion of \cite[Theorem 1.1]{LPSS} states that $\lambda_Q$ is `perfectly stable', which roughly speaking means that the extremal graphs are close to a complete multipartite graph in a very strong sense that in fact implies an exact result.

\begin{theo}\label{theo:LPSS}
    Let $0<\frac{1}{n}\leq \frac{1}{n_0}\ll 1$.
    Suppose that the following properties hold.
    \begin{enumerate}
    \item There is a unique $\bm{x}=(x_1,x_2,\ldots)$ with $x_1 \geq x_2 \geq \ldots \geq 0$ and $\sum_{i \geq 1}x_i \leq 1$ which maximises $\lambda_Q$. Moreover, there is $t \in \mb{N}$ such that $x_i=0$ for all $i>t$. 
    \label{theo:LPSS-lambda}
    \item For all integers $m\geq n_0$, any $K_{m_1,\ldots,m_t}$ where $m_i=x_i m + O(1)$ for all $i \in [t]$ is $Q$-strict.\label{theo:LPSS-strict}
    \end{enumerate}
    Then every extremal graph for $\ind(Q,n)$ is complete $t$-partite with $i$-th part of size $x_i n + o(n)$.
\end{theo}

\begin{proof}[Proof of Theorem~\ref{theo:K4-e2}]
    As mentioned above, it is enough to consider the inducibility problem $\ind(Q,n)$, which we will solve using \cref{theo:LPSS}.
    To verify \cref{theo:LPSS}\cref{theo:LPSS-lambda}, note that
    $$
    \sum_{i \geq 1}x_i^2(1-x_i)^2 \leq \sup_{i \geq 1}x_i(1-x_i)^2 \cdot \sum_{j \geq 1}x_j \leq \sup_{i \geq 1}x_i(1-x_i)^2 
    \leq \mfrac{1}{3}\left(1-\mfrac{1}{3}\right)^2 = \mfrac{4}{27}.
    $$
    Note that the supremum is attained for some $i$.
    We have equality in
    \begin{itemize}
        \item the first inequality if and only if every non-zero $x_i(1-x_i)^2$ is equal;
        \item the second inequality if and only if $\sum_{i \geq 1}x_i=1$;
        \item the third inequality if and only if $\sup_{i \geq 1}x_i=\frac{1}{3}$.
    \end{itemize}
    Thus we have equality everywhere if and only if each $x_i$ equals $0$ or $\frac{1}{3}$ and there are exactly three of the latter type.
    That is, $\bm{x}=(\frac{1}{3},\frac{1}{3},\frac{1}{3},0,\ldots)$ is the unique vector which maximises $\lambda_Q(\bm{x})$, as desired.

    For any graphs $H,G$ and distinct $y_1,\ldots,y_s \in V(G)$, let $I_{y_1,\ldots,y_s}(H,G)$ be the number of induced copies of $H$ in $G$ containing all of $y_1,\ldots,y_s$.
    
    Now, for \cref{theo:LPSS}\cref{theo:LPSS-strict}, fix an integer $m\geq n_0$ and let $m_i=\frac{m}{3}+O(1)$ for each $i\in [3]$. 
    To verify~\cref{def:strict1}, let $G$ be a complete tripartite graph with parts $X_1,X_2,X_3$ where $|X_i|=m_i$ for all $i \in [3]$. Let $x,y\in V(G)$ be distinct.
    Then
    \[\ind(Q,G)-\ind(Q,G\oplus xy) = 2\cdot I_{x,y}(\BBn,G) + I_{x,y}(\Kminus,G) - 2\cdot I_{x,y}(\BBn,G\oplus xy) - I_{x,y}(\Kminus,G\oplus xy).\]
    Suppose first that $x,y$ are in different parts, say $x \in X_1$ and $y \in X_2$ (the other cases are identical). So $xy \in E(G)$. Then 
    $xy$ lies in a copy of $\BBn$ in $G$ whenever there is one more vertex in each of $X_1,X_2$, so $I_{x,y}(\BBn,G) = (m_1-1)(m_2-1)$, and 
    $xy$ lies in a copy of $\Kminus$ in $G$ whenever the copy has at least one vertex in $X_3$, so $I_{x,y}(\Kminus,G) = \binom{m_3}{2}+m_3(m_1+m_2-2)$.
    On the other hand, $I_{x,y}(\BBn,G-xy)=\binom{m_3}{2}$ and $I_{x,y}(\Kminus,G-xy)=0$.
    Therefore, since every $m_i=\frac{m}{3}+O(1)$, we have 
    \[
    \ind(Q,G)-\ind(Q,G-xy)
    =2(m_1-1)(m_2-1)+\mbinom{m_3}{2}+m_3(m_1+m_2-2)-2\mbinom{m_3}{2}
    =\mfrac{7m^2}{18}+O(m).
    \]
    Suppose now that $x,y$ are distinct vertices in the same part, say $x,y \in X_1$ (the other cases are identical). So $xy \notin E(G)$. Then $I_{x,y}(\BBn,G) = \binom{m_2}{2}+\binom{m_3}{2}$ and $I_{x,y}(\Kminus,G)=m_2 m_3$.
    On the other hand, $I_{x,y}(\BBn,G+xy)=0$ and $I_{x,y}(\Kminus,G+xy)=\binom{m_2}{2}+\binom{m_3}{2}$.
    Therefore, we have 
    \[
    \ind(Q,G)-\ind(Q,G+xy)
    =2\left(\mbinom{m_2}{2}+\mbinom{m_3}{2}\right)+m_2m_3 - \left(\mbinom{m_2}{2}+\mbinom{m_3}{2}\right)
    =\mfrac{2m^2}{9}+O(m).
    \]
        For \cref{def:strict2}, let $A\subseteq [3]$ and note that we can count the number of induced copies of $\BBn$ and $\Kminus$ in $G_A$ which cover the new vertex $v_A$ as follows.
    Any such copy of $\BBn$ corresponds to an ordered pair $(\{x,y\},z)$ where $x,y \in X_i$ are distinct and $i\in A$, while $z \in X_j$ for some $j\in [3]\setminus A$.
    Moreover, all such copies of $\Kminus$ correspond to an ordered pair $(\{x,y\},z)$ where either
    \begin{itemize}
        \item $x,y \in X_i$ are distinct and $i\in A$, while $z \in X_j$ for some $j\in A\setminus \{i\}$; or
        \item $x \in X_i$ and $y \in X_j$ for some distinct $i,j\in A$, while $z \in X_k$ for some $k\in [3]\setminus A$.
    \end{itemize}
    Thus, we have
    \begin{align*}
        \ind(Q,G_A)-\ind(Q,G) =2\cdot I_{v_A}(\BBn,G_A) + I_{v_A}(\Kminus,G_A) =
        \begin{cases}
            0 & \text{if } |A|=0\\
            \left(\frac{m}{3}\right)^2\cdot\left(\frac{2m}{3}\right)+O(m^2) & \text{if } |A|=1\\
            2\left(\frac{m}{3}\right)^3+\left(\frac{m}{3}\right)^3+ \left(\frac{m}{3}\right)^3 +O(m^2) & \text{if } |A|=2\\
            3\left(\frac{m}{3}\right)^3+O(m^2) & \text{if } |A|=3\\
        \end{cases}
    \end{align*}
    and so \cref{def:strict2} holds.

    Therefore, \cref{theo:LPSS} implies that whenever $n$ is sufficiently large, every $n$-vertex graph $J$ with $\ind(Q,J)=\ind(Q,n)$ is a complete tripartite graph with parts of size $\frac{n}{3} + o(n)$.
    We complete the proof by calculating the optimal part sizes.

    Let $J\cong K_{n_1,n_2,n_3}$ with $n_1 \geq n_2 \geq n_3$, $n_1+n_2+n_3=n$, $|n_i-\frac{n}{3}|=o(n)$ for all $i \in [3]$, and such that $\ind(Q,J)=\ind(Q,n)$.
    Let $J'\cong K_{n_1-1,n_2,n_3+1}$. Define $f$ to be the real function defined by $f(a)\coloneqq \binom{a}{2}\binom{n-a}{2}$ and note that $\ind(Q,K_{p,q,r})=f(p)+f(q)+f(r)$. Then
    \begin{align*}
    \ind(Q,J')-\ind(Q,J) &= f(n_1-1)+f(n_3+1)-\bigl(f(n_1)+f(n_3)\bigr)\\
    &= \mfrac{n_1-n_3-1}{2}\left(3n(n_1+n_3)-n(n+1) - 2n_1^2 - 2n_1n_3 - 2n_3^2 + n_1-n_3\right)\\
    &= \mfrac{n_1-n_3-1}{2}\left(\mfrac{n^2}{3}+o\left(n^2\right)\right).
    \end{align*}
    If $n_1>n_3+1$, this is positive and so $\ind(Q,J) \neq \ind(Q,n)$, a contraction. Thus $n_1 \leq n_3+1$ as required.
\end{proof}

\section{Concluding remarks}

In this paper, we studied the semi-inducibility problem, Problem~\ref{prob:si}, which asks, given a graph $H$ whose edges are coloured red and blue, to determine $\max(H,n)$: the maximum number of copies of $H$ in a red-blue $K_n$. The case where $H$ is complete is equivalent to the well-known inducibility problem, Problem~\ref{prob:ind}, which has been well-studied but is still very much open. 
We proved, among others, results for alternating paths, alternating cycles of length divisible by four, and every coloured $4$-cycle (see Table~\ref{table}).
Of course, there are many $H$ for which the problem remains open.
Two of the most interesting open cases are the following.

\begin{prob}
    Solve Problem~\ref{prob:si} for
the alternating $6$-cycle (or, more generally, the alternating $(4t+2)$-cycle for $t \geq 1$).
\end{prob}

Recall that we can determine $\max(Q)$ for the quantum (red-blue) graph $Q$ which is the sum of the red-blue-red path and the blue-red-blue path (a special case of~\cref{th:CSpathBound}). 

\begin{prob}
    Solve Problem~\ref{prob:si} for
    the red-blue-red path with three edges (or, more generally, a given path with $s$ red edges and $t$ blue edges).
\end{prob}

\subsection{The feasible region of red-blue graphs}\label{sec:concfeas}

Recall that the semi-inducibility problem (Problem~\ref{prob:si}) is a special case of the inducibility problem for quantum graphs (Problem~\ref{prob:indgen}).
Liu, Mubayi and Reiher~\cite{LMR} initiated the systematic study of the \emph{feasible region} $\Omega_{\rm ind}(Q)$ of a quantum graph $Q$, 
which is the set of points $(\sigma,\beta) \in [0,1]^2$ for which there exists a sequence $J_1,J_2,\ldots$ of graphs with $a_n := |V(J_n)|\to\infty$ and $e(J_n)/\binom{a_n}{2} \to \sigma$
and $\ind(Q,J_n)/\binom{a_n}{|V(Q)|} \to \beta$.
Informally, this is the collection of points $(\sigma,\beta) \in [0,1]^2$ for which there exists an \emph{uncoloured} graph of density arbitrarily close to $\sigma$ such that the density of induced `copies' of $Q$ (see~(\ref{eq:quantum})) is arbitrarily close to $\beta$ (see~\cite[Definition~1.1]{LMR}).
They proved some general results on the shape of $\Omega_{\rm ind}(Q)$ which, they show, is determined by its boundary.
Its upper boundary is a continuous and almost everywhere differentiable function $\ind(Q,\sigma)$ of $\sigma$ (defined in Section~\ref{sec:randmax}), as is its lower boundary.

In the context of the semi-inducibility problem, a special case of determining $\ind(Q,\sigma)$ is, given a red-blue graph $H$, to determine the maximum number of copies of $H$ among all red-blue $K_n$ with red density $\sigma$.

Our result Theorem~\ref{th:RRRBprofile} is a result of this type: it determines the maximum number of copies of $\RRRB$ among all large red-blue $K_n$ with red density $\sigma$ where $\sigma \geq \frac{1}{4}(1+\sqrt{2}) \approx 0.60$. We showed that for each such $\sigma$, the extremal graphs are quasirandom.
An obvious open question is to extend the range of $\sigma$. It is not clear whether we would expect quasirandom graphs to be extremal for smaller $\sigma$.

A solution to the following problem (suggested by Dhruv Mubayi) would generalise Thereom~\ref{th:RBRBbound}.

\begin{prob}
Let $n$ be an integer and let $\sigma \in [0,1]$.
    In a red-blue $K_n$ with red edge density $\sigma+o(1)$, what is the maximum number of alternating $4$-cycles $\CC$?
\end{prob}

In the language of~\cite{LMR}, this is essentially equivalent to determining $\ind(2\cdot\BBn+2\cdot\RRn+\BRn,\sigma)$.

\subsection{Inducibility and random maximisers}\label{sec:concind}

We recall that Jain, Michelen and Wei~\cite{Jain2023binomial} showed that for all single graphs $F$ on at least three vertices and all $\sigma \in (0,1)$, the upper bound $\ind(F,\sigma)$ is strictly below the function ${\rm rand}(F,\sigma)$ achieved asymptotically by a quasirandom graph (Theorem~\ref{th:JMW}).
On the other hand, our Theorem~\ref{th:rand} (a restatement of Theorem~\ref{th:RRRBprofile}) gives a $4$-vertex quantum graph $Q$ with positive coefficients for which the upper boundary $\ind(Q,\sigma)$ is equal to ${\rm rand}(Q,\sigma)$ for an interval of values $\sigma$, and crucially the only sequences of graphs corresponding to points near this boundary are quasirandom ones.
It would be interesting to investigate the role of (quasi)random graphs as maximisers in the inducibility problem further.

Theorem~\ref{th:JMW} (from~\cite{Jain2023binomial}) does not preclude a random bipartite graph, say, from being extremal for some single graph $F$.
We recall that flag algebra calculations of Even-Zohar and Linial~\cite{EvenZoharLinial15} suggest that such graphs could be extremal, e.g.~for $F$ a $4$-cycle with a pendant edge,
and that the results of~\cite{BurkeLidickyPfenderPhillips} show that, in the directed setting, there are extremal graphs with quasirandom components.

\subsection{Inducibility of quantum graphs with two constituents}
Is there a connection between the extremal graphs for the inducibility problem $\ind(\lambda\cdot F_1+(1-\lambda)\cdot F_2)$ of the weighted sum of two graphs $F_1,F_2$ in terms of the individual extremal graphs for $\ind(F_1)$ and $\ind(F_2)$? A special case 
comes from the semi-inducibility problem (suitably normalised, see Fact~\ref{fact:basic}\ref{itm:quantum}).
If $H$ is a red-blue copy of $K_h-e$, then $\max(H,n) = (c_1+c_2)\ind(\frac{c_1}{c_1+c_2} \cdot F_1+\frac{c_2}{c_1+c_2} \cdot F_2,n)$ where $F_1$ is the red graph of $H$, $F_2$ is the graph obtained from $F_1$ by adding $e$, and $c_1$ (respectively $c_2$) is the number of copies of $H$ in the red-blue graph obtained from $H$ by adding $e$ and colouring it blue (respectively red). In particular, do the extremal graphs for $\ind(\lambda\cdot F_1+(1-\lambda)\cdot F_2)$ and $\ind(F_1)$ together imply anything about those for $\ind(F_2)$?
Our results provide some positive pieces of data regarding the first question, and some negative data regarding the second.
\begin{itemize}
    \item 
    The complete balanced bipartite graph $T_2(n)$, which is induced $\BRn$-free, is the common (unique) extremal graph of $\frac{2}{3}\cdot\BBn+\frac{1}{3}\cdot\BRn$ and of $\BBn$.
    \item 
    The extremal graph for $\BBn$ is $T_2(n)$, for $\frac{2}{3}\cdot\BBn+\frac{1}{3}\cdot\Kminus$ is $T_3(n)$, and for $\Kminus$ is $T_5(n)$, all unique.
\end{itemize}

 To see the first bullet point, Pippenger and Golumbic~\cite{PippengerGolumbic75} showed that $T_2(n)$ is an extremal graph for $\BBn$ and Bollob\'as, Egawa, Harris and Jin~\cite{BollobasEgawaHarrisJin95} showed uniqueness, while $\#(\CCext,n)=\ind(2\cdot\BBn+\BRn,n)$ and the assertion about the unique extremal graph here follows from Section~\ref{sec:K4-e1}.

  To see the second bullet point, Hirst~\cite{Hirst14} determined $\ind(\Kminus)$, showing that $T_5(n)$ is asymptotically extremal, while Pikhurko, Slia\v{c}an and Tyros~\cite{PikhurkoSliacanTyros19} showed it is uniquely extremal. Further, $\#(\CCextt,n)=\ind(2\cdot\BBn+\Kminus)$ and the assertion about the unique extremal graph here follows from Theorem~\ref{theo:K4-e2}.
  
In both of these examples, very informally, the extremal graph for $\lambda\cdot F_1+(1-\lambda)\cdot F_2$ is either equal to that of $F_1$ or $F_2$ or is `in between' the two.
In the first case, it equals that of $F_1$ and contains no copies of $F_2$ at all, while in the second the extremal graph of $\lambda\cdot F_1+(1-\lambda)\cdot F_2$ contains copies of both $F_1$ and $F_2$. The first case tells us nothing about the extremal graph(s) of $\BRn$; we recall that determining this is a major open problem.

\section{Acknowledgements}
The authors thank the mathematical research institute MATRIX in Australia where part of this research was performed. 
A preliminary short version of this paper that contains Theorem~\ref{th:RBRBbound} only will appear in the MATRIX Annals book series.
Andr\'e K\"undgen thanks the School of Mathematics at Monash University for hosting him during his sabbatical.
We are grateful to Fan Wei for drawing our attention to her paper with Jain and Michelen~\cite{Jain2023binomial}.

\bibliography{symm}
\bibliographystyle{abbrv}

\end{document}